\documentclass[12pt,reqno,twoside]{amsart}
\usepackage{amsmath}
\usepackage{amssymb}

\usepackage{color}

\newcommand{\vre}{\varepsilon}

\newcommand\RR{\mathbb R}
\newcommand\ZZ{\mathbb Z}

\newcommand\supp{\operatorname{supp}}

\newcommand\vol{\operatorname{vol}}

\newcommand\norm[1]{\left\|#1\right\|}

\newcommand\abs[1]{\left|#1\right|}

\newcommand\set[1]{\left\{{#1}\right\}}

\newtheorem{thm}{Theorem}[section]
\newtheorem{lem}[thm]{Lemma}
\newtheorem{prop}[thm]{Proposition}

\newtheorem{defi}[thm]{Definition}
\newtheorem{rem}[thm]{Remark}

\numberwithin{equation}{section}

\title{Ergodic theory and the duality principle on homogeneous spaces}

\date{\today}

\author{Alexander Gorodnik}
\address{School of Mathematics \\ University of Bristol \\ Bristol, U.K.}
\email{a.gorodnik@bristol.ac.uk}
\thanks{The first author was supported in part by EPSRC, ERC, and RCUK}

\author{Amos Nevo}
\address{Department of Mathematics, Technion, Israel}
\email{anevo@tx.technion.ac.il}
\thanks{The second author was supported by ISF grant}

\begin{document}
\begin{abstract}
We prove mean and pointwise ergodic theorems for the action of a discrete lattice subgroup 
 in a connected algebraic Lie group 
  on infinite volume
homogeneous algebraic varieties. 
Under suitable necessary conditions, our results are quantitative, namely we establish 
rates of convergence in the mean and pointwise ergodic theorems, which can be estimated explicitly.  
Our results give a precise and in most cases optimal quantitative form to the duality principle governing dynamics on homogeneous spaces. We illustrate their scope in a variety of equidistribution problems. 
\end{abstract}

\maketitle
{\small \tableofcontents}

\section{Introduction}\label{sec:intro}
\subsection{Ergodic theory and the duality principle on homogeneous spaces}

The classical framework of ergodic theory usually includes a 
compact space $X$ equipped with finite measure and an action of a countable group $\Gamma$
which preserves this measure. In order to study the distribution of the orbits
$x \Gamma$ in $X$, one chooses an increasing sequence $\{\Gamma_t\}_{t\ge t_0}$ of finite subsets
of $\Gamma$ and considers the averaging operators
$$
\pi_{X}(\lambda_t)\phi(x)=\frac{1}{|\Gamma_t|}\sum_{\gamma\in \Gamma_t} \phi(x\gamma),
$$
defined for functions $\phi$ on $X$. One of the fundamental problems in ergodic theory
is to understand the asymptotic behaviour of $\pi_{X}(\lambda_t)\phi$ as $t\to \infty$.
This question has been studied extensively when $\Gamma$ is an amenable group and the averages are supported on F\o lner sets (see \cite{n1} for a survey, and \cite{AAB} for a detailed recent discussion). 
Subsequently pointwise ergodic theorems were established for some classes of nonamenable groups, including 
 lattice subgroups in semisimple algebraic groups, with the averages supported on norm balls (see \cite{gn} for a comprehensive discussion).

The situation when $X$ is a non-compact locally compact space equipped with an infinite Radon measure
is also of great interest, but it involves new highly nontrivial challenges.
Indeed, in this case, the averages $\pi_{X}(\lambda_t)\phi$
considered above typically converge to zero. In order
to obtain significant information about the distribution of orbits, it is natural to introduce the (normalized) orbit-sampling operators
$$
\pi_{X}(\lambda_t)\phi(x)=\frac{1}{V(t)}\sum_{\gamma\in \Gamma_t} \phi(x\gamma),
$$
where $V(t)$ is a suitable normalization, which one would like to choose
so as to guarantee that the limit as $t\to\infty$ exists and is nontrivial. 
While it is well-known \cite[Th.~2.4.2]{Aa})) that 
for an action of a single transformation no such normalization exists,
we shall achieve this objective for an extensive family of actions of lattices on infinite-volume homogeneous spaces, and proceed to develop 
a systematic and quantitative ergodic theory for the operators $\pi_X(\lambda_t)$. 
Our general results describe, in particular, the distribution of lattice orbits
on the de-Sitter space, answering questions raised by Arnol'd \cite[1996-15, 2002-16]{A}.

The methods that we develop in order to obtain this goal are of very general nature
and amount to establishing a quantitative form of an abstract duality principle for homogeneous spaces.
Namely, if $X\simeq H\backslash G$ is a homogeneous space of a locally compact
second countable group $G$ and $\Gamma$ is a discrete lattice subgroup in $G$, we reduce 
the ergodic-theoretic properties of the $\Gamma$-orbits on $X$ to the ergodic-theoretic properties of the $H$-orbits in the (dual) action
of $H$ on $G/\Gamma$. 

We will develop below an axiomatic framework in which the quantitative duality 
principle will be established in full generality (Sections 2-7). 
    Our principal motivation for taking an abstract approach is the fact that the present  paper does not exhaust the range of validity and the  diverse applications of the quantitative ergodic theorems that we develop. Most importantly, essentially all of our arguments carry over with minor modifications to the case of general $S$-algebraic groups over fields of characteristic zero.  We also note that many of our arguments carry over to homogeneous spaces of adele groups, as well as to $S$-algebraic groups over fields of positive characteristic.  To illustrate their utility, we refer to \cite{ggn} for an application of quantitative ergodic duality arguments to Diophantine approximation on homogeneous algebraic varieties in the $S$-algebraic set-up, answering some long-standing questions raised originally by S. Lang \cite{La}. We plan to return to the quantitative duality principle and its applications in the context of homogenous spaces of $S$-algebraic groups in the future, but in the interest of brevity will confine ourselves in the present paper to connected Lie groups and their homogeneous spaces.

Let us now note that the subject of ergodic theory of non-amenable groups acting on infinite-measure spaces is full of surprises and exhibits several remarkable features which do not arise in the classical case of amenable groups acting on probability spaces.  Let us mention the following ones. 

\begin{enumerate}
\item As already noted, the very existence of a normalization $V(t)$ 
for the orbit-sampling operators $\pi_X(\lambda_t)$ is impossible in the case of $\ZZ$-actions ; but we will also encounter the remarkable phenomenon that the growth 
of the sampling sets $\Gamma_t$ may be exponential in $t$, while the normalization $V(t)$ is polynomial in $t$. Thus, for a point $x$ in a given bounded set  $D\subset X=H\setminus G$, the set of return points $x\cdot \Gamma_t  \cap D $  is {\it logarithmic} in the size of  $x\cdot \Gamma_t$, and yet the set of return points is almost surely equidistributed in $D$. In general, the set of return points will be exponentially small compared to the set of orbit points. 
\item The ergodic theorems we prove assert that under suitable conditions the averages $\pi_X(\lambda_t)$ converge in a suitable sense  to a limiting distribution :
$$\lim_{t\to \infty} \pi_{X}(\lambda_t)\phi(x) =\int_X \phi\,d\nu_x\,. $$
However,  the limiting distribution may fail to be invariant under the $\Gamma$-action, and may depend non-trivially on the initial point $x$, exhibiting distinctly non-amenable phenomena. 
\item When the dual action of $H$ on $G/\Gamma$ has a suitable spectral gap, we will establish an effective rate of convergence of the orbit-sampling operators to the limiting distribution, in the mean and sometimes pointwise. As a consequence, we will obtain quantitative ergodic theorems for actions on homogeneous spaces. Again these are new and distinctly non-amenable phenomena. 
\end{enumerate}
We remark that the ergodic theorems we establish have another significant set of applications  which involves ratio ergodic theorems on homogeneous spaces, a subject raised originally by Kazhdan \cite{K} for the Euclidean group.
We will state some ratio ergodic theorems and comment further on this subject below.


%
%
%

Before turning to the exact statements of our main results in the next section, let us make one further
comment on their scope. As explained in \cite{gn}, the ergodic theory of non-amenable groups has to
contend with the absence of asymptotic invariance and transference arguments that play a pivotal role in
amenable ergodic theory. An indispensable tool 
to compensate for this absence is the existence of detailed quantitative volume estimates for the
sampling sets involved in our analysis. These estimates include 
quantitative volume asymptotics, as well as quantitative stability and regularity properties, which will
be explained in detail and exploited in our analysis below. 
The verification of these volume estimates is an intricate and challenging task which played a central
role in \cite{gn}. Here we elaborate on it further to the extent required to 
establish our  principal objective, which is the systematic development of ergodic theory for lattice
subgroups of algebraic groups acting on homogeneous  algebraic varieties.
In principle, our results hold whenever the required volume estimates are valid, and there are grounds to expect that such volume estimates may be satisfied beyond the case of algebraic groups. However, the volume estimates are definitely not valid for completely general Lie groups and their homogeneous manifolds, as demonstrated in  \cite[12.2]{GW}.
For this reason, we will restrict the discussion to algebraic groups acting on algebraic homogeneous spaces, and with the
sampling sets being defined by a homogeneous polynomial, or in some cases, a norm. 

\subsection{Statement of the main results}
Let us start be introducing notation that will be in force throughout the paper. 
Let $G\subset\hbox{SL}_d(\mathbb{R})$ be a connected closed subgroup.
Let $H\subset G$ be a closed subgroup, let $X=H\setminus G$ be the corresponding homogeneous space, and let $\Gamma$ be a discrete lattice in $G$.

For  a proper function $P$ positive except at $0$,
we consider the family of finite sets $\Gamma_t=\{\gamma\in \Gamma:\, \log P(\gamma)\le t\}$. 
Our main object of study will be the associated orbit-sampling operators $\sum_{\gamma\in \Gamma_t} \phi(x\gamma)$ with $x\in X$ and $\phi:X\to \mathbb{R}$,
whose properties reflect the distribution of the orbits of $\Gamma$ in $X$.  We will use normalization functions $V(t)$ of two kinds for the orbit sampling operators. The first is defined intrinsically:
$$\pi_X(\lambda_t)\phi(x)=\frac{1}{\vol (H_t)}\sum_{\gamma\in \Gamma_t} \phi(x\gamma),$$
where $H_t=\{h\in H:\, \log P(h)\le t\}$.
and the second reflects our knowledge of the volume asymptotics of $H_t$ (when applicable) :
$$\pi_X{(\tilde{\lambda}_t)\phi(x)=\frac{1}{e^{at}t^b}}\sum_{\gamma\in \Gamma_t} \phi(x\gamma)$$
with $a\ge 0$ and $b\ge 0$. In the context of algebraic groups with the sets $H_t$ defined by  the homogeneous polynomial $P$, the 
volume of $H_t$ does indeed have $c\,e^{at}t^b$ as its main term, so that the two operators are comparable.

We fix a smooth measure $\xi$ on $X$ with strictly positive density. Our discussion below 
will focus on a fixed (but arbitrary) compact set $D\subset X$. We assume that $D=\overline{\text{ Int } (D)}$ and that the boundary of $\text{ Int }(D)$ has zero measure, and call $D$ a compact domain in this case.  we denote by $L^p(D)$ the space of $L^p$-integrable functions $\phi$
with $\supp(\phi)\subset D$ equipped with the norm
$$
\|\phi\|_{L^p(D)}=\left(\int_D|\phi|^p \, d\xi\right)^{1/p}.
$$  
Since the space $L^p(D)$ does not depend on the measure $\xi$ and different $\xi$'s
lead to equivalent norms, we suppress $\xi$ from the notation.   
We also denote by $L^p_l(D)$ the space of Sobolev function with support in $D$
(see Section \ref{sec:notation}) and by $L^p_l(D)^+$ the subset of $L^p_l(D)$
consisting of nonnegative functions.

Our main results are formulated in the three theorems stated below. We consider three possibilities 
for the structure of the stability group $H$ and the volume growth of the sets $H_t$, as measured by our choice of Haar measure on $H$. 
As noted above, we will require stringent regularity conditions on $H_t$ and its volume, and for this reason will assume from now on that $G$ and $H$ are almost algebraic groups of $\hbox{SL}_d(\RR)$ (namely they are of finite index in their Zariski closure over $\RR$), and that $P$ is a homogeneous polynomial on the linear space $\hbox{Mat}_d(\mathbb{R})$. 

The first main result, 
Theorem \ref{th:main1},  deals with case of where the volume growth of $H_t$ in polynomial in $t$, so
that the normalization factor $V(t)$ of the sampling operators is polynomial as well. As we shall see
below in Lemma \ref{l:growth}, this forces $H$ to be isomorphic to the almost direct product of an
$\RR$-diagonalizable torus and a compact group. Theorem  \ref{th:main1} establish a mean and pointwise
ergodic theorem for the normalized sampling operators $\pi_X(\tilde{\lambda}_t)$  in every $L^p$-space, $
1 \le p < \infty$. Under the further assumptions that $G$ is semisimple, the lattice $\Gamma$ is
irreducible, and the action of $G$ on $L_0^2(G/\Gamma)$ has strong spectral gap,
it establishes a pointwise ergodic theorem with a polynomial rate of convergence for Sobolev functions.
We recall that a unitary representation of a connected semisimple group is said to have a {\it strong spectral gap} if 
its restriction to every simple factor  $L$ 
is isolated from the trivial representation of $L$.

Our second main result,  Theorem \ref{th:main2},  deals with the case where $G$ is semisimple,  the volume growth of $H_t$ is exponential, and allows $H$ to be either amenable or non-amenable. It establishes a mean ergodic theorems for the normalized sampling operators  in $L^p$, $ 1 \le p < \infty$, as well as a pointwise ergodic theorem for Sobolev functions, and a pointwise ergodic theorem with a rate of convergence when the functions are subanalytic.

Our third main result, Theorem \ref{th:main3},  assumes that of the underlying algebraic groups $G$ and $H$, at least one is semisimple, that the lattice $\Gamma$ is irreducible and that  the stability group $H$ is non-amenable subgroup which is non-amenably embedded $G$ (a term we will define below). Under these conditions,  the conclusions of Theorem \ref{th:main2} can be significantly strengthened, and we prove a mean and pointwise ergodic theorem  for the normalized sampling operators in every $L^p$, $ 1 < p < \infty$. Under a suitable strong spectral gap assumption, the pointwise ergodic theorem holds with a rate of convergence, provided  the function is subanalytic. 

We note that in the generality in which Theorem \ref{th:main2} and Theorem \ref{th:main3} are stated, the quantitative statement for  subanalytic functions is optimal, and  
 therefore the statements of the ergodic theorems are of optimal form.

Our fourth main result, Theorem \ref{th:main5}, assumes that $H$ is semisimple,  that the sets $G_t$ are defined by a norm, and that the volume of the sets $H_t$ is purely exponential.  When $H$ acts with a strong spectral gap on $G/\Gamma$, we prove mean, maximal and pointwise ergodic theorems with exponentially fast rate of convergence for the normalized sampling operators, for all functions in $L^p(D)$, $1 < p < \infty$. This result is of optimal form, and dispenses entirely with the assumption that the function is subanalytic. 

Finally, in Theorem \ref{th:main4} we note that the results just stated imply a wide variety of ratio ergodic theorems on homogeneous spaces.

Let us now turn to stating the main results in precise terms.

\subsection{Polynomial normalization of the sampling operators}

\begin{thm}\label{th:main1}
Assume that 
\begin{itemize}
\item $G$ is an arbitrary almost algebraic group, 
\item for $x\in X$, the stability group $\hbox{\rm Stab}_G(x)=H$ is of finite index 
is an almost direct product of a 
compact subgroup and an abelian diagonalisable subgroup,
\item the action of $\Gamma$ on the homogeneous space $X$ is ergodic.
\end{itemize}
Then there exist $b\in \mathbb{N}_{>0}$ and $t_0\in \RR_{>0}$ such that 
the sampling operators
$$
\pi_{X}(\tilde{\lambda}_t)\phi(x):=\frac{1}{t^b}\sum_{\gamma\in \Gamma_t} \phi(x\gamma)
$$
satisfy the following :
\begin{enumerate}
\item[(i)] {\rm Strong maximal inequality.}
For every $1<p\le\infty$,  compact domain $D$ of $X$, and $\phi\in L^p(D)$,
$$
\left\| \sup_{t\ge t_0} |\pi_{X}(\tilde{\lambda}_t)\phi|\right\|_{L^p(D)}\ll_{p,D} \|\phi\|_{L^p(D)}.
$$

\item[(ii)] {\rm Mean ergodic theorem.  }
For every $1\le p<\infty$,  compact domain $D$ of $X$, and $\phi\in L^p(D)$,
$$
\left\| \pi_{X}(\tilde{\lambda}_t)\phi(x) -\int_X \phi\,d\nu \right\|_{L^p(D)}\to 0
$$
as $t\to \infty$, where $\nu$ is a (nonzero) $G$-invariant measure on $X$.

\item[(iii)] {\rm Pointwise ergodic theorem.}
For every $1\le p\le \infty$, compact domain $D$ of $X$, and $\phi\in L^p(D)$,
$$
\lim_{t\to\infty} \pi_{X}(\tilde{\lambda}_t)\phi(x) = \int_X \phi\,d\nu
$$
for almost every $x\in X$.

\item[(iv)] {\rm Quantitative mean ergodic theorem in Sobolev spaces.}
Assume, in addition, that the group $G$ is semisimple, $\Gamma$ is an irreducible lattice in $G$, and 
$G$ has a strong spectral gap in $L_0^2(G/\Gamma)$.
Then there exists $l\in \mathbb{N}$ such that 
for every $1<p<q\le\infty$,  compact domain $D$ of $X$, and $\phi\in L_l^q(D)^+$,
the following estimate holds with $\delta_{p,q}>0$,
$$
\left\| \pi_{X}(\tilde{\lambda}_t)\phi(x) -\int_X \phi\,d\nu \right\|_{L^p(D)}\ll_{p,q, D} t^{-\delta_{p,q}}\|\phi\|_{L_l^q(D)}
$$
for all $t\ge t_0$.

\end{enumerate}
\end{thm}

Let us illustrate Theorem \ref{th:main1} by giving a pointwise ergodic
theorem for an action of a solvable group of exponential growth on
a space with infinite measure.

Let $\Delta$ be a lattice in $\mathbb{R}^d$,
and let $a$ be a $\mathbb{R}$-diagonalisable hyperbolic element of $\hbox{SL}_d(\mathbb{R})$ that leaves $\Delta$ invariant.
The group $\Gamma:=\left<a\right>\ltimes\Delta$ is a lattice in $G=\RR \ltimes\RR^n$, and  acts on $\mathbb{R}^d$ 
by affine transformations:
\begin{equation}\label{eq:affine}
x\cdot (a^n,v)=xa^n+v, \quad\quad  x\in \mathbb{R}^d,\;\; (a^n,v)\in\Gamma.
\end{equation}
Let $\lambda_{\max} >1$ denote the maximum of absolute values of the eigenvalues of $a$
and $\lambda_{\min}<1$ denotes the minimum of absolute values of the eigenvalues of $a$.
We fix a norm on $\mathbb{R}^d$ and consider the averaging sets 
\begin{equation}\label{eq:ggamma_t}
\Gamma_t=\{(a^n,v)\in \Gamma:\,\, n\in [t/\log(\lambda_{\min}), t/\log(\lambda_{\max})],\, \log \|v\|\le t\}
\end{equation}
Then Theorem \ref{th:main1} applies to the averages
$\sum_{\gamma\in\Gamma_t} \phi(x\gamma)$ on $\mathbb{R}^d$.
In particular, for every $\phi\in L^1(\mathbb{R}^d)$ with compact support,
\begin{equation}\label{eq:affine1}
\lim_{t\to\infty} \frac{1}{t}\sum_{\gamma\in\Gamma_t} \phi(v\gamma)=\frac{1}{\vol(\mathbb{R}^d/\Delta)}\int_{\mathbb{R}^d} \phi(x)\, dx
\quad\hbox{for almost every $v\in \mathbb{R}^d$.}
\end{equation}
Note that while the cardinality of the sets $\Gamma_t$ grows exponentially
(namely, $|\Gamma_t|\sim c\, e^{dt}t$ as $t\to \infty$ with $c>0$), the correct normalisation turns out to be
linear in this case. 
We refer to Section \ref{sec:affine1} below where this example is discussed in detail.

Finally, we note that Theorem \ref{th:main1} {\rm holds as stated for the sets $\Gamma_t$ defined when the homogeneous polynomial $P$ is replaced by any vector space norm, on  not necessarily a polynomial one 
(see Remark \ref{r:gen_norm} below.) }

\subsection{Exponential normalization of the sampling sets} We now turn to consider the situation where the growth of the sets $H_t$ is exponential, and begin by stating our second main result. 

\begin{thm}\label{th:main2}
Assume that
\begin{itemize}
\item the group $G$ is semisimple,  $\Gamma$ is an irreducible lattice in $G$,  and $G$ has a strong spectral gap in $L_0^2(G/\Gamma)$, 
\item for $x\in X$, the stability group $H=\hbox{\rm Stab}_G(x)$
is not of finite index in an almost direct product of a 
compact subgroup and an abelian $\mathbb{R}$-diagonalisable subgroup,
\item the action of $\Gamma$ on the homogeneous space $X=G/H$ is ergodic. 

\end{itemize}
Then there exist $a\in \mathbb{Q}_{> 0}$, $b\in \mathbb{N}_{\ge 0}$ and $t_0\in\mathbb{R}_{>0}$ such that 
the averages
$$
\pi_{X}(\tilde{\lambda}_t)\phi(x):=\frac{1}{e^{at} t^b}\sum_{\gamma\in \Gamma_t} \phi(x\gamma)
$$
satisfy for some $l\in \mathbb{N}_{\ge 0}$,
\begin{enumerate}
\item[(i)] {\rm Strong maximal inequality.} 
For every $1<p\le \infty$,  compact doamin $D$ of $X$, and $\phi\in L_l^p(D)^+$,
$$
\left\| \sup_{t\ge t_0} |\pi_{X}(\tilde{\lambda}_t)\phi|\right\|_{L^p(D)}\ll_{l, p,D} \|\phi\|_{L_l^p(D)}.
$$

\item[(ii)] {\rm Mean ergodic theorem.} There exists
a family of absolutely continuous measures $\{\nu_x\}_{x\in X}$ on $X$, with positive continuous
densities such that for every $1\le p<\infty$,  compact domain $D$ of $X$,  and 
$\phi\in L^p(D)$,
$$
\left\| \pi_{X}(\tilde{\lambda}_t)\phi(x) -\int_X \phi\,d\nu_x \right\|_{L^p(D)}\to 0
$$
as $t\to \infty$.


\item[(iii)] {\rm Pointwise ergodic theorem.}
For every $1<p\le\infty$, compact domain $D$ of $X$, and bounded $\phi\in L_l^p(D)^+$,
$$
\lim_{t\to\infty} \pi_{X}(\tilde{\lambda}_t)\phi(x) = \int_X \phi\,d\nu_x
$$
for almost every $x\in X$.

\item[(iv)] {\rm Quantitative pointwise ergodic theorem.}
For every $1<p\le \infty$,  compact domain $D$ of $X$, and a nonnegative continuous subanalytic function $\phi\in L_l^p(D)$,
the following asymptotic expansion holds 
$$
\pi_{X}\tilde{(\lambda}_t)\phi(x)=\int_X \phi\,d\nu_x+
\sum_{i=1}^b c_i(\phi,x) t^{-i}  +O_{x,\phi}\left(e^{-\delta(x,\phi) t}\right)
$$
for almost every $x\in X$ and all $t\ge t_0$ with some $\delta(x,\phi)>0$.

\end{enumerate}
\end{thm}

Let us give an example of application of  Theorem \ref{th:main2}.
Let $\Gamma$ be a lattice in $\hbox{SL}_d(\mathbb{R})$ and $\Gamma_t=\{\gamma\in\Gamma:\, \log \|\gamma\|\le t\}$
denote the norm balls with respect to the standard Euclidean norm
$\|\gamma\|=\left(\sum_{i,j=1}^d \gamma_{ij}^2\right)^{1/2}$.
We consider the action of $\Gamma$ on the projective space $\mathbb{P}^{d-1}(\mathbb{R})$.
Then Theorem \ref{th:main2}(iv) implies that  for any nonnegative continuous subanalytic function
$\phi\in L_l^p(\mathbb{P}^{d-1}(\mathbb{R}))$ with $p>1$ (for some explicit $l\ge 0$)  and for almost every $v\in X$,
there exists $\delta>0$ such that
\begin{equation}\label{eq:projective}
\frac{1}{e^{(d^2-d)t}} \sum_{\gamma\in\Gamma_t}
  \phi(v\gamma)= c_{d}(\Gamma) \int_{\mathbb{P}^{d-1}(\mathbb{R})} \phi(w)\, d\xi(w)
  +O_{\phi,v}(e^{-\delta t}),
\end{equation}
where $c_{d}(\Gamma)>0$, $\delta=\delta(v,\phi)> 0$,  and $\xi$ denote the $\hbox{SO}_d(\mathbb{R})$-invariant probability 
measure on $\mathbb{P}^{d-1}(\mathbb{R})$. This example is in more detail discussed
in Section \ref{sec:proj}. Further examples and applications of Theorem \ref{th:main2} are discussed in Section \ref{sec:examples}.

 Note that in Theorem \ref{th:main2} the group  $H$ can be a solvable, for example. In that case, no rate of convergence 
can possibly hold for the operators $\pi_X(\tilde{\lambda}_t)$ acting in Lebesgue space, and results in Sobolev spaces are the best that can be achieved. 
The same remark applies of course
to Theorem \ref{th:main1}.

For a special class of homogeneous spaces $X$, we obtain an
improved version of Theorem \ref{th:main2} with $L^p$-norms in place of Sobolev norms, to which we now turn.

\subsection{Non-amenable stabilizers and quantitative ergodic theorems}

\begin{thm}\label{th:main3}
Assume that at least one of the following conditions is satisfied. 
\begin{itemize}
\item  $G$ is an arbitrary almost algebraic group, the stability group $H=\hbox{\rm Stab}_G(x)$ is semisimple, and $H$ 
has a strong spectral gap in $L_0^2(G/\Gamma)$, 
\item $G$ is a semisimple group which has a strong spectral gap in $L_0^2(G/\Gamma)$,  and $H$ is any
  almost algebraic subgroup which is unimodular and non-amenably embedded (see Definition \ref{NA
    embedding} below).  

\end{itemize}
Then there exist $a\in \mathbb{Q}_{> 0}$, $b\in \mathbb{N}_{\ge 0}$ and $t_0\in \mathbb{R}$ such that 
the normalized sampling operators 
$$
\pi_{X}(\tilde{\lambda}_t)\phi(x):=\frac{1}{e^{at} t^b}\sum_{\gamma\in \Gamma_t} \phi(x\gamma)
$$
satisfy the following
\begin{enumerate}
\item[(i)] {\rm Strong maximal inequality.}
For every $1<p\le \infty$,  compact doamin $D$ of $X$, and $\phi\in L^p(D)$,
$$
\left\| \sup_{t\ge t_0} |\pi_{X}(\tilde{\lambda}_t)\phi|\right\|_{L^p(D)}\ll_{p,D} \|\phi\|_{L^p(D)}.
$$

\item[(ii)] {\rm Pointwise ergodic theorem.}
For every $1<p\le\infty$,  compact domain $D$ of $X$, and $\phi\in L^p(D)$,
$$
\lim_{t\to\infty} \pi_{X}(\tilde{\lambda}_t)\phi(x) = \int_X \phi\,d\nu_x
$$
for almost every $x\in X$.

\item[(iii)] {\rm Quantitative pointwise ergodic theorem.}
For every $1<p\le \infty$,  compact domain $D$ of $X$, and a nonnegative continuous subanalytic function
$\phi$ with $\supp(\phi)\subset D$, the following asymptotic expansion holds 
$$
\pi_{X}(\tilde{\lambda}_t)\phi(x)=\int_X \phi\,d\nu_x+
\sum_{i=1}^b c_i(\phi,x) t^{-i}  +O_{\phi,x}\left(e^{-\delta(x,\phi) t}\right)
$$
for almost every $x\in X$ and all $t\ge t_0$ with some $\delta(x,\phi)>0$.

\end{enumerate}
\end{thm}

To exemplify our general results, 
let us consider the action of a lattice $\Gamma$ in the orthogonal group $\hbox{SO}_{d,1}(\mathbb{R})^0$
on the quadratic surface
$$
X=\{x\in\mathbb{R}^{d+1}:\, x_1^2+\cdots +x_d^2-x_{d+1}^2=1\},
$$
The space $X$ is known as the de-Sitter space, and the problem of distribution
of orbits of $\Gamma$ in $X$ was raised by Arnol'd (see \cite[1996-15, 2002-16]{A}).
Theorem \ref{th:main1} solves the problem for $d=2$, and 
Theorem \ref{th:main3} solves this problem for general $d\ge 3$.
It will be convenient to use the polar coordinate system on $X$:
\begin{equation}\label{eq:polar}
\mathbb{R}\times S^{d-1} \to X: (r,\omega)\mapsto (\omega_1\cosh r,\ldots, \omega_d\cosh r,\sinh r).
\end{equation}
For $d=2$, we obtain from Theorem \ref{th:main1} that for every $\phi\in L^1(X)$ with compact support and
for almost every $v\in X$,
\begin{equation}\label{eq:quad0}
\lim_{t\to\infty} \frac{1}{t} \sum_{\gamma\in\Gamma_t}
  \phi(v\gamma)= c_{2}(\Gamma)
 \int_{X} \phi(r,\omega)\, (\cosh r)\,dr\, d\omega.
\end{equation}
for some $c_2(\Gamma)>0$.
For $d\ge 3$, we obtain from Theorem \ref{th:main3}
that for every nonnegative continuous subanalytic function $\phi$ with compact support
and almost every $v\in X$, the following asymptotic expansion holds
\begin{equation}\label{eq:quad1}
\frac{1}{e^{(d-2)t}} \sum_{\gamma\in\Gamma_t}
  \phi(v\gamma)= \frac{c_{d}(\Gamma)}{\left(1+v_d^2\right)^{(d-2)/2}} \int_{X} \phi(r,\omega)\, 
\frac{(\cosh r)^{d-1}dr\, d\omega}{\left(1+(\sinh r)^2\right)^{(d-2)/2}}
+O_{v,\phi}(e^{-\delta t})
\end{equation}
for some $c_d(\Gamma)>0$ and $\delta=\delta(v,\phi)>0$.
Note that the limit measure in this case is not a $\Gamma$-invariant measure,
and moreover it depends nontrivially on the initial point $v$.
Further applications of Theorem \ref{th:main3} are discussed in Section \ref{sec:examples}.

To motivate the discussion immediately below, let us note that in the present example it is in fact possible to obtain a much stronger conclusion. Both the restriction that $\phi$ is subanalytic, as well as the dependence $\delta(v,\phi) $ on $v$ and $\phi$ can be dispensed with, as follows from Theorem \ref{th:main5}.  We will discuss this example in more detail in Section \ref{sec:quad} below.

\subsection{Volume regularity and quantitative ergodic and ratio theorems}

\subsubsection{On the role of volume regularity in the proofs} 
A interesting feature exhibited in Theorem \ref{th:main2}(iv) and Theorem \ref{th:main3}(iii) is that the
quality of quantitative ergodic theorems  stated in them is genuinely constrained. While we assume
spectral gap conditions which imply exponential norm decay of the averaging operators 
supported on $H_t$ acting on $L^2_0(G/\Gamma)$, the dual operators $\pi_X(\tilde{\lambda}_t) $ on $L^2(D)$ do not satisfy such an exponential norm or pointwise decay estimate. For subanalytic function on $D$, the asymptotic expansion :
$$
\pi_{X}(\tilde{\lambda}_t)\phi(x)=\int_X \phi\,d\nu_x+
\sum_{i=1}^b c_i(\phi,x) t^{-i}  +O_{\phi,x}\left(e^{-\delta(x,\phi) t}\right)
$$
which holds for almost every $x\in X$ with some $\delta(x,\phi)>0$, implies that if $b> 0$ the rate of
almost sure convergence to the limiting distribution is at least $t^{-1}$, but typically not faster, so it is {\it not} exponential.  Let us now explain the reason for the occurrence of this phenomenon, and then state a substantial improvement to the quantitative ergodic theorems under suitable conditions. 

A fundamental reduction that appears repeatedly in our analysis below 
is the comparison of the normalized sampling operators on a $\Gamma$-orbit in $X$ :
$$
\pi_{X}(\lambda_t)\phi(x)=\frac{1}{\vol (H_t)}\sum_{\gamma\in \Gamma_t} \phi(x\gamma),
$$
with the normalized  sampling operators on the $G$-orbit in $X$ :

$$
\pi_{X}(\lambda_t^G)\phi(x)=\frac{1}{\vol (H_t)}\int_{g\in G_t} \phi(xg)dm(g)\,
$$
where here we use the intrinsic normalization by $\vol(H_t)$ rather than by $e^{at}t^b$. 

It will develop that the difference $\pi_X(\lambda_t)-\pi_X(\lambda_t^G)$ can be estimated very well
under very general assumptions. It converges to zero in the mean and pointwise, and in fact with an
exponentially fast quantitative rate if the corresponding operators supported on $H_t$ satisfy the spectral gap estimates and quantitative ergodic theorem in $L^2(G/\Gamma)$. Furthermore,  
$\pi_X(\lambda_t^G)\phi(x)-\int_D \phi\, d\nu_x$,  converges to zero almost surely, and  this identifies the limit of $\pi_X(\lambda_t)\phi(x)$ in the mean and pointwise ergodic theorems, as the limiting density  $\nu_x$. 
To  obtain a quantitative ergodic theorem for $\pi_X(\lambda_t)$ all we need to do is make the latter convergence result quantitative.  To understand the limitations here, let us note the following alternative expression for $\pi_X(\lambda_t^G)$ and its connection with the limiting distribution  $\tilde\nu_x$.  Let $\mathsf{p}_X$ denotes the canonical projection from $G$ to $X=H\setminus G$, let $\mathsf{s}$ denotes a measurable section from $X$ to $G$ which is bounded on compact sets, and $\xi$ denote the canonical density on $X$. Then  by the discussion preceding equation \eqref{eq:point} below, with $H_t[g_1,g_2]=H\cap g_1G_t g_2^{-1}$,
\begin{align*}
\pi_X(\lambda^G_t)\phi(x) 
&=
\frac{1}{\rho(H_t)}\int_{G_t} \phi(\mathsf{p}_X(\mathsf{s}(x)g))\, dm(g)\\
&=\frac{1}{\rho(H_t)}\int_{(y,h):\, \mathsf{s}(x)^{-1}h\mathsf{s}(y)\in G_t} \phi(\mathsf{p}_X(h\mathsf{s}(y)))\, d\rho(h)d\xi(y)\\
&=\int_D\phi(y)\frac{\rho(H_t[\mathsf{s}(x),\mathsf{s}(y)])}{\rho(H_t)}\,d\xi(y).
\end{align*}

Now for every $g_1,g_2\in G$, we shall show that the limit
$$
\Theta(g_1,g_2):=\lim_{t\to\infty} \frac{\rho(H_t[g_1,g_2])}{\rho(H_t)} 
$$
exists (see Lemma \ref{p:alpha}), and defines the family of  limiting distributions $\nu_x\,,\, x\in X$
(see \eqref{eq:nu_x}). 

Consequently, the convergence properties of $\pi_X(\lambda_t^G)\phi(x)$ to $\int_D \phi(y)d\nu_x(y)$ are determined by the convergence properties of 
$\rho(H_t[\mathsf{s}(x),\mathsf{s}(y)])/\rho(H_t)$ to the limiting distribution $\nu_x$. Thus the regularity properties of the volume of the sets $H_t$ are crucial, and control
the final conclusion in the quantitative ergodic theorem.  The parameter  $a$ and $b$  that appear in the quantitative results for the operators $\pi_X(\lambda_t)$ acting on subanalytic functions 
in Theorem \ref{th:main2}(iv) and Theorem \ref{th:main3}(iii), are those that appear in the development
$\vol(H_t)= e^{at}(c_b t^b+\cdots + c_0)+ O(e^{(a-\delta_0)t})$ given by Theorem \ref{th:book}(i) below. 

The foregoing discussion shows that the results for subanalytic functions are optimal as stated, but suggest that under the stronger assumption that the asymptotic development of the volume holds with $b=0$, the results in the ergodic theorem should be stronger. We will now state our fourth main result, which gives the optimal formulation of quantitative ergodic theorems under this additional assumption.   As we will see below,  this assumption is satisfied by a large collection of important examples.

\begin{thm}\label{th:main5}
Assume that the following conditions are satisfied. 
\begin{itemize}
\item  $G\subset \hbox{\rm SL}_d(\RR)$ is an arbitrary almost algebraic group,
\item the stability group $H=\hbox{\rm Stab}_G(x)\subset G$ is semisimple and 
has a strong spectral gap in $L_0^2(G/\Gamma)$, 
\item the homogeneous polynomial $P$ is replaced by a norm on $\hbox{\rm Mat}_d(\RR)$, 
\item the volumes of $H_t$ satisfy $\vol(H_t)\sim c\,e^{at}$ as $t\to\infty$, with $a,c>0$.
\end{itemize}
Then the normalized sampling operators $\pi_{X}(\tilde{\lambda}_t)\phi(x)=\frac{1}{e^{at}}\sum_{\gamma\in
  \Gamma_t} \phi(x\gamma)$ satisfy, for every compact domain $D\subset X $ and $t\ge t_0$,
\begin{enumerate}
\item[(i)] {\rm Quantitative mean ergodic theorem in Lebesgue spaces.}
For $ 1 < p < \infty$,  a suitable $\delta_p > 0$ independent of $D$,  and for every $\phi\in L^p(D)$ 
$$ \norm{\pi_{X}(\tilde{\lambda}_t)\phi(x)-\int_X \phi\,d\nu_x}_{L^p(D)} \ll_{p,D} \norm{\phi}_{L^p(D)} e^{-\delta_p t} \,.$$

\item[(ii)] {\rm Quantitative  maximal ergodic theorem in Lebesgue spaces.} The family $\pi_X(\tilde{\lambda}_t)$ satisfies the $(L^p,L^w)$-exponential strong 
maximal inequality  (for  some $1< w <  p< \infty$), namely there exists $\delta_{p,w}>0$ such that for every $\phi\in L^p(D)$, 
$$
\left\|\sup_{t\ge t_0} e^{\delta_{p,w} t}
 \left|\pi_X(\tilde{\lambda}_t)\phi(x)-\int_D \phi\,d\nu_x\right|\right\|_{L^w(D)}
\ll_{p,w,D}\|\phi\|_{L^p(D)}.
$$

\item[(iii)] {\rm Uniform quantitative pointwise ergodic theorem in Lebesgue spaces.}
For every $\phi \in L^p(D)$, $ 1 < p < \infty$, and almost every $x\in X$,
$$\abs{\pi_{X}(\tilde{\lambda}_t)\phi(x)-\int_X \phi\,d\nu_x}\ll_{x,\phi,D} e^{-\delta_p t}\,,$$
where $\delta_p > 0$ is independent of $x$, $\phi$ and $D$. 
\end{enumerate}
\end{thm}

\begin{rem}\label{b> 0}{\rm  \begin{enumerate}
\item We note that when the volume growth is not purely exponential, namely when $b > 0$, the quantitative mean and pointwise ergodic theorems hold as stated 
in Theorem \ref{th:main5}(i) and (iii), but the speed is $t^{-\eta_p}$(with $\eta_p > 0$) rather than
$e^{-\delta_p t}$ (see Subsection 10.4).
\item In \cite[Thm. 1.6]{gn3}, 
we have applied Theorem \ref{th:main5} to the case of a dense $S$-arithmetic lattice in a connected semisimple Lie group, which acts by isometries on the group variety itself.  We have established  there an exponentially fast quantitative equidistribution theorem, {\it for every starting point, and with a fixed rate}, in the case of H\"older functions. The statement of \cite[Thm. 1.6]{gn3} refers to the case of norms with purely exponential growth of balls. When the growth is not purely exponential, the rate of equidistribution is polynomial.   
\end{enumerate}
}
\end{rem}

We note that the assumptions of Theorem \ref{th:main5} are verified in many interesting cases, when we replace the polynomial 
$P$ by a norm with purely exponential growth of balls. These include 
Examples 11.1 of the action of $\hbox{SO}_{n,1}(\mathbb{R})^0$ on de-Sitter space, Example 11.4 of dense subgroups of semisimple Lie groups acting 
by translation on the group, Example 11.5 of distribution of values of indefinite quadratic forms, and Example 1.7 of affine actions of lattices. 
These examples will be explained in detail in \S 11.

\subsubsection{Ratio ergodic theorems}
In the case of a single transformation acting on an infinite, $\sigma$-finite measure space, it is well-known that a general ratio ergodic theorem holds, and 
it is natural to consider ergodic ratio theorems in our context as well.  The results stated above  describe the limiting behavior of the operators $\pi_X(\tilde{\lambda}_t)$, and so it immediately follows that they imply a limit theorem for their ratios. We record this fact in the following result.

\begin{thm}\label{th:main4} 
Let $G$, $H$, $X$, $\Gamma$ and $P$ be as in the previous section. Consider a compact domain $D\subset X$, and any two functions $\phi,\psi\in L^p(D)$ of compact support contained in $D$, with $\psi$ non-negative and non-zero.  
\begin{enumerate}
\item[(i)] {\rm Ratio ergodic theorem.}
Assume that the conditions of Theorem \ref{th:main1},  or Theorem \ref{th:main3}  are satisfied. Then 
as $t\to\infty$,
$$\frac{\sum_{\gamma\in \Gamma_t} \phi(x\gamma)}{\sum_{\gamma\in \Gamma_t} \psi(x\gamma)}\longrightarrow \frac{\int_D \phi\, d\nu_x}{\int_D \psi\, d\nu_x}$$
for almost every $x\in D$.  Under the conditions of Theorem 1.2, the same result applies provided $\phi$ and $\psi$ as above are in $L^p_l(D)$. 
\item[(ii)] {\rm Quantitative ratio ergodic theorem.}
Assume that the conditions of Theorems \ref{th:main2} (iv) or Theorem \ref{th:main3}(iii) hold. Let $b$ and $\delta$ be the volume growth parameters stated there. Assume that $\phi, \psi\in C_c(D)$ are subanalytic and non-negative (and in  $L^p_l(D)$ when assuming the conditions of Theorem \ref{th:main2}(iv)). Then convergence in the ratio ergodic theorem takes place at the rate 
$$\abs{\frac{\sum_{\gamma\in \Gamma_t} \phi(x\gamma)}{\sum_{\gamma\in \Gamma_t} \psi(x\gamma)}-\frac{\int_D \phi\, d\nu_x}{\int_D \psi \,d\nu_x}}\ll_{x,\phi,\psi} E(t)$$
where $E(t)=\frac1t$ if $b>0$, and $E(t)=e^{-\delta_p t}$ for some $\delta_p(x,\phi,\psi) > 0$, if $b=0$.  
\item[(iii)] {\rm Uniform quantitative ratio ergodic theorem.}
Under the conditions of  Theorem \ref{th:main5}, the exponential rate of convergence $\delta_p$ in the ratio ergodic theorem stated in (ii) is independent of $\phi$, $\psi$, $ x$ and $D$. 

\end{enumerate}
\end{thm}

Consider now  the problem of establishing equidistribution for the ratio of averages, namely convergence for every single starting point, when the functions are continuous, or satisfy (say) a H\"older regularity condition.  In the case of  dense subgroups of isometries of Euclidean spaces this problem was introduced half a century ago by Kazhdan \cite{K}. In \cite{gn3},  we have established a quantitative ratio equidistribution theorem for certain dense subgroups of semisimple Lie groups acting on the group manifold, with respect to H\"older functions. This result has the remarkable feature that it holds for every single starting point $x\in X$, with the same rate of exponentially fast convergence (provided the volume growth of balls is purely exponential). 
This result is  based on the quantitative ergodic theorem in $L^2$ stated in Theorem \ref{th:main5}.

\subsection{Comments on the development of the duality principle}

Given an lcsc group $G$, and two closed subgroups $H_1$ and $H_2$, the  principle of duality
(namely, that properties of the $H_1$-action on $G/H_1$ and  the properties 
of the $H_2$-actions on $H_1\setminus G$ are closely related) has been a mainstay of homogeneous dynamics for at least half a century. It has been used in a diverse array of different applications, as demonstrated in the list below.  Needless to say, the list constitutes just a small sample of the extensive literature in homogeneous dynamics which can be construed as applying duality arguments in one form or another.  

As to general duality properties, let us mention the following often-used fundamental facts.  

\begin{itemize}
\item Minimality : $H_1$ is minimal on $G/H_2$ if and only if $H_2$ is minimal on $H_1\setminus G$. This fact and its application to Diophantine approximation in the context of flows on nilmanifolds has been discussed already in \cite{AGH}, see e.g. Chapter VIII by Auslander and Green and Chapter XI by Greenberg. 
\item Ergodicity : $H_1$ is ergodic on $G/H_2$ if and only if $H_2$ is ergodic on $H_1\setminus G$. This result is known as Moore's ergodicity theorem \cite{Mo}.   
\item Amenability : $H_1$ acts amenably on $G/H_2$ if and only if $H_2$ acts amenably on $H_1\setminus G$. The notion of amenability used here is that defined by Zimmer, see  \cite{Z1}. 
\end{itemize}

Let us now restrict the discussion to the case when $G$ is semisimple, $H_2=\Gamma$ is an irreducible lattice, and  $H_1$ is a minimal parabolic subgroup, or a unipotent subgroup of it. 

\begin{itemize}
\item Unique ergodicity of horocycle flows $U$ on $G/\Gamma$ when $G=\hbox{SL}_2(\RR)$ has been proved by Furstenberg \cite{Fu} using the dynamical properties of the $\Gamma$-action on $U\setminus G$.
This method has subsequently been generalized by Veech \cite{V2} for semisimple groups.  
\item Dani and Raghavan \cite{DR} studied orbits of frames under discrete linear groups, and in
  particular considered the connection between the density properties of $\Gamma$ acting on $U\setminus
  G$, and $U$ acting on $G/\Gamma$, where $U$ is a horospherical unipotent subgroup.  Minimality of the
  action of a horospherical group when the lattice is uniform has been proved earlier by Veech
  \cite{V1}. 
Related equidistribution results (namely, the convergence of sampling operators $\pi_X(\lambda_t)$ in the
space of continuous functions) were established in \cite{L1,n,G2}.

\item  Dani \cite{D1} established the topological version of Margulis factor theorem, namely that continuous $\Gamma$-equivariant factors of the boundary $P\setminus G$ are necessarily of the form $Q\setminus G$, where $Q$ is a parabolic subgroup containing $P$. The proof uses the minimality properties of the $P$-action on $G/\Gamma$.  Special cases have previously been established by Zimmer \cite{Z2}, Spatzier \cite{Sp}, and a generalization was established by Shah \cite{Sh}. 
\item Using duality-type arguments, equidistribution of the lattice orbits on the boundary $P\setminus G$
for $G=\hbox{SL}_d(\mathbb{R})$ was established in \cite{G1}. 
The method  utilizes the equidistribution of averages on $P$ acting on
$C_c(G/\Gamma)$.  Subsequently equidistribution of the $\Gamma$-orbits on the boundary was established
for general semisimple groups in \cite{go}, and using a different method in \cite{gm}.  
\end{itemize}

Let us now proceed with $G$ a semisimple group, $H_2=\Gamma$ an irreducible lattice subgroup, $H_1=H$ an algebraic subgrpup, and assume that $H\setminus G$ has a $G$-invariant measure. 

\begin{itemize}
\item The problem of counting the number of points of an orbit in a ball for  isometry groups of
  manifolds of (variable) negative curvature was already considered in Margulis' thesis (see \cite{M1}). The
  basic connection between this problem and the mixing property of the geodesic flow
  established there has been a major source of influence in the development of duality arguments.
\item  Several problems in Diophantine approximation in Euclidean spaces can be approached via duality arguments, a fact originally due to Dani \cite{D2} who studied it systematically.  This fact is referred to as the Dani correspondence by Kleinbock and Margulis \cite{KM1,KM2}, who have extended it further. 
\item Duality considerations also played an important role in Margulis' celebrated solution of Oppenheim
  conjecture (see \cite{M2}), where the method of proof utilizes the dynamics of
  $\hbox{SO}_{2,1}(\mathbb{R})$ in the space of lattices $\hbox{SL}_3(\RR)/\hbox{SL}_3(\mathbb{Z})$ in order  
to analyse the values of the quadratic form on integer points, namely on an orbit of the lattice $\hbox{SL}_3(\ZZ)$.  A quantitative approach to this problem was subsequently developed by Dani and Margulis \cite{DM} and Eskin, Margulis and Mozes \cite{EMM}.  
 \item The general lattice point counting problem on homogeneous spaces is to give precise asymptotics for the number of points in discrete orbits of $\Gamma$ in $H\setminus G$ contained in a ball, and the most studied case is when $H$ is a symmetric subgroup. 
By duality, any such discrete orbit determines canonically  a closed orbit of $H$ in $G/\Gamma$. Quantitative results in the lattice point counting problem can be established by applying harmonic analysis to the 
``$H$-periods'' of suitable automorphic functions on $G/\Gamma$, namely by estimating their integral on such closed orbits.  This spectral approach in the case of higher rank semisimple groups and symmetric subgroups was first applied by  Duke, Rudnick and Sarnak \cite{drs}. 
\item Establishing the main term in the lattice point counting problem for $\Gamma$ on $H\setminus G$ can
  be reduced to establishing equidistribution of translates of the closed $H$-orbit $ H
  \Gamma\subset G/\Gamma$.  Eskin and McMullen \cite{EM} have established the equidistribution of these
  translates using the  mixing property on $A$ on $G/\Gamma$.  Eskin,
  Mozes  and Shah \cite{EMS} have established equidistribution of translates of a closed $H$-orbit 
using the theory of unipotent flows on $G/\Gamma$. The mixing method can also be made quantitative, as shown by Maucourant
  \cite{Mau} and \cite{BO}, using  estimates for rates of mixing in Sobolev spaces.  

\item  An important development regarding the duality principle and equidistribution of lattice actions on homogeneous varieties is due to Ledrappier. In  \cite{L1,L2}
he considered the action of a lattice subgroup of $\hbox{SL}_2(\RR)$ on the plane $\RR^2$, and used
duality arguments to establish convergence of the sampling operators supported on $\Gamma_t$, in the
space of continuous functions. Ledrappier's results  revealed for the first time  several of the
remarkable features that arise in the context of infinite-volume homogeneous spaces, including the
appearance of a limiting density different than the invariant measure, the fact that the limiting density
is not necessarily invariant under the lattice action, and the fact that it  depends on the initial point in the
orbit under  consideration. Subsequently Ledrappier and Pollicott \cite{lp}\cite{lp2} have generalised
these results to $\hbox{SL}_2$ over other fields. Another appearance of duality argument is due to 
Maucourant, who in his thesis considered the $ \Gamma$-action on $G/A$ where $A$ is an $\RR$-split torus, via ergodic
theorems for $A$ acting on $G/\Gamma$. 
A systematic general approach to the duality principle for a large class of groups and homogeneous spaces, including an analysis of the limiting distribution and the requisite properties of volume regularity was carried out in \cite{GW}, and applied to obtain diverse  equidistribution results for actions of lattices on homogeneous spaces.  
\item The quantitative solution of the lattice point counting problem which appears in \cite{gn2} can be viewed as a quantitative duality argument,  applied to the case where the subgroup $H$ is in fact equal to $G$. 
\item The Boltzmann-Grad limit for periodic Lorentz gas has been studied by Marklof and Str\"ombergsson
  \cite{ms1,ms2}. Using periodicity, one can reduce the original problem to analysing distribution
on the space of lattices $\hbox{SL}_d(\RR)/\hbox{SL}_d(\mathbb{Z})$, which can be treated 
using the theory of unipotent flows.
\end{itemize}
Given the centrality of the duality principle in homogeneous dynamics, it is a most natural problem  to establish mean and pointwise ergodic theorems in Lebesgue spaces for lattice actions on homogeneous varieties, including a rate of convergence in the presence of a spectral gap. The present paper is devoted to the solution of this problem, which we expect will have several significant applications, including a quantitative approach to Diophantine 
approximation on homogeneous algebraic varieties, developing \cite{ggn} further. 
 However, we are not aware of any previous results in the literature establishing ergodic theorems in Lebesgue spaces for infinite volume homogeneous spaces.    
\subsection{Organization of the paper}

In Section \ref{sec:notation} we set-up notations that will be used throughout the paper.
Sections \ref{sec:max_ineq}--\ref{sec:quant_pointwise}  form the core of the paper.
In this part we develop 
the general ergodic-theoretic duality framework in the setting of locally compact topological
groups and their homogeneous spaces. 
Section \ref{sec:volume} is devoted to the crucial issue of volume regularity, and contains a number of results
concerning regularity properties of balls defined by polynomials in algebraic group.
In Section \ref{sec:alg} we develop ergodic theory of actions of algebraic
groups on probability measure space, which are subsequently used as an input
for the general duality principle developed in the previous sections.
In Section \ref{sec:proof} we complete the proof of the four main theorems stated in the introduction.
Finally, Section \ref{sec:examples} is devoted to examples illustrating our main results.

\subsection*{ Acknowledgements.} The authors would like to thank Barak Weiss for his very valuable
contributions to this project. We also would like to express our gratitude to the Ergodic Theory Group
at the F\'ed\'eration Denis Poisson who gave us an opportunity to explain the present work
in the lecture series ``Th\'eorie ergodique des actions de groupe'' held in Tours in April, 2011.
A.G. would like to thank the Princeton University and especially E. Lindenstrauss for their hospitality
during 2007--2008 academic year when part of this work has been completed.

\section{Basic notation}\label{sec:notation} 

In this section we introduce notation that will be used throughout the paper.  
\subsection{Sobolev norms}\label{sec:sobolev}
Given a compact domain $D$ of $\mathbb{R}^d$, we denote by
$L^p_l(D)$ the space of functions $\phi$ on $\mathbb{R}^d$ with support contained in 
$D$ for which all distributional derivatives of $\phi$
of order at most $l$ are in $L^p(\mathbb{R}^d)$.
The space $L^p_l(D)$ is equipped with the Sobolev norm
$$
\|\phi\|_{L^p_l (D)}=\sum_{0\le|s|\le l} \|\mathcal{D}^s\phi\|_{L^p(\mathbb{R}^d)}, 
$$
where $\mathcal{D}^s$ denotes the partial derivative corresponding to the multi-index $s$.
More generally, let $D$ be a compact subset of a manifold $X$ of dimension $d$.
We pick a system of coordinate charts 
$\{\omega_i:\mathbb{R}^d\to M\}_{i=1}^n$ that cover $D$
and a partition of unity $\{\psi_i\}_{i=1}^n$ subordinate to this system of charts
such that $\sum_{i=1}^n \psi_i=1$ on $D$.
We denote by $L^p_l(D)$ the space of functions $\phi$ on $X$ with support contained in 
$D$ such that for every $i=1,\ldots,n$, the function
$(\phi\cdot \psi_i) \circ \omega_i$ belongs to $L^p_l(\mathbb{R}^d)$.
The space $L^p_l(D)$ is equipped with the Sobolev norm
\begin{equation}\label{eq:sobolev_norm}  
\|\phi\|_{L^p_l(D)}=\sum_{i=1}^n \|(\phi\cdot \psi_i)\circ \omega_i\|_{L^p_l}.
\end{equation}
Clearly, this definition of the Sobolev norm depends on a choice of the system of
coordinate charts and the partition of unity, but using compactness of $D$,
one checks that different
choices lead to equivalent norms, and for this reason we have suppressed
$\omega_i$'s and $\psi_i$'s from the notation.

\subsection{Homogeneous spaces and measures}\label{sec:space}
Let $G$ be a locally compact second countable (lcsc) group, $H$ a closed subgroup of $G$,
and $X:= H\backslash G$ is a homogeneous space of $G$. In the case when our discussion
involves Sobolev norms, we always assume, in addition, that $G$ is a connected Lie group and $H$ is a 
closed Lie subgroup. Then $X$ has the structure of a smooth manifold.
We denote by $\mathsf{p}_X$ the natural projection map 
$$
\mathsf{p}_X:G\to X: g\mapsto H g,
$$
and by $\mathsf{s} :X\to G$ a measurable section of this map, 
i.e., a map such that $\mathsf{p}_X\circ \mathsf{s}=id$.
Since $G$ is locally compact, such a section
always exists, with the additional property that it is  bounded on compact sets. 
We also use the notation:
$$
\mathsf{r}:=\mathsf{s}\circ\mathsf{p}_X: G\to G\quad\hbox{and} \quad \mathsf{h}(g):=\mathsf{r}(g)g^{-1}\in H.
$$
Then for $g\in G$,
\begin{equation}\label{eq:r_h}
g= \mathsf{h}(g)^{-1} \mathsf{r}(g).
\end{equation}
We note that these maps satisfy
\begin{equation}\label{eq:invariance}
\mathsf{p}_X(hg)=\mathsf{p}_X(g)\quad\hbox{and}\quad \mathsf{h}(hg)=\mathsf{h}(g)h^{-1}
\end{equation}
for $g\in G$ and $h\in H$.

Let $\Gamma$ be a lattice subgroup in $G$, that is, a discrete subgroup of finite covolume.
We set $Y:=G/\Gamma$ and denote by 
$$
\mathsf{p}_Y:G\to Y: g\mapsto g\Gamma
$$
the natural projection map.

Let $m$ be a left Haar measure on $G$ and $\rho$ a left Haar measure $H$.
It follows from invariance that the measure $m$ has the decomposition
\begin{equation}\label{eq:measure}
\int_G f(g)\, dm(g)=\int_{H\times X} f(h\cdot \mathsf{s}(x))\,d\rho(h)d\xi(x),\quad \phi\in L^1(G),
\end{equation}
where $\xi$ is a Borel measure on $X$. We note that the measure $\xi$ does depend on the section
$\mathsf{s}$. If fact, if $\mathsf{s}_1$ and $\mathsf{s}_2$ are two sections, then $\mathsf{s}_1(x)\mathsf{s}_2 (x)^{-1}\in H$ and the corresponding
measures $\xi_1$ and $\xi_2$ on $X$ are related by
\begin{equation}\label{eq:1_2}
d\xi_1(x)=\Delta_H(\mathsf{s}_1(x)\mathsf{s}_2(x)^{-1})\, d\xi_2(x),
\end{equation}
where $\Delta_H$ denotes the modular function for the Haar measure on $H$.
In particular, when $H$ is unimodular, the measure $\xi$ is canonically defined.

Our arguments below involve functions supported in a compact domain  $D$ of $X$, and the measurable
section $\mathsf{s}:D\to G$ which 
we choose so that $\mathsf{s}(D)$ is bounded in $G$. 
It follows that  $\xi(D)<\infty$. We denote by $L^p(D)$ the space of functions $\phi$
on $X$ with $\supp(\phi)\subset D$ and
$$
\|\phi\|_{L^p(D)}:=\left(\int_D |\phi|^p\, d\xi\right)^{1/p}<\infty.
$$
While this norm depends on a choice of $\mathsf{s}$, it follows from (\ref{eq:1_2})
that different choices lead to equivalent norms. 

Let $\mu$ be the Haar measure on $Y=G/\Gamma$ induced by $m$.
Namely, $\mu$ is defined by $\mu(A)=m(\mathsf{p}_Y^{-1}(A)\cap \mathcal{F})$ where $\mathcal{F}\subset G$
is a fundamental domain for the right action of $\Gamma$ on $G$.
We normalise Haar measure $m$ on $G$ so that $\mu(G/\Gamma)=1$.

Throughout the paper, 
unless stated otherwise, we use the measure $m$ on $G$,
the measure $\xi$ on $X$, and the measure $\mu$ on $Y$.

Let $\{G_t\}$ be an increasing sequence of compact subsets of $G$. For a subset $S$ of
$G$, we set 
$$
S_t:=S\cap G_t,
$$
and more generally for $g_1,g_2\in G$, we set
\begin{equation}\label{eq:ss_t}
S_t[g_1,g_2]:=S\cap g_1G_tg_2^{-1}.
\end{equation}
Let now $H$ be a closed subgroup. Assuming that the limit
$$
\Theta(g_1,g_2):=\lim_{t\to\infty} \frac{\rho(H_t[g_1,g_2])}{\rho(H_t)}
$$
exists (see, for instance, Lemma \ref{p:alpha} below), we consider 
the family of measures $\nu_x$ on $X$ indexed by $x\in X$ and defined by
\begin{align}\label{eq:nu_x}
d\nu_x(z):=\Theta(\mathsf{s}(x),\mathsf{s}(z))\,d\xi(z),\quad z\in X.
\end{align}
We note that while the measure $\xi$ depends on a choice of the section $\mathsf{s}$,
the measures $\nu_x$ are canonically defined. 

In the case when $G$ is a Lie group and
$H$ is a closed subgroup, the measures $m$ and $\rho$ are defined by smooth differential forms on $X$.
Since the section $\mathsf{s}$ can be chosen locally smooth,
it follows that the measures $\nu_x$ are absolutely continuous.

We note that the "skew balls" $H_t(g_1,g_2)$ and the limiting measures $\nu_x$ were originally introduced 
in \cite{GW}, in which further details can be found.

\subsection{Ergodic-theoretic terminology}

Let $(X,\xi)$ be a standard Borel space equipped with an action of an lcsc group $L$ 
which is denoted by $\pi_X$.
For a family of finite Borel 
measures $\vartheta_t$ on $L$, we denote by $\pi_X(\vartheta_t)$ the averaging operator
defined on measurable functions $\phi$ on $X$ by
\begin{equation}\label{eq:operator}
\pi_X(\vartheta_t)\phi(x)=\int_L \phi(x\cdot l)\,d\vartheta_t(l),\quad x\in X.
\end{equation}

\begin{defi}
{\rm 
Let $(\mathbb{B},\|\cdot\|_{\mathbb{B}})$ be a Banach space of measurable functions on $X$.
\begin{itemize} 
\item 
We say that the family of measures $\vartheta_t$ satisfies the {\it strong maximal inequality} in 
$\mathbb{B}$ with respect to a seminorm $\|\cdot \|$
if for every $\phi \in \mathbb{B}$, 
$$
\left\| \sup_{t\ge 1}
  |\pi_X(\vartheta_t)\phi|\right\|\ll \|\phi\|_{\mathbb{B}}.
$$
\item We say that the family of  measures $\vartheta_t$ satisfies the
{\it mean ergodic theorem} in $\mathbb{B}$
if for every $\phi\in \mathbb{B}$, the sequence
$\pi_X(\vartheta_t)\phi$ converges in the norm of $\mathbb{B}$ as $t\to \infty$.

\item We say that the family of measures $\vartheta_t$ satisfies the {\it quantitative mean ergodic theorem} in
$\mathbb{B}$ with respect to a seminorm $\|\cdot \|$ on $\mathbb{B}$ 
if for every $\phi\in \mathbb{B}$, the limit of $\pi_X(\vartheta_t)\phi$ as $t\to\infty$ 
exists in the norm of $\mathbb{B}$ and
$$
\left\|
  \pi_X(\vartheta_t)\phi - \left(\lim_{t\to\infty}\pi_X(\vartheta_t)\phi\right) \right\|\le E(t)\, \|\phi\|_{\mathbb{B}}
$$
where $E(t)\to 0$ as $t\to\infty$.

\item We say that the family of  measures $\vartheta_t$ satisfies the
{\it pointwise ergodic theorem} in $\mathbb{B}$
if for every $\phi\in \mathbb{B}$, the sequence
$\pi_X(\vartheta_t)\phi(x)$ converges as $t\to \infty$ for almost every $x\in X$.

\end{itemize}
}
\end{defi}

We note that our measures $\theta_t$ will be either atomic measures on the finite sets $\Gamma_t$, or absolutely continuous bounded Borel measures on compact subsets of the   lcsc group $H$.
Therefore measurability of the maximal functions when the supremum is taken over all $t\in \RR_+$  follows from standard arguments. 

The space of particular interest  to us is the spaces $L^p(X)$ consisting of $L^p$-integrable
functions on $X$. However, when $(X,\xi)$ has infinite measure, 
it is natural to consider a filtration of $X$ by domains $D$
with finite measure and the spaces $L^p(D)$ of $L^p$-integrable functions
with support contained in $D$. Moreover, when $X$ has a manifold structure,
we also consider the spaces $L_l^p(D)$ of Sobolev functions introduced in Section \ref{sec:sobolev}.

Given a space of functions $\mathbb{B}$, we denote $\mathbb{B}^+$ the cone 
in $\mathbb{B}$ consisting of nonnegative functions.

\subsection{Orbit sampling operators}
We conclude this section by introducing the operators that will be in the centre of our
discussion. Here and throughout the paper 
we use the notation introduced in Section \ref{sec:space}. 

We consider the normalized orbit sampling operators, defined for a function $\phi:X\to\mathbb{R}$ and $x\in X=H \setminus G$ by 
\begin{equation}\label{eq:lambda}
\pi_{X}(\lambda_t)(\phi)(x):=\frac{1}{\rho(H_t)}\sum_{\gamma\in\Gamma_t} \phi(x\cdot\gamma).
\end{equation}
As noted in the discussion of the duality principle in the Introduction, our plan is to deduce the asymptotic behavior of these sampling operators
from the asymptotic behavior of the averaging operators on $Y=G/\Gamma$ defined by
\begin{equation}\label{eq:b_beta}
\pi_{Y}(\beta^{g_1,g_2}_t)(F)(y):=\frac{1}{\rho(H_t[g_1,g_2])}\int_{H_t[g_1,g_2]}
F(h^{-1}\cdot y)\, d\rho(h).
\end{equation}
for $g_1,g_2\in G$, a function $F:Y\to \mathbb{R}$, and $y\in Y$.

To simplify notation, we also set $\beta_t=\beta_t^{e,e}$.

\section{The basic norm bounds and the strong maximal inequality}\label{sec:max_ineq}

\subsection{Coarse admissibility}\label{sec:CA}
The goal of this section is prove the strong maximal inequality for the operators
$\pi_{X}(\lambda_t)$ defined in (\ref{eq:lambda}), based on the validity of a maximal inequality for the operators defined in (\ref{eq:b_beta}). 
We shall use notation from Section \ref{sec:notation} and assume that the sets $G_t$ satisfy
the following additional regularity properties (see the definition of coarse admissibility in  \cite{gn}):
\begin{enumerate}
\item[(CA1)] for every bounded $\Omega\subset G$, there exists $c>0$ such that
$$
\Omega\cdot G_t \cdot \Omega\subset G_{t+c}
$$
for all $t\ge t_0$.
\item[(CA2)] for every $c>0$,
$$
\sup_{t\ge t_0} \frac{\rho (H_{t+c})}{\rho (H_{t})}<\infty.
$$
\end{enumerate}

\begin{thm}\label{th_max_ineq}
Assume that properties {\rm (CA1)} and  {\rm (CA2)} hold. Then
\begin{enumerate}
\item[(i)] For every $1\le p\le \infty$, compact domain $D\subset X$ and $\phi\in L^p(D)$,
$$
\left\| \pi_X(\lambda_t)\phi\right\|_{L^p(D)}\ll_{p,D}\, \|\phi\|_{L^p(D)}.
$$
\item[(ii)]
Let $1\le p\le \infty$ and $l\in\mathbb{Z}_{\ge 0}$. Assume that 
for every compact domain $B\subset Y$, the averages $\beta_t$ supprted on $H_t$ satisfy
the strong maximal inequality in $L_l^p(B)^+$ with respect to $\|\cdot \|_{L^p(B)}$,
namely, for every $F\in L_l^p(B)^+$, 
$$
\left\| \sup_{t\ge t_0} \pi_Y(\beta_t)F\right\|_{L^p(B)}\ll_{p,l,B}\,\|F\|_{L^p_l(B)}.
$$
Then for every compact domain $D\subset X$,
the family of measures $\lambda_t$ satisfies the strong maximal inequality
in $L_l^p(D)^+$ with respect to $\|\cdot \|_{L^p(D)}$,
namely, for every $\phi\in L_l^p(D)^+$, 
$$
\left\| \sup_{t\ge t_0} \pi_X(\lambda_t)\phi\right\|_{L^p(D)}\ll_{p,l,D}\,\|\phi\|_{L^p_l(D)}.
$$
\end{enumerate}
\end{thm}

We start by establishing a relation between $L^p$-norms 
on the spaces $G$ and $Y=G/\Gamma$.

\begin{lem}\label{lem:norm_est1}
Let $1\le p\le \infty$ and $\Omega$ be a compact domain of $G$. Then
for every $F\in L^{p}(Y)$,
$$
\|F\circ \mathsf{p}_Y\|_{L^{p}(\Omega)}\ll_{p,\Omega}\,\|F\|_{L^p(\mathsf{p}_Y(\Omega))}.
$$
\end{lem}
 
\begin{proof}
We consider a measurable partition $\Omega=\sqcup_{i=1}^n \Omega_i$ 
such that the map $\mathsf{p}_Y$ is one-to-one on each $\Omega_i$. 
Then by the triangle inequality and the definition of the measure on $Y$,
\begin{align*}
\|F\circ \mathsf{p}_Y\|_{L^{p}(\Omega)}&\le \sum_{i=1}^n\|(F\circ \mathsf{p}_Y)\chi_{\Omega_i}\|_{L^{p}(G)}
=\sum_{i=1}^n\|F\chi_{\mathsf{p}_Y(\Omega_i)}\|_{L^{p}(Y)}
\le n \|F\|_{L^{p}(\mathsf{p}_Y(\Omega))}.
\end{align*}
This proves the lemma.

\end{proof}

Given functions $\phi\in L_l^p(D)$ with $D\subset X$ and $\chi\in C_c^l(H)$ (the space of continuous functions with compact support and $l$ continuous derivatives),
we introduce functions $f:G\to\mathbb{R}$ and $F:Y\to\mathbb{R}$ defined by
\begin{align}\label{eq:funcion}
f(g):=\chi(\mathsf{h}(g))\phi(\mathsf{p}_X(g))\quad\hbox{and}\quad
F(g\Gamma):=\sum_{\gamma\in\Gamma} f(g\gamma).
\end{align}
We note that if $\phi \in L^1(D)$, then 
it follows from (\ref{eq:measure}) that $f\in L^1(G)$, and, hence, $F\in L^1(Y)$.
The following lemma gives a more general estimate.

\begin{lem}\label{lem:norm_est2}
For every $1\le p\le \infty$, $l\in\mathbb{N}_{\ge 0}$, and compact  domain $D\subset X$,
\begin{equation}\label{eq:norm_est2}
\|F\|_{L_l^{p}(Y)}\ll_{p,l,D,\chi}\, \|\phi\|_{L_l^{p}(D)}.
\end{equation}
\end{lem}

\begin{proof}
We first consider the case when $l=0$.
It follows from (\ref{eq:r_h}) that 
$$
\supp(f)\subset \Omega:=\supp(\chi)^{-1} \mathsf{s}(D).
$$
Recall that we choose the section $\mathsf{s}$ to be bounded on $D$,
so that $\Omega$ is bounded. 
Since 
\begin{equation}\label{eq:f}
f(h\cdot \mathsf{s}(x))=\chi(h^{-1})\phi(x)\quad\hbox{for $(h,x)\in H\times X$,}
\end{equation}
it follows from (\ref{eq:measure}) that
\begin{equation}\label{eq:est_1}
\|f\|_{L^p(\Omega)}\ll_\chi \|\phi\|_{L^{p}(D)}.
\end{equation}
Let  $\Omega=\sqcup_{i=1}^n \Omega_i$ be a measurable partition
such that the map $\mathsf{p}_Y$ is one-to-one on each $\Omega_i$. 
We set $f_i=f\chi_{\Omega_i}$ and $F_i(g)=\sum_{\gamma\in\Gamma} f_i(g\gamma)$,
where the sum is finite because $f_i$ has compact support.
Then 
\begin{equation}\label{eq:est_2}
\|F_i\|_{L^p(Y)}=\|f_i\|_{L^p(\Omega)}\le \|f\|_{L^p(\Omega)}.
\end{equation}
Since $F=\sum_{i=1}^nF_i$, estimate (\ref{eq:norm_est2}) now follows
from (\ref{eq:est_1}) and (\ref{eq:est_2}).

Now let $l>0$. In this case, $G$ and $H$ are assumed to be Lie groups, and every point in $X$ has an
open neighbourhood $U$ with a smooth section $\mathsf{s}:U\to G$ of the factor map $\mathsf{p}_X$.
Using a partition of unity, we reduce the proof to the case when $\supp(\phi)$
is contained in one of these neighbourhoods $U$. Since the map
$$
(\mathsf{h},\mathsf{p}_X): \mathsf{s}^{-1} (U)\to H\times U
$$
defines a diffeomorphism on its domain, it follows from equation (\ref{eq:f}) that 
\begin{equation}\label{eq:est_1_1}
\|f\|_{L_l^p(\Omega)}\ll_\chi \|\phi\|_{L_l^{p}(D)}.
\end{equation}
Let
$\{\psi_i\}_{j=1}^{n}$ be  a partition of unity on $G$
such that $\sum_{i=1}^{n}\psi_i=1$ on $\Omega$,
and for every $i=1,\ldots n$, the map $\mathsf{p}_Y$ is a diffeomorphism on $\supp(\psi_i)$.
Then for functions
$$
F_i(g):=\sum_{\gamma\in\Gamma} (f\psi_i)(g\gamma),
$$
we have
\begin{equation}\label{eq:est_2_1}
\|F_i\|_{L_l^{p}(Y)}=\|f\psi_i\|_{L_l^{p}(\Omega)}\ll \|f\|_{L_l^{p}(\Omega)} \|\psi_i\|_{C^l}.
\end{equation}
Since $F=\sum_{i=1}^n F_i$, we deduce that
$$
\|F\|_{L_l^{p}(Y)}\le \sum_{i=1}^n \|F_i\|_{L_l^{p}(Y)},
$$
and the claim follows from (\ref{eq:est_1_1}) and (\ref{eq:est_2_1}).
\end{proof}
 
\subsection{Coarse geometric comparison argument}

\begin{proof}[Proof of Theorem \ref{th_max_ineq}]
Writing $\phi=\phi^+-\phi^-$ with $\phi^+=\max(\phi,0)$
and $\phi^-=\max(-\phi,0)$, we reduce 
the proof of the first part of the theorem to the case when $\phi\in L^p(D)^+$.

Let $\phi\in L_l^p(D)^+$ and $\chi\in C_c^l(H)$ be  a nonnegative function such that $\int_H \chi\,
d\rho=1$. We define functions $f:G\to\mathbb{R}$ and $F:Y\to\mathbb{R}$
as in (\ref{eq:funcion}). We note that by equation (\ref{eq:f}) we have 
$$
\supp(f)\subset \Omega:=\supp(\chi)^{-1}\mathsf{s}(D)\quad\hbox{and}\quad
\supp(F)\subset B:=\mathsf{p}_Y(\Omega).
$$
Since the section $\mathsf{s}$ is chosen to be bounded on $D$, it follows that $\Omega$ and $B$
are bounded. By \eqref{eq:measure},
\begin{align}\label{eq:mmm}
\left\|   \pi_X(\lambda_t)\phi\right\|_{L^p(D)} 
\ll_D\left\| 
  \pi_X(\lambda_t)(\phi)\circ\mathsf{p}_X\right\|_{L^p(\Omega)}, 
\end{align}
and 
\begin{align}\label{eq:mmmm}
\left\| \sup_{t\ge t_0}
  \pi_X(\lambda_t)\phi\right\|_{L^p(D)} 
\ll_D \left\| \sup_{t\ge t_0}
  \pi_X(\lambda_t)(\phi)\circ\mathsf{p}_X\right\|_{L^p(\Omega)} 
\end{align}

We claim that there exists $c>0$ (depending only on $D$ and $\chi$)
such that for every $g\in \Omega$ and $t\ge t_0$,
\begin{equation}\label{eq:max1}
\sum_{\gamma\in\Gamma_t}\phi(\mathsf{p}_X(g\gamma))\le
\int_{H_{t+c}}
F(\mathsf{p}_Y(h^{-1}g))\,d\rho(h)=\int_{H_{t+c}} F(h^{-1}g\Gamma)d\rho(h).
\end{equation}
Indeed, for $g\in \Omega$ and 
$\gamma\in \Gamma_t$ such that $\mathsf{p}_X (g\gamma)\in \hbox{supp}(\phi)\subset D$, we have
by (\ref{eq:r_h}),
$$
\mathsf{h}(g\gamma)^{-1} \mathsf{s} (\mathsf{p}_X(g\gamma))=g\gamma\in \Omega G_t.
$$
Hence, by (CA1),  there exists $c>0$ such that
$$
\mathsf{h}(g\gamma)^{-1}\supp(\chi)\subset \Omega G_t\mathsf{s}(D)^{-1}\supp(\chi)\subset G_{t+c},
$$
and thus also 
$$
\supp(\chi) \subset \mathsf{h}(g\gamma) H_{t+c}.
$$
Then we conclude using (\ref{eq:invariance}) and (\ref{eq:funcion}) that 
\begin{align*}
\phi(\mathsf{p}_X(g\gamma)) & = \phi(\mathsf{p}_X(g\gamma))
\int_{\mathsf{h}(g\gamma)H_{t+c}} \chi(h)\, d\rho(h)
= \phi(\mathsf{p}_X(g\gamma)) \int_{H_{t+c}} \chi(\mathsf{h}(g\gamma)h) \, d\rho(h)\\
&= \int_{H_{t+c}} \chi(\mathsf{h}(h^{-1}g\gamma)) \phi(\mathsf{p}_X(h^{-1}g\gamma)) \, d\rho(h)
 =  \int_{H_{t+c}} f(h^{-1}g\gamma)\, d\rho(h).
\end{align*}
This implies that
$$
\sum_{\gamma\in\Gamma_t}\phi(\mathsf{p}_X(g\gamma))=
\sum_{\gamma\in\Gamma_t}\int_{H_{t+c}}
f(h^{-1}g\gamma)\,d\rho(h),
$$
and since $f$ is nonnegative, by \ref{eq:funcion} we have 
$$
\sum_{\gamma\in\Gamma_t}\phi(\mathsf{p}_X(g\gamma))\le 
\sum_{\gamma\in\Gamma}\int_{H_{t+c}}
f(h^{-1}g\gamma)\,d\rho(h)=
$$
$$=\int_{H_{t+c}}
F(\mathsf{p}_Y(h^{-1}g))\,d\rho(h)=\int_{H_{t+c}}F(h^{-1}g\Gamma)d\rho(h).
$$
This proves (\ref{eq:max1}). 

Now it follows from (\ref{eq:max1}) and (CA2), using also (\ref{eq:lambda}) and (\ref{eq:b_beta}) that
on $\Omega$,
\begin{align}\label{eq:max_last}
\pi_X(\lambda_t)(\phi)\circ\mathsf{p}_X
\le \left(\sup_{t\ge t_0} \frac{\rho(H_{t+c})}{\rho(H_{t})}\right)
 \pi_Y(\beta_{t+c})F\circ \mathsf{p}_Y
\ll \pi_Y(\beta_{t+c})F \circ \mathsf{p}_Y.
\end{align}
Hence, by (\ref{eq:mmm}), Lemma \ref{lem:norm_est1},  Jensen's inequality, and Lemma \ref{lem:norm_est2},
\begin{align*}
\left\|\pi_X(\lambda_t)\phi \right\|_{L^p(D)} 
&\ll  \left\| \pi_Y(\beta_{t+c})(F)\circ\mathsf{p}_Y \right\|_{L^p(\Omega)}
 \ll \left\| \pi_Y(\beta_{t+c})F \right\|_{L^{p}(B)}\\
&\le \left\| F \right\|_{L^{p}(B)}\ll \|\phi\|_{L^p(D)}.
\end{align*}
This proves the first part of the theorem.

To prove the second part of the theorem,
we observe that by Lemma \ref{lem:norm_est2}, we have $F\in L_l^p(B)^+$,
and by the strong maximal inequality for the family $\beta_t$ in $L_l^p(B)^+$ with respect to $\|\cdot \|_{L^p(B)}$,
the function $\sup_{t\ge t_0} \pi_Y(\beta_{t+c})F$ is in $L^p(B)$.
Hence, by (\ref{eq:mmmm}), (\ref{eq:max_last}), and Lemma \ref{lem:norm_est1},
\begin{align*}
\left\| \sup_{t\ge t_0}
  \pi_X(\lambda_t)\phi\right\|_{L^p(D)} 
\ll  \left\|\sup_{t\ge t_0} \pi_Y(\beta_{t+c})(F)\circ\mathsf{p}_Y
\right\|_{L^{p}(\Omega)}
 \ll \left\|\sup_{t\ge t_0} \pi_Y(\beta_{t+c})F
\right\|_{L^{p}(B)}.
\end{align*}
Now the second part of the theorem
follows from the strong maximal inequality in $L_l^p(B)^+$ combined with Lemma \ref{lem:norm_est2}.
\end{proof}

\section{The mean ergodic theorem}\label{sec:mean}
\subsection{Average admissibility of the restricted sets}
The goal of this section is prove the mean ergodic theorem  for the averages
$\pi_{X}(\lambda_t)$ defined in (\ref{eq:lambda}), based on the mean ergodic theorems for the averages $\pi_Y(\beta_t)$ defined in (\ref{eq:b_beta}).
We shall use notation from Section \ref{sec:notation} and assume that the sets $G_t$ satisfy additional regularity properties, as follows. Note that properties (A2) and (A2') depend on a parameter $r\in (1,\infty)$, but for simplicity we suppress this dependence in the notation. 
\begin{enumerate}
\item[(A1)] For every $\vre\in (0,1)$, there exists a neighbourhood $O_\vre$
of identity in $G$ such that 
$$
O_\vre\cdot G_t\cdot O_\vre\subset G_{t+\vre}
$$
for all $t\ge t_0$.

\item[(A2)] 
For every $u\in G$, and for every compact $\Omega\subset G$, there exists $\omega(\vre)>0$
that converges to $0$ as $\vre\to 0^+$ such that 
$$
\left(\int_{\Omega} (\rho(H_{t+\vre}[u,v])- \rho(H_t[u,v]))^r\,dm(v)\right)^{1/r}\le  \omega(\vre) \rho(H_t)
$$ 
for all $t\ge t_0$ and $\vre\in (0,1)$.

\item[(A2$^\prime$)] 
For every compact $\Omega\subset G$, there exists $\omega(\vre)>0$
that converges to $0$ as $\vre\to 0^+$ such that 
$$
\left(\int_{\Omega\times \Omega} (\rho(H_{t+\vre}[u,v])- \rho(H_t[u,v]))^r\,dm(u)dm(v)\right)^{1/r}\le  \omega(\vre) \rho(H_t).
$$ 
for all $t\ge t_0$ and $\vre\in (0,1)$.

\item[(A3)] For every $g_1,g_2\in G$, the limit
$$
\Theta(g_1,g_2):=\lim_{t\to\infty} \frac{\rho(H_t[g_1,g_2])}{\rho(H_t)} 
$$
exists.

\item[(S)] every $x\in X$ has a neighborhood $U$ such that there exists a continuous
section $\mathsf{s}:U\to G$ of the factor map $\mathsf{p}_X$.
\end{enumerate}

Property (A2$^\prime$) is  used only in the proof of the mean ergodic theorem 
(Theorem \ref{th:mean}), and property (A2) is  used only in the proof of the pointwise ergodic theorem
(Theorem \ref{th:pointwise}).

\begin{thm}\label{th:mean}
Let $1\le p< \infty$.
Assume that {\rm (CA1), (CA2), (A1), (A2$^\prime$)} with some $r\ge p$, {\rm (A3), (S)} hold,
and for every $g_1,g_2\in G$ and every compact $B\subset Y$,
the family $\beta_t^{g_1,g_2}$ satisfies the mean ergodic theorem 
in $L^p(B)$, namely, for every $F\in L^p(B)$,
$$
\left\| \pi_Y(\beta_t^{g_1,g_2})F-\int_Y F\,
  d\mu\right\|_{L^p(B)}\to 0\quad\hbox{as $t\to\infty$.}
$$
Then for every compact domain $D\subset X$, 
the family $\lambda_t$ satisfies the mean ergodic theorem
in $L^p(D)$, namely,
for every $\phi\in L^p(D)$ and for almost every $x\in D$
\begin{equation}\label{eq:mean_0}
\left\| \pi_X(\lambda_t)\phi(x)-\int_X \phi\,
  d\nu_x\right\|_{L^p(D)}\to 0\quad\hbox{as $t\to\infty$,}
\end{equation}
where $\nu_x$ is the measure defined in (\ref{eq:nu_x}).
\end{thm}

Besides the discrete averages $\pi_X(\lambda_t)$ supported on $\Gamma$, it will be also convenient
to consider their continuous analogues which are defined by
\begin{equation}\label{lambdaG}
\pi_X(\lambda^G_t)\phi(x):=\frac{1}{\rho(H_t)} \int_{G_t}\phi(xg)dm(g),\quad \phi\in L^p(D),
\end{equation}
for every $x\in X$.
It would be convenient in the proof to use averages without normalisation. Namely, we set
\begin{equation}\label{Lambda}
\pi_X(\Lambda_t)\phi(x):= \sum_{\gamma\in \Gamma\cap G_t}\phi(x\gamma)\quad\hbox{and}\quad
\pi_X(\Lambda^G_t)\phi(x):= \int_{G_t}\phi(xg)dm(g).
\end{equation}
We start the proof of Theroem \ref{th:mean} with a lemma:

\begin{lem}\label{c:reg}
 Let $D$ be a compact domain in $X$, $1\le p<q\le \infty$, and $\phi\in L^{q}(D)$.
\begin{enumerate}
\item[(i)] 
Suppose that {\rm (A2)} holds with $r=q/(q-1)$ (here $r=1$ if $q=\infty$).
Then for every $x\in X$,
the estimate 
$$
\left|\pi_X(\Lambda_{t+\vre}^G)\phi(x)-\pi_X(\Lambda_t^G)\phi(x)\right|
\ll_{q,x,D}\, \omega(\vre) \rho(H_t)\|\phi\|_{L^q(D)}
$$
holds for all $t\ge t_0$ and $\vre\in (0,1)$.

\item[(ii)] 
Suppose that {\rm (A2$^\prime$)} holds where $r=pq/(q-p)$ (here $r=p$ if $q=\infty$).
Then the estimate 
$$
\left\|\pi_X(\Lambda_{t+\vre}^G)\phi-\pi_X(\Lambda_t^G)\phi\right\|_{L^p(D)}
\ll_{p,q,D}\, \omega(\vre) \rho(H_t)\|\phi\|_{L^q(D)}
$$
holds for all $t\ge t_0$ and $\vre\in (0,1)$.
\end{enumerate}
\end{lem}

\begin{proof}
To prove (i), we are required to estimate
$$
\left|\int_{G_{t+\vre}} \phi(x g)\, dm(g)-\int_{G_t} \phi(x g)\, dm(g)\right|
=\left|\int_{G_{t+\vre}-G_t} \phi(x g)\, dm(g)\right|.
$$
It follows from invariance of $m$, the equivariance of $\mathsf{p}_X$ and (\ref{eq:measure}) that
\begin{align}\label{eq:long}
\int_{G_{t+\vre}-G_t} \phi(x g)\, dm(g)
=&
\int_{G_{t+\vre}-G_t} \phi(\mathsf{p}_X(\mathsf{s}(x)g))\, dm(g)\\
=&\int_{\mathsf{s}(x)(G_{t+\vre}-G_t)} \phi(\mathsf{p}_X(g))\, dm(g)\nonumber \\
=&\int_{(y,h):\, \mathsf{s}(x)^{-1}h\mathsf{s}(y)\in G_{t+\vre}-G_t} \phi(\mathsf{p}_X(h\mathsf{s}(y)))\, d\rho(h)d\xi(y)\nonumber\\
=&\int_X\phi(y)\left(\rho(H_{t+\vre}[\mathsf{s}(x),\mathsf{s}(y)])-
\rho(H_{t}[\mathsf{s}(x),\mathsf{s}(y)])\right)\,d\xi(y),\nonumber
\end{align}
Hence, by H\"older's  inequality,
\begin{align*}
&\left|\int_{G_{t+\vre}-G_t} \phi(x g)\, dm(g)\right|\\
\le & 
\|\phi\|_{L^q(D)}
\left(\int_{D}\left(\rho(H_{t+\vre}[\mathsf{s}(x),\mathsf{s}(y)])-
\rho(H_{t}[\mathsf{s}(x),\mathsf{s}(y)])\right)^{r}\,d\xi(y)\right)^{1/r},
\end{align*}
where $r=q/(q-1)$ is the exponent conjugate to $q$.
We pick a compact set $O$ of  $H$ with positive measure and put $\Omega=O\mathsf{s}(D)$.
Then it follows from (\ref{eq:measure}) and (A2) that since $\mathsf{s}(x)$ and $\mathsf{s}(y)$ vary in a compact set 
\begin{align*}
&\int_{D}\left(\rho(H_{t+\vre}[\mathsf{s}(x),\mathsf{s}(y)])-
\rho(H_{t}[\mathsf{s}(x),\mathsf{s}(y)])\right)^{r}\,d\xi(y)\\
\ll_O&\;\; 
\int_{\Omega}\left(\rho(H_{t+\vre}[\mathsf{s}(x),v])-
\rho(H_{t}[\mathsf{s}(x),v])\right)^{r}\,dm(v)\\
\ll &\;\;\; \omega(\vre)^r \rho(H_t)^{r}.
\end{align*}
This implies the first part of the lemma.

To prove the second part, we are required to estimate
\begin{align*}
&\int_D\left|\int_{G_{t+\vre}-G_t} \phi(x g)\, dm(g)\right|^p\, d\xi(x)\\
=&\int_D \left|\int_X\phi(y)\left(\rho(H_{t+\vre}[\mathsf{s}(x),\mathsf{s}(y)])-
\rho(H_{t}[\mathsf{s}(x),\mathsf{s}(y)])\right)\,d\xi(y)\right|^p\,d\xi(x)
\end{align*}
Since $\supp(\phi)\subset D$, it follows from (\ref{eq:long}) and H\"older's inequality that
\begin{align*}
&\left|\int_{G_{t+\vre}-G_t} \phi(x g)\, dm(g)\right|^p\\
\le & \xi(D)^{p(1-1/p)}
\int_D|\phi(y)|^p\left(\rho(H_{t+\vre}[\mathsf{s}(x),\mathsf{s}(y)])-
\rho(H_{t}[\mathsf{s}(x),\mathsf{s}(y)])\right)^p\,d\xi(y).
\end{align*}
Moreover, integrating over $x\in D$ w.r.t. $\xi$ and applying H\"older's inequality one more time, we deduce that
\begin{align*}
&\int_{D\times D}|\phi(y)|^p\left(\rho(H_{t+\vre}[\mathsf{s}(x),\mathsf{s}(y)])-
\rho(H_{t}[\mathsf{s}(x),\mathsf{s}(y)])\right)^p\,d\xi(x)d\xi(y)\\
\le & \xi(D)^{p/q}\|\phi\|^p_{L^q(D)}
\left(\int_{D\times D}\left(\rho(H_{t+\vre}[\mathsf{s}(x),\mathsf{s}(y)])-
\rho(H_{t}[\mathsf{s}(x),\mathsf{s}(y)])\right)^{ps}\,d\xi(x)d\xi(y)\right)^{1/s},
\end{align*}
where $s=q/(q-p)$ is the exponent conjugate to $q/p$ (here $s=1$ if $q=\infty$).
Hence, we conclude that
\begin{align*}
& \int_D\left|\int_{G_{t+\vre}-G_t} \phi(x g)\, dm(g)\right|^p\, d\xi(x)\\
\ll&
\|\phi\|^p_{L^q(D)}
\left(\int_{D\times D}\left(\rho(H_{t+\vre}[\mathsf{s}(x),\mathsf{s}(y)])-
\rho(H_{t}[\mathsf{s}(x),\mathsf{s}(y)])\right)^{ps}\,d\xi(x)d\xi(y)\right)^{1/s}.
\end{align*}
As in the first part of the argument,
we pick a compact set $O$ of $H$ with positive measure and set $\Omega=O\mathsf{s}(D)$.
Then it follows from (\ref{eq:measure}) and (A2$^\prime$) that
\begin{align*}
&\int_{D\times D}\left(\rho(H_{t+\vre}[\mathsf{s}(x),\mathsf{s}(y)])-
\rho(H_{t}[\mathsf{s}(x),\mathsf{s}(y)])\right)^{ps}\,d\xi(x)d\xi(y)\\
\ll_O&\;\;
\int_{\Omega\times \Omega}\left(\rho(H_{t+\vre}[u,v])-
\rho(H_{t}[u,v])\right)^{ps}\,dm(u)dm(v)\\
\ll &\;\;\; \omega(\vre)^{ps} \rho(H_t)^{ps}.
\end{align*}
This implies the required estimate.
\end{proof}

The proof of Theorem \ref{th:mean} now continues with 
\subsection{Geometric comparison argument}

We first observe that for every $\phi,\psi\in L^p(D)$ and $x\in X$ 
we have 
\begin{align*}
\left\| \pi_X(\lambda_t)\phi-\pi_X(\lambda_t)\psi \right\|_{L^p(D)}&\ll 
\left\| \phi-\psi \right\|_{L^p(D)},\\
\left\|\int_X \phi\, d\nu_x- \int_X \psi\, d\nu_x \right\|_{L^p(D)}
&\ll \left\| \phi-\psi \right\|_{L^p(D)}.
\end{align*}
The first estimate is proved in Theorem \ref{th_max_ineq}(i). 
To prove the second estimate we observe 
that it follows from (CA1)--(CA2) that the density $\Theta(x,\cdot)$ of
the measure $\nu_x$ is uniformly bounded on $D$, with the bound uniform as $x$ varies in compact sets in $G$.
Hence, since $\xi(D)<\infty$, the second estimate follows from H\"older's inequality, by definition of $\nu_x$.

The above  estimates imply that it is sufficient to verify (\ref{eq:mean_0})
for a dense family of functions in $L^p(D)$, and we shall prove
that (\ref{eq:mean_0}) holds for $\phi \in L^q(D)$ with $q>p$ such that $r=pq/(q-p)$.
Since every such $\phi$
can be written as $\phi=\phi^+-\phi^-$ with $\phi^+\ge 0$ and $\phi^-\ge 0$
are in $L^q(D)$, the proof reduces to the case when $\phi\ge 0$.
Moreover, because condition (S) is satisfied, decomposing $\phi$ as a finite sum of functions with small supports,
we reduce the proof to the situation when there exists a section $\mathsf{s}:X\to G$
of the factor map $\mathsf{p}_X:G\to H\backslash G=X$ such that $\mathsf{s}|_D$ is continuous.

Using \eqref{eq:measure}, we deduce that for every $x\in X$, we have
\begin{align*}
\pi_X(\lambda^G_t)\phi(x) 
&=
\frac{1}{\rho(H_t)}\int_{G_t} \phi(\mathsf{p}_X(\mathsf{s}(x)g))\, dm(g)\\
&=\frac{1}{\rho(H_t)}\int_{(y,h):\, \mathsf{s}(x)^{-1}h\mathsf{s}(y)\in G_t} \phi(\mathsf{p}_X(h\mathsf{s}(y)))\, d\rho(h)d\xi(y)\\
&=\int_D\phi(y)\frac{\rho(H_t[\mathsf{s}(x),\mathsf{s}(y)])}{\rho(H_t)}\,d\xi(y).
\end{align*}
By (CA1)--(CA2), 
$\frac{\rho(H_t[\mathsf{s}(x),\mathsf{s}(y)])}{\rho(H_t)}$
is bounded uniformly when $x$ is fixed and $y$ varies on $ D$. Hence, it follows from (A3) and
the dominated convergence theorem (since $\phi$ is certainly in $L^1(D)$) , that for every $x\in X$,
\begin{equation}\label{eq:point}
\pi_X(\lambda^G_t)\phi(x) \to \int_D \phi\, d\nu_x =\int_X \phi\, d\nu_x\quad\hbox{as $t\to\infty$}.
\end{equation}

To obtain convergence in $L^p(D)$, consider the difference 
$$\abs{\pi_X(\lambda^G_t)\phi(x) - \int_D \phi\, d\nu_x}^p= \abs{\int_D\phi(y)\left(\frac{\rho(H_t[\mathsf{s}(x),\mathsf{s}(y)])}{\rho(H_t)}-\Theta(\mathsf{s}(x),\mathsf{s}(y))\right)\,d\xi(y)}^p$$
Therefore, applying the dominated convergence theorem one more time, and the fact that $\phi\in L^p(D)$,  we deduce by integrating over  $x\in D$ that 
the convergence also holds in $L^p(D)$.

To conclude the proof of Theorem \ref{th:mean}, it remains to show that
\begin{equation}\label{eq_main0}
\left\| \pi_X(\lambda_t)\phi-\pi_X(\lambda^G_t)\phi\right\|_{L^p(D)}\to 0\quad\hbox{as $t\to\infty$}
\end{equation}
for every $\phi$ in the dense subset of $L^p(D)^+$ we chose, namely every $\phi\in L^q(D)^+$.
This calls for a comparison between the discrete average supported on $x\cdot \Gamma_t$, and the continuous averages supported on $x\cdot G_t$. We will therefore apply local analysis to effect this comparison. 

Let $\vre\in (0,1/42)$ and $O$ be a compact symmetric neighborhood of identity in $G$ such that
\begin{align}
O\cdot G_t\cdot g^{-1} O g&\subset
G_{t+\vre}\quad\hbox{for every $t\ge t_0$ and $g\in\mathsf{s}(D)$,}\label{eq:pert}\\
G_t\cdot \mathsf{s}(y O)^{-1} &\subset G_{t+\vre}\cdot \mathsf{s}(y)^{-1}\quad\hbox{for 
every $t\ge t_0$ and $y\in D$}.\label{eq:pert2}
\end{align}
Such a neighbourhood exists by (A1) and continuity of $\mathsf{s}$ on $D$.
Let $\chi\in C_c(H)^+$ with $\supp(\chi)\subset O\cap H$ be 
normalized so that $\int_H \chi\, d\rho=1$, and let
$f:G\to\mathbb{R}$ be defined as in (\ref{eq:funcion}).

We claim that for $u\in G$,
\begin{equation}\label{eq_claim2_0}
\pi_X(\Lambda_t)\phi(\mathsf{p}_X(u))\le \sum_{\gamma\in\Gamma}\int_{H_{t+\vre}[u,\mathsf{r}(u\gamma)]} f(h^{-1}u\gamma)\,d\rho(h),
\end{equation}
and
\begin{equation}\label{eq_claim2a_0}
\pi_X(\Lambda_t)\phi(\mathsf{p}_X(u))\ge \sum_{\gamma\in\Gamma}\int_{H_{t-\varepsilon}[u,\mathsf{r}(u\gamma)]} f(h^{-1}u\gamma)\,d\rho(h).
\end{equation}
To establish (\ref{eq_claim2_0}), it suffices to consider $\gamma\in \Gamma_t$ satisfying $\mathsf{p}_X (u\gamma)\in \hbox{supp}(\phi)\subset D$, where we have
by (\ref{eq:r_h}),
$$
\mathsf{h}(u\gamma)^{-1}\mathsf{r} (u\gamma)=u\gamma\in uG_t,
$$
and by (\ref{eq:pert}),
$$
\mathsf{h}(u\gamma)^{-1}\mathsf{r} (u\gamma)\left(\mathsf{r} (u\gamma)^{-1}{O} \mathsf{r} (u\gamma)\right)=
\mathsf{h}(u\gamma)^{-1}{O} \mathsf{r} (u\gamma)\subset 
u G_t\cdot \mathsf{r} (u\gamma)^{-1} {O} \mathsf{r} (u\gamma)
\subset uG_{t+\vre}.
$$
Hence,
$$
\mathsf{h}(u\gamma)^{-1}{O} \subset uG_{t+\vre}\mathsf{r} (u\gamma)^{-1}
$$
and thus also 
$$
\supp(\chi) \subset \mathsf{h}(u\gamma) H_{t+\vre}[u,\mathsf{r} (u\gamma)].
$$
Therefore by invariance of $\rho$ and (\ref{eq:invariance}),
\[
\begin{split}
\phi(\mathsf{p}_X(u\gamma)) & = \phi(\mathsf{p}_X(u\gamma))
\int_{\mathsf{h}(u\gamma)H_{t+\vre}[u,\mathsf{r}(u\gamma)]} \chi(h)\, d\rho(h)\\ 
&= \phi(\mathsf{p}_X(u\gamma)) \int_{H_{t+\vre}[u,\mathsf{r}(u\gamma)]} \chi(\mathsf{h}(u\gamma)h) \, d\rho(h)\\
&= \int_{H_{t+\vre}[u,\mathsf{r}(u\gamma)]} \chi(\mathsf{h}(h^{-1}u\gamma)) \phi(\mathsf{p}_X(h^{-1}u\gamma)) \, d\rho(h)\\
& =  \int_{H_{t+\vre}[u,\mathsf{r}(u\gamma)]} f(h^{-1}u\gamma)\, d\rho(h),
\end{split}
\]
and we conclude that
$$
\pi_X(\Lambda_t)\phi(\mathsf{p}_X(u))\le \sum_{\gamma\in\Gamma_t}\int_{H_{t+\vre}[u,\mathsf{r}(u\gamma)]} f(h^{-1}u\gamma)\,d\rho(h),
$$
Since $f\ge 0$, we can sum over all $\gamma\in \Gamma$, and this implies (\ref{eq_claim2_0}).

To prove \eqref{eq_claim2a_0}, we observe that 
for $\gamma\in \Gamma-\Gamma_t$ such that $\mathsf{p}_X(u\gamma)\in\hbox{supp}(\phi)\subset D$, we have
\begin{equation}\label{eq:new1}
\supp(\chi)\cap \mathsf{h}(u\gamma) H_{t-\vre}[u, \mathsf{r} (u\gamma)] = \emptyset.
\end{equation}
Indeed, if $h\in \supp(\chi)\subset O$ belongs to this intersection, then
$$
u^{-1}\mathsf{h}(u\gamma)^{-1} h \mathsf{r} (u\gamma) \in G_{t-\vre},
$$
and by (\ref{eq:r_h}) and \eqref{eq:pert},
$$
\gamma=u^{-1}\mathsf{h}(u\gamma)^{-1}\mathsf{r} (u\gamma)
\in G_{t-\vre}\cdot \mathsf{r} (u\gamma)^{-1}  h^{-1} \mathsf{r} (u\gamma)\subset G_t,
$$
which gives a contradiction. Thus we can now deduce from \eqref{eq:new1} that
\begin{align*}
 &\sum_{\gamma\in\Gamma} \int_{H_{t-\vre}[u,\mathsf{r}(u\gamma)]} f(h^{-1}u\gamma)\,d \rho(h)\\
=&\sum_{\gamma\in\Gamma} \int_{H_{t-\vre}[u, \mathsf{r}(u\gamma)]} \chi (\mathsf{h}(h^{-1}u\gamma)) \phi(\mathsf{p}_X(hu\gamma))\,d \rho(h)\\
=& \sum_{\gamma\in\Gamma} \phi(\mathsf{p}_X(u\gamma)) \int_{\mathsf{h}(u\gamma) H_{t-\vre}[u,\mathsf{r}(u\gamma)]} \chi (h)\,d\rho(h) \\ 
=& \sum_{\gamma\in\Gamma_t} \phi(\mathsf{p}_X(u\gamma))
\int_{\mathsf{h}(u\gamma)H_{t-\vre}[u,\mathsf{r}(u\gamma)]} \chi\,d\rho
\le \sum_{\gamma\in\Gamma_t} \phi(\mathsf{p}_X(u\gamma))= \pi_X(\Lambda_t)\phi(\mathsf{p}_X(u)),
 \end{align*}
which proves (\ref{eq_claim2a_0}). 

A similar line of reasoning also yields
the analogous estimates for the averages $\pi_X(\Lambda^G_t)$. 
Namely, for $u\in G$ 
\begin{equation}\label{eq_claim2b_0}
\pi_X(\Lambda^G_t)\phi(\mathsf{p}_X(u))\le \int_G\int_{H_{t+\vre}[u,\mathsf{r}(ug)]} f(h^{-1}ug)\,d\rho(h) dm(g),
\end{equation}
and
\begin{equation}\label{eq_claim2a'_0}
\pi_X(\Lambda^G_t)\phi(\mathsf{p}_X(u))\ge \int_G\int_{H_{t-\vre}[u,\mathsf{r}(ug)]} f(h^{-1}ug)\,d\rho(h)dm(g).
\end{equation}

\subsection{Local analysis}

Let us now begin with the next step of the argument. 
Let $\Omega:=\supp(\chi)^{-1}\mathsf{s}(D)$ and $B=\mathsf{p}_Y(\Omega)$.
Then $\supp(f)\subset \Omega$, and 
there exists a finite subset $\{v_i\}_{i=1}^n\subset \Omega$ such that
$$
\Omega\subset \bigcup_{i=1}^n v_i O.
$$
Then it follows from \eqref{eq:pert2} that for every $g\in Hv_iO$,
\begin{equation}\label{eq:round}
H_{t-\vre}[u,\mathsf{r}(v_i)]\subseteq H_t[u,\mathsf{r}(g)]\subseteq H_{t+\vre}[u,\mathsf{r}(v_i)].
\end{equation}
We fix a finite measurable partition
$$
\supp(f)=\bigsqcup_{i=1}^n C_i\quad\hbox{such that $C_i\subset v_i O$.}
$$
Let 
$$
f_{i}=f\cdot\chi_{C_i}\quad\hbox{and}\quad F_i(\mathsf{p}_Y(g))=\sum_{\gamma\in\Gamma}f_{i}(g\gamma),
$$
where the sum is finite because $f_i$ has compact support.
We note that by the definition of $\mu$, we have
$\int_Y F_i\, d\mu=\int_X f_i\,dm$, and by Lemma \ref{lem:norm_est2},
$F_i\in L^q(B)$. Now we deduce from (\ref{eq:round}) that
\begin{align}\label{eq_step1_0}
\sum_{\gamma\in\Gamma} \int_{H_t[u,\mathsf{r}(u\gamma)]} f(h^{-1}u\gamma)\,d\rho(h)
&=\sum_{i=1}^n\sum_{\gamma\in\Gamma} \int_{H_t[u,\mathsf{r}(u\gamma)]} f_i(h^{-1}u\gamma)\,d\rho(h) \nonumber\\
&\le \sum_{i=1}^n\sum_{\gamma\in\Gamma} \int_{H_{t+\vre}[u,\mathsf{r}(v_i)]} f_{i}(h^{-1}u\gamma)\,d\rho(h)\nonumber\\
&=\sum_{i=1}^n\int_{H_{t+\vre}[u,\mathsf{r}(v_i)]} F_{i}(h^{-1}\mathsf{p}_Y(u))\,d\rho(h).
\end{align}
A similar argument also gives the lower estimate
\begin{align}\label{eq_step1a_0}
\sum_{\gamma\in\Gamma} \int_{H_t[u,\mathsf{r}(u\gamma)]} f(h^{-1}u\gamma)\,d\rho(h)
\ge \sum_{i=1}^n\int_{H_{t-\vre}[u,\mathsf{r}(v_i)]} F_{i}(h^{-1}\mathsf{p}_Y(u))\,d\rho(h).
\end{align}

There exists a finite subset $\{u_j\}_{j=1}^m\subset \Omega$ and a finite measurable partition
$$
\Omega=\bigsqcup_{j=1}^m \Omega_j\quad\hbox{such that $\Omega_j\subset u_jO$.}
$$
By \eqref{eq:pert}, for every $u\in u_j{O}$,
$$
H_{t-\vre}[u_j,\mathsf{r}(v_i)]\subseteq H_t[u,\mathsf{r}(v_i)]\subseteq H_{t+\vre}[u_j,\mathsf{r}(v_i)].
$$
We introduce a measurable function $\mathcal{A}_t:\Omega\to \mathbb{R}$ which is defined piecewise by
\begin{equation}\label{eq_g_0}
\mathcal{A}_t(u):=\sum_{i=1}^n\int_{H_{t}[u_j,\mathsf{r}(v_i)]} F_{i}(h^{-1}\mathsf{p}_Y(u))\,d\rho(h)\quad\hbox{for $u\in \Omega_{j}$.}
\end{equation}
Then for $u\in \Omega$,
$$
\mathcal{A}_{t-\vre}(u)\le\sum_{i=1}^n\int_{H_{t}[u,\mathsf{r}(v_i)]}
F_{i}(h^{-1}\mathsf{p}_Y(u))\,d\rho(h)\le \mathcal{A}_{t+\vre}(u),
$$
and combining (\ref{eq_claim2_0}), (\ref{eq_claim2a_0}),
(\ref{eq_step1_0}), (\ref{eq_step1a_0}), we deduce that for $u\in \Omega$,
\begin{equation}\label{eq_g_a_0}
\mathcal{A}_{t-3\vre}(u)\le \pi_X(\Lambda_t)\phi(\mathsf{p}_X(u)) \le \mathcal{A}_{t+3\vre}(u).
\end{equation}

\subsection{The duality argument}

We would now like to exploit the information regarding the ergodic behavior of the $H$-orbits in $G/\Gamma$, and our next task is to prove a similar estimate for the averages $\pi_X(\Lambda^G_t)$.
For $u\in \Omega_{j}$, we have by (\ref{eq:round}) and invariance of $m$,
\begin{align*}
\int_G\int_{H_{t}[u,\mathsf{r}(ug)]} f(h^{-1}ug)\,d\rho(h)dm(g) 
&\le \sum_{i=1}^n\int_G\int_{H_{t+\vre}[u,\mathsf{r}(v_i)]} f_i(h^{-1}ug)\,d\rho(h)dm(g) \\
&=\sum_{i=1}^n\rho(H_{t+\vre}[u,\mathsf{r}(v_i)])\int_{G} f_{i}\, dm\\
&\le \sum_{i=1}^n\rho(H_{t+2\vre}[u_j,\mathsf{r}(v_i)])\int_{G} f_{i}\, dm,
\end{align*}
and similarly,
\begin{align*}
\int_G\int_{H_{t}[u,\mathsf{r}(ug)]} f(h^{-1}ug)\,d\rho(h)dm(g)
\ge\sum_{i=1}^n\rho(H_{t-2\vre}[u_j,\mathsf{r}(v_i)])\int_{G} f_{i}\, dm.
\end{align*}
As in \eqref{eq_g_0}, we introduce a function $\mathcal{A}^G_t:\Omega\to \mathbb{R}$ defined by
\begin{equation}\label{eq_d_0}
\mathcal{A}^G_t(u)=\sum_{i=1}^n\rho(H_{t}[u_j,\mathsf{r}(v_i)])\int_{G} f_{i}\, dm\quad 
\hbox{for $u\in \Omega_{j}$.}
\end{equation}
Then it follows from the above estimates combined with \eqref{eq_claim2b_0} and \eqref{eq_claim2a'_0}
that for $u\in \Omega$,
\begin{equation}\label{eq_d_b_0}
\mathcal{A}^G_{t-3\vre}(u)\le \pi_X(\Lambda^G_t)\phi(\mathsf{p}_X(u)) \le 
\mathcal{A}^G_{t+3\vre}(u).
\end{equation}

By (\ref{eq:measure}) and Lemma \ref{c:reg}(ii) combined with (A2$^\prime$), we have
\begin{equation}\label{eq_d_est_0}
\left\|\pi_X(\Lambda^G_{t+\vre})\phi\circ\mathsf{p}_X - \pi_X(\Lambda^G_{t})\phi\circ\mathsf{p}_X \right\|_{L^p(\Omega)}
\le \omega(\vre) \rho(H_t)\|\phi\|_{L^q(D)},
\end{equation}
where $\omega(\vre)\to 0$ as $\vre\to 0^+$. Combining (\ref{eq_d_b_0}) and (\ref{eq_d_est_0}), we deduce
that
\begin{align}\label{eq_d_est_00}
\left\|
\mathcal{A}^G_{t+\vre}-\mathcal{A}^G_{t}
\right\|_{L^p(\Omega)} &\le
\left\|\pi_X(\Lambda^G_{t+4\vre})\phi\circ\mathsf{p}_X - \pi_X(\Lambda^G_{t-3\vre})\phi\circ\mathsf{p}_X
\right\|_{L^p(\Omega)}\\
&\ll \omega(7\vre) \rho(H_t)\|\phi\|_{L^q(D)}.\nonumber
\end{align}
Now we show that
\begin{equation}\label{eq_g_d_norm_0}
\|\mathcal{A}_t-\mathcal{A}^G_t\|_{L^p(\Omega)}=o(\rho(H_t))\quad\hbox{as $t\to\infty$.}
\end{equation}
This is  where we utilise the mean ergodic theorem for $H$ acting in $L^p(B)\subset L^p(G/\Gamma)$.
Using the triangle inequality, Lemma \ref{lem:norm_est1}, and (CA1)--(CA2), we obtain
\begin{align*}
&\|\mathcal{A}_t-\mathcal{A}^G_t\|_{L^p(\Omega)}\le \sum_{j=1}^m\|\mathcal{A}_t-\mathcal{A}^G_t\|_{L^p(\Omega_j)} \\
=& \sum_{j=1}^m \left\| \sum_{i=1}^n \int_{H_{t}[u_j,\mathsf{r}(v_i)]}
  F_{i}(h^{-1}\mathsf{p}_Y(g))\,d\rho(h) 
- \sum_{i=1}^n\rho(H_{t}[u_j,\mathsf{r}(v_i)])\int_{G} f_{i}\, dm\right\|_{L^p(\Omega_j)}\\
\ll& \sum_{i=1}^n\sum_{j=1}^m
\rho(H_{t}[u_j,\mathsf{r}(v_i)])\left\|\mathcal{E}_{t}^{ij}\right\|_{L^p(B)}
\ll \rho(H_t)\sum_{i=1}^n\sum_{j=1}^m \left\|\mathcal{E}_{t}^{ij}\right\|_{L^p(B)},
\end{align*}
where
$$
\mathcal{E}_{t}^{ij}(y):=\left|\frac{1}{\rho(H_t[u_j,\mathsf{r}(v_i)])}\int_{H_t[u_j,\mathsf{r}(v_i)]}
F_i(h^{-1}y)\,d\rho(h)
- \int_{Y} F_i\, d\mu\right|.
$$
Since $\mathcal{E}_{t}^{ij}\to 0$  in $L^p(B)$ as $t\to\infty$
by the mean ergodic theorem for the family $\beta_t^{g_1,g_2}$, this proves (\ref{eq_g_d_norm_0}).

Using that
\begin{align*}
\|\mathcal{A}_{t+\vre} -\mathcal{A}_{t}\|_{L^p(\Omega)}\le &\|\mathcal{A}_{t+\vre}-\mathcal{A}^G_{t+\vre}\|_{L^p(\Omega)}
+\|\mathcal{A}^G_{t+\vre}-\mathcal{A}^G_{t}\|_{L^p(\Omega)}\\
&+\|\mathcal{A}^G_{t}-\mathcal{A}_{t}\|_{L^p(\Omega)}.
\end{align*}
we deduce from  (\ref{eq_d_est_00}) and (\ref{eq_g_d_norm_0}) that
\begin{equation}\label{eq:gamma0} 
\limsup_{t\to\infty}\frac{\|\mathcal{A}_{t+\vre}
  -\mathcal{A}_{t}\|_{L^p(\Omega)}}{\rho(H_t)}\ll \omega(7\vre)\|\phi\|_{L^q(D)}.
\end{equation}

Now we finally complete the proof by showing that the difference between the discrete sampling operators and the continuous ones converges to zero in norm, namely : 
$$
\left\| \pi_X(\Lambda_t)\phi-\pi_X(\Lambda^G_t)\phi\right\|_{L^p(D)}=o(\rho(H_t))\quad\hbox{as $t\to\infty$.}
$$
By \eqref{eq:measure}, 
$$
\left\| \pi_X(\Lambda_t)\phi-\pi_X(\Lambda^G_t)\phi\right\|_{L^p(D)}
\ll\left\| \pi_X(\Lambda_t)\phi\circ\mathsf{p}_X-\pi_X(\Lambda^G_t)\phi\circ\mathsf{p}_X\right\|_{L^p(\Omega)}.
$$
By (\ref{eq_g_a_0}),
\begin{align*}
&\left\| \pi_X(\Lambda_t)\phi\circ\mathsf{p}_X -\pi_X(\Lambda^G_t)\phi\circ\mathsf{p}_X\right\|_{L^p(\Omega)}\\
\le& \left\|\pi_X(\Lambda_t)\phi\circ\mathsf{p}_X -\mathcal{A}_{t-3\vre}\right\|_{L^p(\Omega)}
 +\left\|\mathcal{A}_{t-3\vre} -\pi_X(\Lambda^G_t)\phi\circ\mathsf{p}_X \right\|_{L^p(\Omega)}\\
\le& \left\|\mathcal{A}_{t+3\vre} -\mathcal{A}_{t-3\vre}\right\|_{L^p(\Omega)}
 +\left\|\mathcal{A}_{t-3\vre} -\pi_X(\Lambda^G_t)\phi\circ\mathsf{p}_X \right\|_{L^p(\Omega)},
\end{align*}
and by (\ref{eq_d_b_0}),
\begin{align*}
&\left\|\mathcal{A}_{t-3\vre} -\pi_X(\Lambda^G_t)\phi\circ\mathsf{p}_X \right\|_{L^p(\Omega)}\\
\le& \left\|\mathcal{A}_{t-3\vre} - \mathcal{A}^G_{t-3\vre}\right\|_{L^p(\Omega)}
+\left\|\pi_X(\Lambda^G_t)\phi\circ\mathsf{p}_X - \mathcal{A}^G_{t-3\vre}\right\|_{L^p(\Omega)}\\
\le & 
\left\|\mathcal{A}_{t-3\vre} - \mathcal{A}^G_{t-3\vre}\right\|_{L^p(\Omega)}
+\left\|\pi_X(\Lambda^G_t)\phi\circ\mathsf{p}_X - \pi_X(\Lambda^G_{t-6\vre})\phi\circ\mathsf{p}_X \right\|_{L^p(\Omega)}.
\end{align*}
Therefore, it follows from (\ref{eq:gamma0}), (\ref{eq_g_d_norm_0}), and (\ref{eq_d_est_0})
that
$$
\limsup_{t\to\infty}
\frac{\left\| \pi_X(\Lambda_t)\phi\circ\mathsf{p}_X -\pi_X(\Lambda^G_t)\phi\circ\mathsf{p}_X\right\|_{L^p(\Omega)}}{\rho(H_t)}
\ll (\omega(42\vre)+\omega(6\vre))\|\phi\|_{L^q(D)}.
$$
Since $\omega(\vre)\to 0$ as $\vre\to 0^+$, this completes the proof of Theorem \ref{th:mean}.

\section{The pointwise ergodic theorem}\label{sec:pointwise}

We now prove that the validity of the the pointwise ergodic theorem for $\pi_Y(\beta_t)$ implies its validity 
for $\pi_X(\lambda_t)$. 

\begin{thm}\label{th:pointwise}
Let $1<p\le \infty$.
Assume that {\rm (CA1), (CA2), (A1), (A2)} with $r=p/(p-1)$, {\rm (A3), and (S)} hold. 
Assume that for every $g_1,g_2\in G$ and every compact domain $B\subset Y$,
family $\beta_t^{g_1,g_2}$ satisfies the pointwise ergodic theorem 
in $L^p(B)$, namely, for every $F\in L^p(B)$,
$$
\lim\pi_Y(\beta_t^{g_1,g_2})F(y)\to\int_Y F\,
  d\mu\quad\hbox{for almost every $y\in B$.}
$$
Then for every compact domain $D\subset X$, 
the family $\lambda_t$ satisfies the pointwise ergodic theorem
in $L^p(D)$, namely,
for every $\phi\in L^p(D)$,
\begin{equation}\label{eq:mean}
\lim_{t\to\infty} \pi_X(\lambda_t)\phi(x)=\int_X \phi\,d\nu_x\quad\hbox{for almost every $x\in D$.}
\end{equation}
\end{thm}

\begin{proof}
In the proof we shall use notations introduced in the proof of Theorem \ref{th:mean}.
As in that proof, we reduce our argument to the case when $\phi\ge 0$ and the section
$\mathsf{s}$ is continuous on $D$. Moreover, because of (\ref{eq:point}), it is sufficient to show that
$$
\lim_{t\to\infty} (\pi_X(\lambda_t)\phi(x)-\pi_X(\lambda^G_t)\phi(x))= 0\quad\hbox{for almost every $x\in D$.}
$$
Let $\vre\in (0,1/12)$, $\Omega=D_0\mathsf{s}(D)$ where $D_0$ is a compact
subset of $H$ with positive measure, and $B=\mathsf{p}_Y(\Omega)$.
For $g\in \Omega$ and $x=\mathsf{p}_X(g)$, we have by (\ref{eq_g_a_0}) and
(\ref{eq_d_b_0}),
\begin{align*}
&\left| \pi_X(\Lambda_t)\phi(x) -\pi_X(\Lambda^G_t)\phi(x)\right|\\
\le& \left|\pi_X(\Lambda_t)\phi(x) -\mathcal{A}_{t-3\vre}(g)\right|
 +\left|\mathcal{A}_{t-3\vre}(g) -\pi_X(\Lambda^G_t)\phi(x) \right|\\
\le& \left|\mathcal{A}_{t+3\vre}(g) -\mathcal{A}_{t-3\vre}(g)\right|
 +\left|\mathcal{A}_{t-3\vre}(g) -\pi_X(\Lambda^G_t)\phi(x) \right|\\
\le&  \left|\mathcal{A}_{t+3\vre}(g) -\mathcal{A}_{t-3\vre}(g)\right|
+ \left|\mathcal{A}_{t-3\vre}(g) - \mathcal{A}^G_{t-3\vre}(g)\right|
+\left|\pi_X(\Lambda^G_t)\phi(x) - \mathcal{A}^G_{t-3\vre}(g)\right|\\
\le & \left|\mathcal{A}_{t+3\vre}(g) -\mathcal{A}_{t-3\vre}(g)\right|+
\left|\mathcal{A}_{t-3\vre}(g) - \mathcal{A}^G_{t-3\vre}(g)\right|
+\left|\mathcal{A}^G_{t+3\vre}(g) - \mathcal{A}^G_{t-3\vre}(g)\right|.
\end{align*}
We will estimate each of these three terms separately.

Recall that $\Omega=\sqcup_{j=1}^m \Omega_j$, and it follows from the definition of 
$\mathcal{A}_t$ and $\mathcal{A}_t^G$ (see (\ref{eq_g_0}) and (\ref{eq_d_0})) 
and (CA1)--(CA2) that for $g\in \Omega_j$,
\begin{align*}
&|\mathcal{A}_t(g)-\mathcal{A}^G_t(g)|\\
=& \left|\sum_{i=1}^n \int_{H_{t}[u_j,\mathsf{r}(v_i)]}
  F_{i}(h^{-1}\mathsf{p}_Y(g))\,d\rho(h) 
- \sum_{i=1}^n\rho(H_{t}[u_j,\mathsf{r}(v_i)])\int_{Y} F_{i}\, d\mu\right|\\
\ll& \sum_{i=1}^n \rho(H_t)\left| 
\frac{1}{\rho(H_{t}[u_j,\mathsf{r}(v_i)])}\int_{H_{t}[u_j,\mathsf{r}(v_i)]}
  F_{i}(h^{-1}\mathsf{p}_Y(g))\,d\rho(h) 
- \int_{Y} F_{i}\, d\mu
\right|.
\end{align*}
By Lemma \ref{lem:norm_est2}, we have $F_i\in L^p(B)$.
Hence, it follows from the pointwise ergodic theorem in $L^p(B)$ that
\begin{equation}\label{eq:gdfdrd}
\lim_{t\to\infty} \frac{|\mathcal{A}_t(g)-\mathcal{A}^G_t(g)|}{\rho(H_t)}=0
\end{equation}
for almost every $g\in \Omega$. By (\ref{eq_d_b_0}) and Lemma \ref{c:reg}(i) combined with (A2),
\begin{align}\label{eq:lslss}
\left|
\mathcal{A}^G_{t+3\vre}(g)-\mathcal{A}^G_{t-3\vre}(g)
\right| &\le
\left|\pi_X(\Lambda^G_{t+6\vre})\phi(x) - \pi_X(\Lambda^G_{t-6\vre})\phi(x)
\right|\\
& \ll \omega(12\vre) \rho(H_t)\|\phi\|_{L^p(D)}\nonumber 
\end{align}
on a set of full measure of $g\in \Omega$.
Since 
\begin{align*}
|\mathcal{A}_{t+3\vre}(g) -\mathcal{A}_{t-3\vre}(g)|\le &|\mathcal{A}_{t+3\vre}(g)-\mathcal{A}^G_{t+3\vre}(g)|
+|\mathcal{A}^G_{t+3\vre}(g)-\mathcal{A}^G_{t-3\vre}(g)|\\
&+|\mathcal{A}^G_{t-3\vre}(g)-\mathcal{A}_{t-3\vre}(g)|,
\end{align*}
combining (\ref{eq:gdfdrd}) and (\ref{eq:lslss}), we deduce that
\begin{align}\label{eq:lslss_s}
\limsup_{t\to\infty} \frac{|\mathcal{A}_{t+3\vre}(g)-\mathcal{A}_{t-3\vre}(g)|}{\rho(H_t)}
\ll \omega(12\vre) \rho(H_t)\|\phi\|_{L^p(D)}
\end{align}
on a set of full measure.
Finally, we deduce from (\ref{eq:gdfdrd}), (\ref{eq:lslss}), and (\ref{eq:lslss_s})
that on a set of full measure in $D$,
\begin{align*}
\limsup_{t\to\infty} \frac{\left| \pi_X(\Lambda_t)\phi(x)
    -\pi_X(\Lambda^G_t)\phi(x)\right|}{\rho(H_t)}\ll
\omega(12\vre) \|\phi\|_{L^p(D)}.
\end{align*}
Since $\omega(\vre)\to 0$ as $\vre\to 0^+$, this proves the theorem.
\end{proof}

\section{Quantitative mean  ergodic theorem}\label{sec:mean_quant}

\subsection{Admissibility}

The goal of this section is prove the quantitative mean and pointwise ergodic theorems for the normalized sampling operators 
$\pi_{X}(\lambda_t)$ defined in (\ref{eq:lambda}).
We shall use notation from Section \ref{sec:notation} and assume that the sets $G_t$ satisfy
the following additional regularity properties:

\begin{enumerate}
\item[(HA1)] There exist 
a basis $\{O_\vre\}_{\vre\in (0,1]}$ of symmetric neighborhoods  of the identity in $G$ 
and $c>0$ such that for every $\vre\in (0,1)$ and $t\ge t_0$
\begin{equation}\label{eq:ha1}
O_\vre \cdot G_t\cdot O_\vre \subset G_{t+c\vre}.
\end{equation}
Moreover, for all sufficiently small $\vre$,
there exists a nonnegative function $\chi_\vre\in C_c^l(H)$ such that
\begin{align}\label{eq_psi_e}
\hbox{\rm supp}(\chi_\vre)\subset O_\vre,\quad
 \int_H \chi_\vre \, d\rho=1,\quad  \|\chi_\vre\|_{L^q_l(H)}\ll \varepsilon^{-\kappa},
\end{align}
and for every compact $\Omega\subset G$ and $\vre\in (0,1)$, there exists a cover
\begin{equation}\label{eq:cover}
\Omega\subset \bigcup_{i=1}^{n_\vre} v_i O_\vre
\end{equation}
with $n_\vre\ll \vre^{-d}$.

\item[(HA2)] 
For every $u\in G$  and for a compact domain $\Omega\subset G$, there exist $c,\theta>0$
such that 
for every $t\ge t_0$ and $\vre\in (0,1)$,
$$
\left(\int_{\Omega} (\rho(H_{t+\vre}[u,v])- \rho(H_t[u,v]))^r\,dm(v)\right)^{1/r}\le  c\, \vre^\theta \rho(H_t).
$$ 

\item[(HA2$^\prime$)] 
For every compact $\Omega\subset G$, there exist $c,\theta>0$
such that 
for every $t\ge t_0$ and $\vre\in (0,1)$,
$$
\left(\int_{\Omega\times \Omega} (\rho(H_{t+\vre}[u,v])- \rho(H_t[u,v]))^r\,dm(u)dm(v)\right)^{1/r}\le  c\,\vre^\theta \rho(H_t).
$$ 

\item[(HS)] every $x\in X$ has a neighbourhood $U$ such that there exists a Lipschitz
(with respect to the neighbourhoods $O_\vre$) section $\mathsf{s}:U\to G$ of the map $\mathsf{p}_X$.

\end{enumerate}

We recall that when $l>0$ we assume that $G$ is a Lie group and $H$ is  a closed subgroup. 
In this case, we take $O_\vre$ to be the symmetric $\vre$-neighbourhoods of identity with respect to 
a fixed Riemannian metric in $G$.
Then (\ref{eq_psi_e}) holds with $\kappa=l+\dim(H)(1-1/q)$ and (\ref{eq:cover}) holds with $d=\dim(G)$.
Moreover, in this case (HS) also holds, and one can choose the section $\mathsf{s}$
to be smooth on $U$.


We say that a function $E(t)$ is {\it coarsely admissible} if 
there exists $c>0$ such that $\sup_{s\in [t,t+1]} E(s)\le c\, E(t)$ for all $t$.

\begin{thm}\label{th_dual_Ger}
Let $1\le p<q\le \infty$ and  $l\in\mathbb{Z}_{\ge 0}$.
Suppose that {\rm (CA1), (CA2), (HA1), (HA2$^\prime$)} with $r=pq/(q-p)$, and {\rm (HS)} hold. Assume that 
for every $g_1,g_2\in G$ and compact $B\subset Y$,
the family $\beta_t^{g_1,g_2}$ satisfies the quantitative mean ergodic
theorem in $L_l^q(B)^+$ with respect to $\|\cdot\|_{L^p(B)}$, and the error term $E(t)$ is 
coarsely admissible and uniform 
over  $g_1,g_2$ in a compact subset of $G$. Namely, for every $F\in L_l^p(B)^+$
and sufficiently large $t$,
$$
\left\| \pi_Y(\beta_t^{g_1,g_2})F-\int_Y F\,
  d\mu\right\|_{L^p(B)}\ll_{p,q,l,B} E(t)\,\|F\|_{L_l^{q}(B)}.
$$
Then for every compact $D\subset X$, $\phi\in L_l^q(D)^+$ and  sufficiently large $t$,
$$
\left\| \pi_X(\lambda_t)\phi- \pi_X(\lambda^G_t)\phi
\right\|_{L^p(D)}\ll_{p,q,l,D} 
E(t)^{\delta} \,\|\phi\|_{L_l^{q}(D)}
$$
with $\delta>0$ independent of $D$ and $\phi$. 
\end{thm}

\begin{proof}
The proof of theorem will follows the same outline as the proof of Theorem \ref{th:mean}.
Throughout the proof, we shall use a parameter $\vre=\vre(t)\in (0,1)$ such that
$\vre(t)\to 0$ as $t\to\infty$, which will be specified later.
Because of (HS),
decomposing the function $\phi$ into a sum of functions with small support, we reduce the proof
to the situation when there exists a section $\mathsf{s}:D\to G$ of the factor map $\mathsf{p}_X:G\to H\backslash
G=X$ such that $\mathsf{s}|_D$ is Lipschitz. Moreover, when $G$ and $H$ are Lie groups,
we may assume that $\mathsf{s}$ is smooth on $D$.

Let the function $\chi_\vre\in C_c^l(H)$ be as in (\ref{eq_psi_e}).
We define the function $f_\vre: G\to\mathbb{R}$ as in (\ref{eq:funcion}).
Since $\mathsf{s}|_D$ is Lipschitz and $D$ is compact, it follows from (HA1) that 
there exists $c>0$ such that for all sufficiently small $\vre$,
\begin{align}
O_\vre\cdot G_t\cdot g^{-1} O_\vre g&\subset
G_{t+c\vre}\quad\hbox{for every $g\in\mathsf{s}(D)$,}\label{eq:pert_1}\\
G_t\cdot \mathsf{s}(y O_\vre)^{-1} &\subset G_{t+c\vre}\cdot \mathsf{s}(y)^{-1}\quad\hbox{for 
every $y\in D$}.\label{eq:pert2_1}
\end{align}
Therefore, we may argue as in the proof of Theorem \ref{th:mean}
(cf. (\ref{eq_claim2_0})--(\ref{eq_claim2a_0}) and (\ref{eq_claim2b_0})--(\ref{eq_claim2a'_0})) to show that for $u\in G$,
\begin{align}
\pi_X(\Lambda_t)\phi(\mathsf{p}_X(u))\le \sum_{\gamma\in\Gamma}\int_{H_{t+c\vre}[u,\mathsf{r}(u\gamma)]} f_\varepsilon(h^{-1}u\gamma)\,d\rho(h),\label{eq_claim2}\\
\pi_X(\Lambda_t)\phi(\mathsf{p}_X(u))\ge \sum_{\gamma\in\Gamma}\int_{H_{t-c\varepsilon}[u,\mathsf{r}(u\gamma)]} f_\varepsilon(h^{-1}u\gamma)\,d\rho(h),\label{eq_claim2a}
\end{align}
and
\begin{align}
\pi_X(\Lambda^G_t)\phi(\mathsf{p}_X(u))&\le \int_G\int_{H_{t+c\vre}[u,\mathsf{r}(ug)]} f_\vre (h^{-1}ug)\,d\rho(h), \label{eq_claim2'}\\
\pi_X(\Lambda^G_t)\phi(\mathsf{p}_X(u))&\ge \int_G\int_{H_{t-c\vre}[u,\mathsf{r}(ug)]}
f_\vre(h^{-1}ug)\,d\rho(h)dm(g),\label{eq_claim2a'}
\end{align}

It follows from (\ref{eq:pert2_1}) for every $u\in G$, $g\in H\mathsf{s}(D)$ and $g'\in HgO_\vre$,
\begin{equation}\label{eq_norm_est}
H_{t-c\vre}[u,\mathsf{r}(g')]\subseteq H_t[u,\mathsf{r}(g)]\subseteq H_{t+c\vre}[u,\mathsf{r}(g')].
\end{equation}

\subsection{Local analysis} 

Let $\Omega:=(\overline{{O}_1\cap H})^{-1}\mathsf{s}(D)$ and $B=\mathsf{p}_Y(\Omega)$. Then $\supp(f_\vre)\subset
\Omega$, and by (HA1), there exists a finite cover 
$$
\Omega\subset \bigcup_{i=1}^{n_\vre} v_i O_\vre
$$ 
with $v_i\in \Omega$ and $n_\vre\ll \vre^{-d}$. When $l=0$, we choose a family of 
bounded measurable functions $\{\psi_{\vre,i}\}_{i=1}^{n_\vre}$
such that $\sum_{i=1}^{n_\vre} \psi_{\vre,i}=1$ on  $\Omega$ and $\supp(\psi_{\vre,i})\subset v_iO_\vre$.
When $l>0$, we choose a smooth partition of unity $\{\psi_{\vre,i}\}_{i=1}^{n_\vre}$
satisfying the above properties.
Using the standard construction of the partition of unity (see, for instance,
\cite[Th.~2.13]{rudin}), we get the estimate
\begin{equation}\label{eq:pppppsss}
\|\psi_{\vre,i}\|_{C^l}\ll \vre^{-l}.
\end{equation}
Let $f_{\varepsilon,i}=f_\varepsilon\cdot\psi_{\vre,i}$ and 
$F_{\varepsilon,i}(\mathsf{p}_Y(g))=\sum_{\gamma\in\Gamma}f_{\varepsilon,i}(g\gamma)$.
For sufficiently small $\vre$, the map $\mathsf{p}_Y:G\to Y=G/\Gamma$ is a bijection
on $\supp(f_{\varepsilon,i})$, so that
\begin{align}\label{eq:nnn}
\int_Y F_{\vre,i} \, d\mu=\int_G f_{\vre,i}\, dm\quad\hbox{and}\quad
\|F_{\vre,i}\|_{L^q_l(B)}&\ll \|f_{\vre,i}\|_{L^q_l(\Omega)}.
\end{align}
Using (\ref{eq_norm_est}) we deduce as in the proof of Theorem \ref{th:mean} (cf. (\ref{eq_step1_0})--(\ref{eq_step1a_0})) that
\begin{align}\label{eq_step1}
\sum_{\gamma\in\Gamma} \int_{H_t[u,\mathsf{r}(u\gamma)]}
f_\varepsilon(h^{-1}u\gamma)\,d\rho(h)
\le\sum_{i=1}^{n_\vre}\int_{H_{t+c\vre}[u,\mathsf{r}(v_i)]} F_{\varepsilon,i}(h^{-1}\mathsf{p}_Y(u))\,d\rho(h),
\end{align}
and
\begin{align}\label{eq_step1a}
\sum_{\gamma\in\Gamma} \int_{H_t[u,\mathsf{r}(u\gamma)]} f_\varepsilon(h^{-1}u\gamma)\,d\rho(h)
\ge \sum_{i=1}^{n_\vre}\int_{H_{t-c\vre}[u,\mathsf{r}(v_i)]} F_{\varepsilon,i}(h^{-1}\mathsf{p}_Y(u))\,d\rho(h).
\end{align}

By (HA1) there exist a subset $\{u_j\}_{j=1}^{m_\vre}$ of $\Omega$ with $m_\vre\ll \vre^{-d}$
and a measurable partition 
$$
\Omega=\bigsqcup_{j=1}^{m_\vre} \Omega_j\quad\hbox{such that}\quad \Omega_j\subset u_j O_\vre.
$$
It follows from (\ref{eq:pert_1}) that for all $u\in \Omega_j$,
\begin{equation}\label{eq_norm_est2}
H_{t-c\vre}[u_j,\mathsf{r}(v_i)]\subseteq H_t[u,\mathsf{r}(v_i)]\subseteq H_{t+c\vre}[u_j,\mathsf{r}(v_i)].
\end{equation}
We introduce a  measurable function $\mathcal{A}_t:\Omega\to \mathbb{R}$ defined piecewise by
\begin{equation}\label{eq:omega_t}
\mathcal{A}_t(u)=\sum_{i=1}^{n_\vre}\int_{H_{t}[u_j,\mathsf{r}(v_i)]} F_{\varepsilon,i}(h^{-1}\mathsf{p}_Y(u))\,d\rho(h),\quad u\in \Omega_j.
\end{equation}
By (\ref{eq_norm_est2}), for $u\in \Omega$,
$$
\mathcal{A}_{t-c\vre}(u)\le\sum_{i=1}^{n_\vre}\int_{H_{t}[u,\mathsf{r}(v_i)]}
F_{\varepsilon,i}(h^{-1}\mathsf{p}_Y(u))\,d\rho(h)\le \mathcal{A}_{t+c\vre}(u).
$$
Therefore, combining (\ref{eq_claim2}), (\ref{eq_claim2a}),
(\ref{eq_step1}), (\ref{eq_step1a}), we deduce 
for $u\in \Omega$,
\begin{equation}\label{eq_g_a}
\mathcal{A}_{t-3c\vre}(u)\le \pi_X(\Lambda_t)\phi(\mathsf{p}_X(u))\le \mathcal{A}_{t+3c\vre}(u).
\end{equation}

We also prove similar estimate for the averages $\pi_X(\Lambda_t^G)$.
By (\ref{eq_norm_est}) and (\ref{eq_norm_est2}),  we have for $u\in \Omega_{j}$,
\begin{align*}
\int_G\int_{H_{t}[u,\mathsf{r}(ug)]} f(h^{-1}ug)\,d\rho(h)dm(g) 
&\le \sum_{i=1}^{n_\vre}\int_G\int_{H_{t+c\vre}[u,\mathsf{r}(v_i)]} f_{\vre,i}(h^{-1}ug)\,d\rho(h)dm(g) \\
&=\sum_{i=1}^n\rho(H_{t+c\vre}[u,\mathsf{r}(v_i)])\int_{G} f_{\vre,i}\, dm\\
&\le \sum_{i=1}^{n_\vre}\rho(H_{t+2c\vre}[u_j,\mathsf{r}(v_i)])\int_{G} f_{\vre,i}\, dm,
\end{align*}
and similarly,
\begin{align*}
\int_G\int_{H_{t}[u,\mathsf{r}(ug)]} f(h^{-1}ug)\,d\rho(h)dm(g)
\ge\sum_{i=1}^{n_\vre}\rho(H_{t-2c\vre}[u_j,\mathsf{r}(v_i)])\int_{G} f_{\varepsilon,i}\, dm.
\end{align*}
Hence, it follows from (\ref{eq_claim2'}) and (\ref{eq_claim2a'}) that the function $\mathcal{A}^G_t: \Omega\to \mathbb{R}$ defined by
\begin{equation}\label{eq:omegaG_t}
\mathcal{A}^G_t(u)=\sum_{i=1}^{n_\vre}\rho(H_{t}[u_j,\mathsf{r}(v_i)])\int_{G} f_{\varepsilon,i}\, dm,\quad u\in \Omega_{j},
\end{equation}
satisfies
\begin{equation}\label{eq_d_b}
\mathcal{A}^G_{t-3c\vre}(u)\le \pi_X(\Lambda_t^G)\phi(\mathsf{p}_X(u))\le \mathcal{A}^G_{t+3c\vre}(u).
\end{equation}
for $u\in \Omega$.

Since by (\ref{eq:measure}),
$$
\left\|\pi_X(\Lambda^G_{t+c\vre})\phi\circ\mathsf{p}_X - \pi_X(\Lambda^G_{t})\phi\circ\mathsf{p}_X
\right\|_{L^p(\Omega)}\ll
\left\|\pi_X(\Lambda^G_{t+c\vre})\phi - \pi_X(\Lambda^G_{t})\phi \right\|_{L^p(D)},
$$
it follows from Lemma \ref{c:reg}(ii) combined with (HA2$^\prime$) that
\begin{equation}\label{eq_d_est_0_0}
\left\|\pi_X(\Lambda^G_{t+c\vre})\phi\circ\mathsf{p}_X - \pi_X(\Lambda^G_{t})\phi\circ\mathsf{p}_X \right\|_{L^p(\Omega)}
\ll \vre^\theta \rho(H_t)\|\phi\|_{L^q(D)},
\end{equation}
Moreover, combining (\ref{eq_d_b}) and (\ref{eq_d_est_0_0}), we deduce
that
\begin{align}\label{eq_d_est_00_0}
\left\|
\mathcal{A}^G_{t+c\vre}-\mathcal{A}^G_{t}
\right\|_{L^p(\Omega)} &\le
\left\|\pi_X(\Lambda^G_{t+4\vre})\phi\circ\mathsf{p}_X - \pi_X(\Lambda^G_{t-3\vre})\phi\circ\mathsf{p}_X
\right\|_{L^p(\Omega)}\\
&\ll \vre^\theta \rho(H_t)\|\phi\|_{L^q(D)}.\nonumber
\end{align}

\subsection{The duality argument}

Our next task is to estimate $\|\mathcal{A}_t-\mathcal{A}^G_t\|_{L^p(\Omega)}$. Let
\begin{align*}
\mathcal{E}_t^{ij}(y)&=
\frac{1}{\rho(H_{t}[u_j,\mathsf{r}(v_i)])}\int_{H_{t}[u_j,\mathsf{r}(v_i)]} F_{\varepsilon,i}(h^{-1}y)\,d\rho(h) 
-\int_{Y} F_{\varepsilon,i}\, d\mu\\
&=\pi_Y(\beta^{u_j,\mathsf{r}(v_i)})F_{\varepsilon,i}(y)-\int_{Y} F_{\varepsilon,i}\, d\mu.
\end{align*}
It follows from the quantitative mean ergodic theorem in $L^q_l(B)^+$ with respect to 
$\|\cdot\|_{L^p(B)}$
that
\begin{align*}
\|\mathcal{E}_t^{ij}\|_{L^p(B)}\le E(t)\|F_{\varepsilon,i}\|_{L^q_l(B)}.
\end{align*}
By (\ref{eq:nnn}) and (\ref{eq:pppppsss}),
\begin{align*}
\|F_{\varepsilon,i}\|_{L^q_l(B)}\ll
\|f_{\varepsilon,i}\|_{L^q_l(\Omega)}=\|f_\varepsilon\cdot\psi_{\vre,i}\|_{L^q_l(\Omega)}
\ll \|\psi_{\vre,i}\|_{C^l}\cdot \|f_\varepsilon\|_{L^q_l(\Omega)}
\ll \vre^{-l} \|f_\varepsilon\|_{L^q_l(\Omega)}
\end{align*}
Recall that when $l>0$, we are assuming that $G$ and $H$ are Lie groups, and
the section $\mathsf{s}$ is smooth on $D$.
This implies that the map 
$$
\mathsf{p}^{-1}_X(D)\to H\times D: g\mapsto (\mathsf{h}(g),\mathsf{p}_X(g))
$$
is a diffeomorphism. Then it follows from 
the definition of $f_\vre$ (see (\ref{eq:funcion}))
that 
$$
\|f_\varepsilon\|_{L^q_l(\Omega)}\ll 
\|\chi_\vre\|_{L^q_l(H)} \|\phi\|_{L^q_l(D)}
\ll \vre^{-\kappa} \|\phi\|_{L^q_l(D)}.
$$
A similar estimate when $l=0$  follows from (\ref{eq:measure}) and (\ref{eq_psi_e}).
Hence, we conclude that
$$
\|\mathcal{E}_t^{ij}\|_{L^p(B)}\ll E(t)\vre^{-(l+\kappa)} \|\phi\|_{L^q_l(D)}.
$$
Now it follows from the triangle inequality, Lemma \ref{lem:norm_est1}, and (CA1)--(CA2) that
\begin{align}\label{eq:llll}
&\|\mathcal{A}_t-\mathcal{A}^G_t\|_{L^p(\Omega)}\le \sum_{j=1}^{m_\vre}\|\mathcal{A}_t-\mathcal{A}^G_t\|_{L^p(\Omega_j)} \nonumber\\
=& \sum_{j=1}^{m_\vre}\left\| \sum_{i=1}^{n_\vre}\int_{H_{t}[u_j,\mathsf{r}(v_i)]} F_{\varepsilon,i}(h^{-1}\mathsf{p}_Y(g))\,d\rho(h) 
-\sum_{i=1}^{n_\vre}\rho(H_{t}[u_j,\mathsf{r}(v_i)])\int_{G} f_{\varepsilon,i}\, dm\right\|_{L^p(\Omega_j)}\nonumber\\
\ll&  \sum_{i=1}^{n_\vre}\sum_{j=1}^{m_\vre}
\rho(H_{t}[u_j,\mathsf{r}(v_i)])\left\|\mathcal{E}_{t}^{ij}\right\|_{L^p(B)}\nonumber\\
\ll & \,\,\vre^{-2d} \rho(H_t) E(t)\vre^{-(l+\kappa)} \|\phi\|_{L^q_l(D)}.
\end{align}
Combining (\ref{eq_d_est_00_0}) and (\ref{eq:llll}), we obtain
\begin{align*}
&\|\mathcal{A}_{t+c\vre} -\mathcal{A}_{t}\|_{L^p(\Omega)}\\
\le &\|\mathcal{A}_{t+c\vre}-\mathcal{A}^G_{t+c\vre}\|_{L^p(\Omega)}
+\|\mathcal{A}^G_{t+c\vre}-\mathcal{A}^G_{t}\|_{L^p(\Omega)}
+\|\mathcal{A}^G_{t}-\mathcal{A}_{t}\|_{L^p(\Omega)}\\
\ll &
\left(\rho(H_{t+c\vre}) E(t+c\vre)+
\rho(H_{t}) E(t)\right)\vre^{-(2d+l+\kappa)} \|\phi\|_{L^q_l(D)}+\vre^\theta \rho(H_t)\|\phi\|_{L^q_l(D)}.
\end{align*}
Hence, it follows from (CA2) and coarse admissibility of the error term $E(t)$ that
\begin{equation}\label{eq:omega}
\|\mathcal{A}_{t+c\vre} -\mathcal{A}_{t}\|_{L^p(B)}
\ll  
(E(t)\vre^{-(2d+l+\kappa)}+\vre^\theta)\rho(H_t)\|\phi\|_{L^q(D)}.
\end{equation}
By \eqref{eq:measure} and (\ref{eq_g_a}),
\begin{align*}
&\left\| \pi_X(\Lambda_t)\phi-\pi_X(\Lambda^G_t)\phi\right\|_{L^p(D)}\\
\ll&\left\| \pi_X(\Lambda_t)\phi\circ\mathsf{p}_X-\pi_X(\Lambda^G_t)\phi\circ\mathsf{p}_X\right\|_{L^p(\Omega)}\\
\le& \left\|\pi_X(\Lambda_t)\phi\circ\mathsf{p}_X -\mathcal{A}_{t-3c\vre}\right\|_{L^p(\Omega)}
 +\left\|\mathcal{A}_{t-3c\vre} -\pi_X(\Lambda^G_t)\phi\circ\mathsf{p}_X \right\|_{L^p(\Omega)}\\
\le& \left\|\mathcal{A}_{t+3c\vre} -\mathcal{A}_{t-3c\vre}\right\|_{L^p(\Omega)}
 +\left\|\mathcal{A}_{t-3c\vre} -\pi_X(\Lambda^G_t)\phi\circ\mathsf{p}_X \right\|_{L^p(\Omega)},
\end{align*}
and by (\ref{eq_d_b}),
\begin{align*}
&\left\|\mathcal{A}_{t-3c\vre} -\pi_X(\Lambda^G_t)\phi\circ\mathsf{p}_X \right\|_{L^p(\Omega)}\\
\le& \left\|\mathcal{A}_{t-3c\vre} - \mathcal{A}^G_{t-3c\vre}\right\|_{L^p(\Omega)}
+\left\|\pi_X(\Lambda^G_t)\phi\circ\mathsf{p}_X - \mathcal{A}^G_{t-3c\vre}\right\|_{L^p(\Omega)}\\
\le & 
\left\|\mathcal{A}_{t-3c\vre} - \mathcal{A}^G_{t-3c\vre}\right\|_{L^p(\Omega)}
+\left\|\pi_X(\Lambda^G_t)\phi\circ\mathsf{p}_X - \pi_X(\Lambda^G_{t-6c\vre})\phi\circ\mathsf{p}_X \right\|_{L^p(\Omega)}.
\end{align*}
Therefore, it follows from (\ref{eq:omega}),  (\ref{eq:llll}) combined with 
coarse admissibility of $E(t)$, and (\ref{eq_d_est_0_0})
that
$$
\left\| \pi_X(\Lambda_t)\phi -\pi_X(\Lambda^G_t)\phi \right\|_{L^p(D)}
\ll(E(t)\vre^{-(2d+l+\kappa)}+\vre^\theta)\rho(H_t)\|\phi\|_{L^q_l(D)}.
$$
This estimate holds for all sufficiently small $\vre$.
In order to optimise it, we pick
\begin{equation}\label{eq:epsilon}
\vre=E(t)^{(\theta+2d+l+\kappa)^{-1}}.
\end{equation}
This proves the theorem with $\delta=\theta/(\theta+2d+l+\kappa)$.
\end{proof}

\section{Quantitative pointwise ergodic theorem}\label{sec:quant_pointwise}

We now turn to the quantitative pointwise ergodic theorem for the normalized sampling operators $\lambda_X(\lambda_t)$. 

\begin{thm}\label{th_dual_pointwise}
Let $1\le p<q\le \infty$ and $l\in\mathbb{Z}_{\ge 0}$.
Suppose that {\rm (CA1), (CA2), (HA1), (HA2)} with $r=q/(q-1)$, {\rm (HA2$^\prime$)} with $r=pq/(q-p)$, and {\rm (HS)} hold. Assume that 
for every $g_1,g_2\in G$ and every compact domain $B\subset Y$,
the family $\beta_t^{g_1,g_2}$ satisfies the quantitative mean ergodic
theorem in $L_l^q(B)^+$ with respect to $\|\cdot\|_{L^p(B)}$
with exponential rate, and the error term is uniform
over  $g_1,g_2$ in a compact subset of $G$.
Then  for every compact domain $D\subset X$, a function  $\phi\in L^p_l(D)^+$ 
and almost every $x\in D$, there exists $\delta>0$ such that
$$
\left| \pi_X(\lambda_t)\phi(x)- \pi_X(\lambda^G_t)\phi(x)
\right| \le C(\phi,x) e^{-\delta t}
$$
for all $t\ge t_0$. Furthermore, $\delta$ is idependent of $D$, $\phi$ and $x$, and $\norm{C(\phi,\cdot)}_{L^p(D)}\le  C_{p,q}\norm{\phi}_{L_l^q(D)}$. 
\end{thm}

\begin{proof} 
This proof is a refinement of the proof of Theorem \ref{th_dual_Ger},
and we shall use some of the notation and estimates obtained there.
As in that proof, we reduce
our argument to the case when the section $\mathsf{s}$ is Lipschitz on $D$.
Moreover, if $G$ and $H$ are Lie groups, we reduce the proof to the case when
the section $\mathsf{s}$ is smooth on $D$.

By Theorem \ref{th_dual_Ger}, for some $\delta_1>0$,
\begin{equation}\label{eq:mean_est_2}
\left\|\pi_X(\Lambda_t)\phi-\pi_X(\Lambda^G_t)\phi\right\|_{L^p(D)}\ll e^{-\delta_1 t} \rho(H_t)
\|\phi\|_{L^q_l(D)}.
\end{equation}
Let $\Omega$ be a compact subset of $G$ as in the proof of Theorem \ref{th_dual_Ger},
and let $\mathcal{A}_t$ and $\mathcal{A}_t^G$ be functions on $\Omega$
defined as in \eqref{eq:omega_t} and \eqref{eq:omegaG_t}.
During the proof of Theorem \ref{th_dual_Ger}, we have established estimate (\ref{eq:llll})
with $\vre$ as in \eqref{eq:epsilon} which implies that
\begin{equation}\label{eq:mean_est_3}
\left\|\mathcal{A}_t-\mathcal{A}^G_t\right\|_{L^p(\Omega)}\ll e^{-\delta_2 t} \rho(H_t)
\|\phi\|_{L^q_l(D)}
\end{equation}
for some $\delta_2>0$. Let $\delta=\min\{\delta_1,\delta_2\}$.

We take an increasing sequence $\{t_i\}_{i\ge 0}$ that contains
all positive integers greater than $t_0$ and has uniform spacing $\lfloor e^{p\delta n/4}\rfloor^{-1}$
on the intervals $[n,n+1]$, $n\in \mathbb{N}$. Then
\begin{equation*}
t_{i+1}-t_i\ll e^{-p\delta \lfloor t_i\rfloor/4}
\end{equation*}
for all $i\ge 0$.
Let
$$
C(x,\phi):=\left(\sum_{i\ge 0} e^{p\delta t_i/2} \rho(H_{t_i})^{-p}
  \left|\pi_X(\Lambda_{t_i})\phi(x)-\pi_X(\Lambda^G_{t_i})\phi(x)\right|^p\right)^{1/p}.
$$
Then we have
\begin{align}\label{eq:hhh0}
\left|\pi_X(\Lambda_{t_i})\phi(x)-\pi_X(\Lambda^G_{t_i})\phi(x) \right|\le C(x,\phi) e^{-\delta t_i/2}\rho(H_{t_i})
\end{align}
for all $i\ge 0$.
It follows from (\ref{eq:mean_est_2}) that
\begin{align*}
\|C(\cdot ,\phi)\|^p_{L^p(D)}&=
\sum_{i\ge 0} \left\| e^{\delta t_i/2} \rho(H_{t_i})^{-1}
  \left|\pi_X(\Lambda_{t_i})\phi-\pi_X(\Lambda^G_{t_i})\phi\right|\right\|^p_{L^p(D)}\\
&\ll \sum_{i\ge 0} e^{-p\delta t_i/2}\|\phi\|^p_{L^q_l(D)}
\le \sum_{n\ge \lfloor t_0 \rfloor } e^{-p\delta n/2} \lfloor e^{p\delta n/4}\rfloor 
\|\phi\|_{L^q_l(D)}^p<\infty,
\end{align*}
and, in particular, $C(x,\phi)$ is finite for almost every $x\in D$.

Using similar argument, we deduce from \eqref{eq:mean_est_3} that
\begin{align}\label{eq:hhhh00}
\left| \mathcal{A}_{t_i}(u)-\mathcal{A}^G_{t_i}(u) \right|\le C^\prime(u,\phi) e^{-\delta t_i/2}\rho(H_{t_i})
\end{align}
for all $i\ge 0$, where the estimator $C^\prime(u,\phi)$ is finite for almost all $u\in \Omega$.

To finish the proof, we need to extend estimate (\ref{eq:hhh0}) to general $t$.
We pick $t_i<t$ such that
$$
\vre:=t-t_i\ll  e^{-p\delta  \lfloor t_i\rfloor /4}\ll e^{-p\delta  t/4}.
$$
Let $t^+$ be the least element of $\{t_i\}_{i\ge 0}$ that satisfies
$t_i\ge t+3c\vre$, and let 
$t^-$ be the greatest element of $\{t_i\}_{i\ge 0}$ that satisfies
$t_i\le t-3c\vre$. Note that
$$
t^+-t^-\ll e^{-p\delta  t/4}.
$$
We deduce from (\ref{eq_g_a}) that for $u\in \Omega$ and $x=\mathsf{p}_X(u)\in D$,
$$
\left|\pi_X(\Lambda_{t})\phi(x)-\pi_X(\Lambda_{t_i})\phi(x) \right|
\le \left|\mathcal{A}_{t^+}(u)-\mathcal{A}_{t^-}(u)\right|.
$$
Then it follows from (\ref{eq:hhhh00}),  (CA2), and (\ref{eq_d_est_00_0}) (where we use again the pointwise estimate arising from summing the differences 
$e^{\eta  t} \norm{\mathcal{A}^G_{t^+}-\mathcal{A}^G_{t^-}}_p$ over the sequence $t_i$, for suitable $\eta$)  
that for $u$ in a set of full measure in $\Omega$,
\begin{align*}
|\mathcal{A}_{t^+}(u)-\mathcal{A}_{t^-}(u)|
&\le |\mathcal{A}_{t^+}(u)-\mathcal{A}^G_{t^+}(u)|
+ |\mathcal{A}^G_{t^+}(u)-\mathcal{A}^G_{t^-}(u)|
+ |\mathcal{A}^G_{t^-}(u)-\mathcal{A}_{t^-}(u)|\\
&\ll_{\phi,u}  \; e^{-\delta t^+/2}\rho(H_{t^+})+(e^{-p\delta  t/4})^\theta\rho(H_{t^--3c\vre})+
e^{-\delta t^-/2}\rho(H_{t^-})\\
&\ll_{\phi,u} \; e^{-\delta't}\rho(H_t)  
\end{align*}
with some $\delta'>0$, and hence for $x=\mathsf{p}_X(u)$,
\begin{equation}\label{eq:hhh2}
\left|\pi_X(\Lambda_{t})\phi(x)-\pi_X(\Lambda_{t_i})\phi(x) \right|
\ll_{\phi, x} e^{-\delta'  t}\rho(H_t).
\end{equation}
Also, it follows from Lemma \ref{c:reg}(i) combined with {\rm (HA2)} that for almost every $x\in X$,
\begin{equation}\label{eq:hhh1}
\left|\pi_X(\Lambda^G_{t})\phi(x)-\pi_X(\Lambda^G_{t_i})\phi(x) \right|
\ll_{\phi, x} (e^{-p\delta  t/4})^\theta\rho(H_t).
\end{equation}

Finally, combining (\ref{eq:hhh2}),~(\ref{eq:hhh0}),~(\ref{eq:hhh1})
we conclude that for $x$ in a set of full measure in $D$,
\begin{align*}
\left|\pi_X(\Lambda_{t})\phi(x)-\pi_X(\Lambda^G_{t})\phi(x) \right|
\le&
\left|\pi_X(\Lambda_{t})\phi(x)-\pi_X(\Lambda_{t_i})\phi(x) \right|\\
&+\left|\pi_X(\Lambda_{t_i})\phi(x)-\pi_X(\Lambda^G_{t_i})\phi(x) \right|\\
&+\left|\pi_X(\Lambda^G_{t})\phi(x)-\pi_X(\Lambda^G_{t_i})\phi(x) \right|\\
\ll_{\phi,x} &\; e^{-\delta'' t}\rho(H_t)
\end{align*}
with some $\delta''>0$. 

It is clear from the foregoing proof that as $x$ varies over the compact domain $D$, the constant $C(\phi,x)$ implied in the last estimate 
satisfies the integrability properties stated in the theorem, namely $\norm{C(\phi,\cdot)}_{L^p(D)}\le  C_{p,q}\norm{\phi}_{L_l^q(D)}$. 
This completes the proof of Theorem \ref{th_dual_pointwise}.
\end{proof}

\section{Volume estimates}\label{sec:volume}

\subsection{Volume asymptotics on algebraic varieties}
Let $G\subset \hbox{SL}_d(\mathbb{R})$ be a real almost algebraic group, and 
$H$ a noncompact almost algebraic subgroup of $G$. 
Let $\rho$ denote a left Haar measure on $H$ and $m$ a left Haar measure on $G$.
We fix a non-negative proper homogeneous polynomial $P$ on $\hbox{Mat}_d(\mathbb{R})$, and consider the family of sets
\begin{equation}\label{eq:hh_t}
H_t:=\{h\in H:\, \log P(h)\le t\}.
\end{equation}
The aim of this section is to discuss the properties of the sets $H_t$ and their volumes.
In order to prove our main results stated in Section 1,
we will also need to consider more general families of sets defined by
\begin{equation}\label{eq:hhh_t}
H_t[g_1,g_2]:=\{h\in H:\, \log P(g_1^{-1}hg_2)\le t\}
\end{equation}
for $g_1,g_2\in G$. 

We recall the results \cite[Th.7.17--7.18]{gn}.
While these results were stated for algebraic sets, the proof,
which is based on resolution of singularities, applies
to semialgebraic sets as well and, in particular, to the almost algebraic group $H$.

\begin{thm}[\cite{gn}] \label{th:book}
\begin{enumerate}
\item[(i)] {\rm (volume asymptotics)}
There exist $a\in\mathbb{Q}_{\ge 0}$, $b\in\mathbb{Z}_{\ge 0}$, and $\delta_0>0$
such that 
$$
\rho(H_t)=e^{a t}\left(\sum_{i=0}^b c_i t^i \right) +O\left(e^{(a-\delta_0)t}\right)
$$
for all $t\ge t_0$, where $c_b>0$.

\item[(ii)] {\rm (volume regularity)} 
There exist $c,\theta>0$ such that the estimate
\begin{equation}\label{holder}
\rho(H_{t+\vre})-\rho(H_t)\le c\, \vre^\theta\rho(H_t)
\end{equation}
holds for all $t\ge t_0$ and $\vre\in (0,1)$.
\end{enumerate}
\end{thm}

We note that since we are assuming that $H$ is noncompact, 
it follows that
$\rho(H_t)\to\infty$ as $t\to\infty$, and hence $(a,b)\ne (0,0)$.

\subsection{Polynomial volume growth of the restricted sets}

We say that the group $H$ is of {\it exponential type} if $a>0$, and 
of {\it subexponential type} otherwise. 
The following lemma gives a group-theoretic characterisation of these
notions. In particular, it implies that they
do not depend on a choice of the polynomial $P$,
and on a choice of the embedding of $H$ in  $\hbox{SL}_d(\mathbb{R})$.

\begin{lem}\label{l:growth}
The group $H$ is of subexponential type if and only if its Zariski closure 
is an almost direct product of a compact subgroup and an abelian $\RR$-diagonalisable subgroup.
\end{lem}

\begin{proof}
Let $\|\cdot\|$ denote the Euclidean norm on $\hbox{Mat}_d(\mathbb{R})$.
Then since $P$ is non-negative proper and homogeneous of degree $\deg P$, it follows by compactness of the unit sphere that there exists $C>1$ such that
\begin{equation}\label{eq:compact2}
C^{-1}\, \|x\|^{\deg(P)}\le P(x)\le C\, \|x\|^{\deg(P)}
\end{equation}
for all $x\in \hbox{Mat}_d(\mathbb{R})$. Therefore, it is sufficient to prove the claim
for the sets $H_t=\{h\in H:\, \log \|h\|\le t\}$.

Since $H$ has finite index in its Zariski closure, it has exponential type if and only if
the algebraic envelop does. Hence, we may assume that $H$ is algebraic.
The group $H$ has the Levi decomposition $H=LAU$, where $L$ is a semisimple almost algebraic subgroup, 
$A$ is an abelian $\RR$-diagonalizable almost algebraic subgroup, and $U$ is a unipotent normal algebraic subgroup.
Moreover, $LA$ is an almost direct product of $L$ and $A$.
Under the product map, the left invariant measure $\rho$ on $H$ is equal (up to a constant factor)
to the product of the invariant measures on the (unimodular) factors. If $Q$ is a compact subset of $LA$
of positive measure, 
then there exists $c>0$ such that $\log \|kh\|\le \log \|h\|+c $ for all $k\in Q$ and $h\in H$.
Hence, 
$$
QU_t\subset H_{t+c}\quad\hbox{and}\quad \rho(H_{t+c})\gg \vol(U_t).
$$
This
implies that if $H$ is not of exponential type,
then $U$ is not of exponential type as well. 
Suppose that $U\ne 1$. 
The exponential map $\exp:\hbox{Lie}(U)\to U$
is a diffeomorphism and the invariant measure on $U$ is up to a constant equal to the image
under $\exp$ of the Lebesgue measure on $\hbox{Lie}(U)$. 
Since there exists $c>0$ such that $\|\exp(v)\|\le c\|v\|^d$
for all $v\in \hbox{Lie}(U)$, it follows that
$$
\vol(U_t)\ge \vol(\{v\in \hbox{Lie}(U):\, \|v\|\le (c^{-1}\, e^{t})^{1/d}\}).
$$
This gives a contradiction and shows that $U=1$.

If $L$ is not compact, then it is of exponential type
as follows from (see \cite[Sec.~7]{GW} and \cite{Mau}). Since for a compact $Q\subset A$,
there exists $c>0$ such that $L_t Q\subset H_{t+c}$.
As above, this would imply that $H$ is
of exponential type. Hence, we conclude that $L$ must be compact, which
completes the proof of the lemma.
\end{proof}
 
\subsection{The limiting density in the ergodic theorem}

Let $O_\vre$ denotes the symmetric neighborhood of identity in $G$
with respect to a Riemannian metric, and $G_t:=\{g\in G:\, \log P(g)\le t\}$.

\begin{lem}\label{l:ccc_c}
\begin{enumerate}
\item[(i)]
Given a compact subset $\Omega$ of $G$, there exists $c=c(\Omega)>0$, such that
for every $t\ge t_0$
$$
\Omega\cdot G_t\cdot \Omega\subset G_{t+c}.
$$
\item[(ii)]
There exists $c>0$ such that for every
$\vre\in (0,1)$ and $t\ge t_0$,
$$
O_\vre\cdot G_t\cdot O_\vre\subset G_{t+c\vre}.
$$
\end{enumerate}
\end{lem}

\begin{proof}
In order to prove (i), it is sufficient to show that there exists $C>0$
such that for every $b_1,b_2\in \Omega$ and $x\in \hbox{Mat}_d(\mathbb{R})$, we have
$$
P(b_1xb_2)\le C\, P(x).
$$
Since $\Omega$ is compact, 
$$
P(b_1xb_2)\ll_\Omega \left(\max_{i,j} |x_{ij}|\right)^{\deg(P)},\quad x\in \hbox{Mat}_d(\mathbb{R}),
$$
and since $P$ is non-negative, proper and homogeneous, it follows by compactness that
\begin{equation}\label{eq:norm2}
\left(\max_{i,j} |x_{ij}|\right)^{\deg(P)} \ll P(x),\quad x\in \hbox{Mat}_d(\mathbb{R}).
\end{equation}
This completes the proof of (i).

To prove (ii), we observe that
for every $b_1,b_2\in O_\vre$ and $x\in \hbox{Mat}_d(\mathbb{R})$,
$$
P(b_1xb_2)-P(x)\ll \vre \left(\max_{i,j} |x_{ij}|\right)^{\deg(P)}\ll \vre\, P(x).
$$
This implies (ii).
\end{proof}

The following lemma follows immediately from Lemma \ref{l:ccc_c}.

\begin{lem}\label{l:ccc}
\begin{enumerate}
\item[(i)]
Given a compact subset $\Omega$ of $G$, there exists $c=c(\Omega)>0$, such that
for every $g_1,g_2\in G$, $b_1,b_2\in \Omega$, and $t\ge t_0$
$$
H_{t-c}[g_1,g_2]\subset H_t[g_1b_1,g_2b_2]\subset H_{t+c}[g_1,g_2].
$$
\item[(ii)]
There exists $c>0$ such that for every
$g_1,g_2\in G$, $b_1,b_2\in O_\vre$ with $\vre\in (0,1)$, and $t\ge t_0$,
$$
H_{t-c\vre}[g_1,g_2]\subset H_t[g_1b_1,g_2b_2]\subset H_{t+c\vre}[g_1,g_2].
$$
\end{enumerate}
\end{lem}

The next proposition justifies existence of the limit measures $\nu_x$
defined in (\ref{eq:nu_x}).

\begin{prop}\label{p:alpha}
For every $g_1,g_2\in G$, the limit
$$
\Theta(g_1,g_2):=\lim_{t\to\infty} \frac{\rho(H_t[g_1,g_2])}{\rho(H_t)}
$$
exists. Moreover, the function $\Theta$ is positive and continuous.
\end{prop}

\begin{proof}
It follows from Theorem \ref{th:book}(i) applied to the homogeneous polynomial $P(g_1^{-1}x g_2)$  that 
\begin{equation}\label{eq:b_book}
\rho(H_t[g_1,g_2])\sim c(g_1,g_2)e^{a(g_1,g_2)t}t^{b(g_1,g_2)} \quad\hbox{as $t\to\infty$,}
\end{equation}
for some $c(g_1,g_2)>0$, $a(g_1,g_2)\in\mathbb{Q}_{\ge 0}$, $b(g_1,g_2)\in \mathbb{Z}_{\ge 0}$.
Moreover, it follows from Lemma \ref{l:ccc}(i) that $a(g_1,g_2)$ and $b(g_1,g_2)$
are independent of $g_1,g_2\in G$. This implies that the limit exists and is positive.

To prove continuity, we observe that for every $b_1,b_2\in O_\vre$,
\begin{align*}
&|\Theta(g_1b_1,g_2b_2)-\Theta(g_1,g_2)|
= \lim_{t\to\infty} \frac{|\rho(H_t[g_1b_1,g_2b_2])-\rho(H_t[g_1,g_2])|}{\rho(H_t)}\\
=& \max \set{ \lim_{t\to\infty} \frac{\rho(H_t[g_1b_1,g_2b_2])-\rho(H_t[g_1,g_2])}{\rho(H_t)} \,,\,
\lim_{t\to\infty} \frac{\rho(H_t[g_1,g_2])-\rho(H_t[g_1b_1,g_2b_2])}{\rho(H_t)}},
\end{align*}
By Lemma \ref{l:ccc}(ii) and \eqref{eq:b_book},
\begin{align*}
\lim_{t\to\infty} \frac{\rho(H_t[g_1b_1,g_2b_2])-\rho(H_t[g_1,g_2])}{\rho(H_t)}
&\le \lim_{t\to\infty} \frac{\rho(H_{t+c\vre} [g_1,g_2])-\rho(H_t[g_1,g_2])}{\rho(H_t)}\\
&= \frac{c(g_1,g_2)}{c(e,e)}\left(e^{a(c\vre)}-1\right)\ll \vre.
\end{align*}
Similarly,
\begin{align*} 
\lim_{t\to\infty} \frac{\rho(H_t[g_1,g_2])-\rho(H_t[g_1b_1,g_2b_2])}{\rho(H_t)}
\le \frac{c(g_1,g_2)}{c(e,e)}\left(1- e^{-a(c\vre)}\right)\ll \vre.
\end{align*}
This proves continuity.
\end{proof}

In the case of groups of subexponential type, we have the following asymptotic formula for the measure of the sets 
$H_t[g_1,g_2]$, which is independent of $g_1,g_2$, generalizing  Theorem \ref{th:book}(i).

\begin{prop}\label{p:alpha_sub}
Let $H\subset G$ be of subexponential type.
Then there exists $c_b>0$ and $b\in\mathbb{N}$ such that
uniformly over $g_1,g_2$ in compact subsets of $G$,
$$
\rho(H_t[g_1,g_2])=c_b\, t^b+O(t^{b-1})
$$
for all $t\ge t_0$.
\end{prop}

\begin{proof}
Let $\|\cdot\|$ be a Euclidean norm on $\hbox{\rm Mat}_d(\mathbb{R})$,
and $H_t'[g_1,g_2]$ the corresponding balls in $H$. It follows from
(\ref{eq:compact2}) that there exists $c>0$ such that for every
$g_1,g_2\in G$ and $t\ge t_0$,
$$
H'_{t-c}[g_1,g_2] \subset H_t[g_1,g_2]\subset H'_{t+c}[g_1,g_2].
$$
Therefore, it is sufficient to prove the claim of the lemma for a Euclidean norm.

By Lemma \ref{l:ccc}(i), there exists $c>0$, uniform over $g_1,g_2$ in compact sets,
such that for every $t\ge t_0$,
$$
H_{t-c}\subset H_t[g_1,g_2]\subset H_{t+c}.
$$
Since by Theorem \ref{th:book}(i), we have
$$
\rho(H_t)=c_b\, t^b+O(t^{b-1}),
$$
this implies the proposition.
\end{proof}

\begin{rem} \label{r:gen_norm} {\rm
Proposition \ref{p:alpha_sub} implies that for the groups of subexponential type,
the regularity properties of sets $H_t[g_1,g_2]$ are straightforward to establish.
In particular, the function $\Theta$ in Proposition \ref{p:alpha} is constant,
and the claim of Proposition \ref{p:reg} below follows directly from Proposition \ref{p:alpha_sub}.

We also note that the argument of Proposition \ref{p:alpha_sub} applies to sets
defined by $P(x)=\|x\|$ where $\|\cdot \|$ is a general norm on $\hbox{Mat}_d(\mathbb{R})$, not necessarily a {\it polynomial} one. 
This implies that Theorem \ref{th:main1} holds for the averages along the sets
$\Gamma_t:=\{\gamma\in\Gamma:\, \log\|\gamma\|\le t\}$ defined by general norms.
}\end{rem}

\subsection{Volume regularity properties : average admissibility} 

The following proposition gives an averaged  version of Theorem \ref{th:book}(ii). 

\begin{prop}\label{p:reg}
Let $1\le r<\infty$ and $\Omega$ be a compact subset of $G$.
\begin{enumerate}
\item[(i)] For every $u\in G$,
there exist $c=c(u,\Omega,r)>0$ and $\theta=\theta(u,\Omega,r)>0$ such that the estimate 
$$
\left(\int_{\Omega} (\rho(H_{t+\vre}[u,v])- \rho(H_t[u,v]))^r\,dm(v)\right)^{1/r}\le  c\,\vre^\theta \rho(H_t)
$$ 
holds for all $t\ge t_0$ and $\vre\in (0,1)$.
\item[(ii)]
There exist $c=c(\Omega,r)>0$ and $\theta=\theta(\Omega,r)>0$ such that the estimate 
$$
\left(\int_{\Omega\times \Omega} (\rho(H_{t+\vre}[u,v])- \rho(H_t[u,v]))^r\,dm(u)dm(v)\right)^{1/r}\le  c\,\vre^\theta \rho(H_t)
$$ 
holds for all $t\ge t_0$ and $\vre\in (0,1)$.
\end{enumerate}
\end{prop}

\begin{proof}
Since the proofs of (i) and (ii) is very similar, we only present the proof of (i).

Without loss of generality, we may assume that $\Omega$ is bounded semialgebraic set.
We first prove the assertion when $r=n$ is an integer.  
Let 
$$
v(t):=\int_{\Omega} \rho(H_{t}[u,v])^n\,dm(v).
$$
We claim that 
for some $c_1,\theta_1>0$ this function satisfies the estimate
\begin{equation}\label{eq:vvv}
v(t+\vre)-v(t)\le c_1\,\vre^{\theta_1}\, v(t)
\end{equation}
for all $t\ge t_0$ and $\vre\in (0,1)$.
We note that by Lemma \ref{l:ccc}(i) and Theorem \ref{th:book}(i),
$v(t)\ll \rho(H_t)^n$.
Therefore, using the inequality $(a-b)^n\le a^n-b^n$ with  $a\ge b\ge 0$, 
we conclude that (\ref{eq:vvv}) implies (ii) with $p=n$.

Now to prove (\ref{eq:vvv}), we observe that
\begin{align*}
v(t)&=\int_{\Omega\times H^n} \chi_{\{\log P(u^{-1}h_1v)\le t,\cdots, \log P(u^{-1}h_nv)\le t\}}\,
dm(v)d\rho(h_1)\cdots d\rho(h_n) \\
&=\int_{\Omega\times H^n} \chi_{\{\log \Psi(u,v,h_1,\ldots,h_n)\le t\}}\,
dm(v)d\rho(h_1)\ldots d\rho(h_n),
\end{align*}
where $\Psi(v,h_1,\ldots,h_n)=\max\{P(u^{-1}h_1v),\ldots, P(u^{-1}h_nv)\}$ is a semialgebraic function
on $\Omega\times H^n$. 
Let $\overline{H}$ denote the projective closure of $H$.
Then $\Psi^{-1}$ is  
a semialgebraic function on $\Omega\times \overline{H}^n$
that  vanishes on the complement of $\Omega\times H^n$.
Now we can apply the argument of \cite[Theorems~7.17]{gn} to deduce \eqref{eq:vvv}.

To prove (i) for general $r\ge 1$, we observe that H\"older's inequality with
$q=(\lfloor r\rfloor +1)/r$ gives
\begin{align*}
&\int_{\Omega} (\rho(H_{t+\vre}[u,v])- \rho(H_t[u,v]))^r\,dm(v)\\
\le &\;m(\Omega)^{1-1/q} \left(\int_{\Omega\times \Omega} (\rho(H_{t+\vre}[u,v])- \rho(H_t[u,v]))^{rq}\,dm(v)\right)^{1/q}.
\end{align*}
Hence, the general estimate follows from the case when $r$ is an integer. 
\end{proof}

\begin{rem} {\rm Let us note that the quality of the estimate stated in Proposition \ref{p:reg}(i) is a key ingredient controlling the quality of the mean and pointwise ergodic theorems. 
Obtaining results of the quality stated in Theorem \ref{th:main5} hinges upon establishing an estimate
which is uniform in the rate $\theta$ and the constant $c$ as $u$ varies in compact sets in $G$. Since
${\sf s}(x)$ and ${\sf s}(y)$ are not in $H$, the integral in Proposition \ref{p:reg}(i) depends
non-trivially on ${\sf s}(x)$ and ${\sf s}(y)$,
so there is no obvious way to exploit invariance of the measure. Rather, the proofs of Proposition \ref{p:reg}(i) and of  Proposition \ref{p:infinity} below apply resolution of singularities to
the parametric family of polynomial maps $h\mapsto P({\sf s}(x)^{-1} h {\sf s}(y)).$
The estimate produced as a result of this procedure for a given polynomial in the family  depends on its coefficients, 
which in turn depend non-trivially on ${\sf s}(x)$ and ${\sf s}(y)$. We  will establish uniform estimates for the  parametric family  of polynomials that arises when $P$ is a norm, and this accounts for the appearance of this assumption in Theorem \ref{th:main5}. 
}
\end{rem}

\subsection{Volume regularity properties : subanalytic functions} 

The following proposition refines Theorem \ref{th:book}(i) and generalizes the discussion to the case of a general subanalytic function. 

\begin{prop}\label{p:infinity}
Let $a\in\mathbb{Q}_{\ge 0}$ and $b\in\mathbb{Z}_{\ge 0}$ be as in Theorem \ref{th:book}(i).
Then for every nonnegative continuous subanalytic function $\phi\ne 0$
with compact support and $x\in X$, 
there exists $\delta>0$ such that for all $t\ge t_0$,
\begin{equation}\label{eq:asymptt}
\int_{G_t} \phi(x g)\, dm(g)=e^{a t}\left(\sum_{i=0}^b c_i(\phi,x) t^i \right) +O_{\phi,x}\left(e^{(a-\delta)t}\right),
\end{equation}
where $c_b(\phi,x)>0$.
\end{prop}

\begin{proof}
Decomposing $\phi$ into a sum of subanalytic functions with compact supports, we reduce
the proof to the case when $\supp(\phi)$ is contained a compact semianalytic set $D$,
and there exists a section $\mathsf{s}: X\to G$ of the factor map
$\mathsf{p}_X:G\to H\backslash G=X$ such that $\mathsf{s}|_{D}$ is analytic.
It follows from invariance of $m$ and \eqref{eq:measure} that
\begin{align*}
v(t)&:=\int_{G_t} \phi(x g)\, dm(g)
=
\int_{G_t} \phi(\mathsf{p}_X(\mathsf{s}(x)g))\, dm(g)\\
&=\int_{(y,h):\, \mathsf{s}(x)^{-1}h\mathsf{s}(y)\in G_t} \phi(\mathsf{p}_X(h\mathsf{s}(y)))\, d\rho(h)d\xi(y)\\
&=\int_D\phi(y)\rho(H_t[\mathsf{s}(x),\mathsf{s}(y)])\,d\xi(y).
\end{align*}
We also note that it follows from  \eqref{eq:measure} that the measure $\xi$ is given by an 
analytic differential form on $D$. We consider the transform of $v(t)$:
$$
f(s)=\int_0^\infty  t^{-s} v(\log t)dt.
$$
Note that it follows from Lemma \ref{l:ccc}(i) and Theorem \ref{th:book}(i) that for some $c>0$,
\begin{equation}\label{eq:uuper}
v(t)\le \|\phi\|_\infty \xi(D)\rho(H_{t+c})\ll e^{at}t^b
\end{equation}
and, in particular, the integral $f(s)$ converges
when  $\hbox{Re}(s)$ is sufficiently large.
In this region, we have
\begin{align*}
f(s)&=\int_D\phi(y)\left(\int_0^\infty t^{-s} \rho(H_{\log t}[\mathsf{s}(x),\mathsf{s}(y)])dt\right) \,d\xi(y)\\
&=(s-1)^{-1} \int_D\phi(y) \left(\int_H P(\mathsf{s}(x)^{-1}h\mathsf{s}(y))^{-s+1} d\rho(h)\right)\,d\xi(y) \\
&=(s-1)^{-1} \int_{D\times H} \phi(y) P(\mathsf{s}(x)^{-1}h\mathsf{s}(y))^{-s+1} d\xi(y)d\rho(h). 
\end{align*}
We observe that the map $(y\times h)\mapsto P(\mathsf{s}(x)^{-1}h\mathsf{s}(y))^{-1}$ extends
to a semianalytic function $D\times \overline{H}$, where $\overline{H}$ denotes the projective
closure of $H$, and vanishes on $D\times (\overline{H}-H)$.
Now to finish the proof, we can apply the argument of \cite[Theorems~7.17]{gn},
but instead of the Hironaka resolution of singularities, we use the rectilinearization 
of subanalytic functions \cite[Theorem~2.7]{par}.
Hence, we conclude that \eqref{eq:asymptt} holds, but
a priori the parameters $a$ and $b$ in \eqref{eq:asymptt} may depend on $\phi$ and $x$.
However, we observe that since $\phi$ is continuous, there exists bounded open $O\subset X$ and 
$m_0>0$ such that $\phi(y)\ge m_0$ for all $y\in O$. Therefore, we deduce from Lemma \ref{l:ccc}(i)
and Theorem \ref{th:book}(i) that for some $c>0$, 
$$
v(t)=\int_X\phi(y)\rho(H_t[\mathsf{s}(x),\mathsf{s}(y)])\,d\xi(y)
\ge m_0\xi(O)\rho(H_{t-c})\gg e^{at}t^b.
$$
Combining this estimate with \eqref{eq:uuper}, we conclude that $a$ and $b$
are independent of $\phi$ and $x$. 
\end{proof} 

\begin{rem} {\rm
Let $G=\prod_{i=1}^l G_i$ where $G_i\subset \hbox{SL}_{d_i}(\mathbb{R})$ be a real
almost algebraic group, and let $P_i$'s be non-negative proper homogeneous polynomials on $\hbox{Mat}_{d_i}(\mathbb{R})$.
For a non-compact real almost algebraic subgroup $H$ of $G$, we consider the sets
$$
H_t:=\{h\in H:\, \log(P_1(h_1)\cdots P_l(h_l))\le t\},
$$
which appear in number-theoretic applications as height function 
Then the arguments developed in this section apply to such sets, and the theorems
established in the Introduction hold for averages supported on the sets
$$
\Gamma_t:=\{\gamma\in \Gamma:\, \log(P_1(\gamma_1)\cdots P_l(\gamma_l))\le t\}.
$$
}\end{rem}

\section{Ergodic theory of algebraic subgroups}\label{sec:alg}

As noted in the introduction, the possibility of applying the duality principle to establish ergodic theorems for (properly normalized) sampling operators for $\Gamma$ acting on $H\setminus G$ depends on the validity 
of ergodic theorems for  averages on $H$ acting on $G/\Gamma$. We therefore turn now to consider the ergodic theory of algebraic subgroups, namely to consider an algebraic group $G$ acting by measure-preserving transformation on a probability space $Y$,  and to the study of the action restricted to  an algebraic subgroup $H$. In the discussion of this problem it is natural to consider  spaces $Y$ more general than just $G/\Gamma$ (so that $G$ is no longer transitive), and also general closed subgroups of $G$ which are not necessarily algebraic. 

We will first consider in \S \ref{sec:erg-poly} the case where the volume growth of the restricted sets $H_t$ is subexponential, where one can apply the traditional arguments regarding regular F\o lner families to obtain ergodic theorems for actions on a general probability space $Y$.

We will then assume that $H$ is a subgroup of a connected semisimple Lie group $G$, which acts on a probability space $Y$ with a strong spectral gap, namely such that each of the  simple components of $G$ has a spectral gap. Under this assumption we will prove quantitative mean, maximal and pointwise ergodic theorems, of two kinds. In  \S \ref{sec:erg-sobolev}  and \S \ref{sec:pointwise-sobolev}  we assume that  $Y$ is a manifold and establish quantitative mean ergodic theorem in Sobolev spaces, for a  general  closed subgroup, including of course  connected algebraic subgroups $H$. When the volume growth of $H_t$ is exponential, the rate of convergence in the mean ergodic theorem will be exponentially fast, and we also establish then an exponentially fast pointwise ergodic theorem for bounded functions. 

In \S  \ref{sec: Iwasaw} and \S \ref{sec:erg-Iwasaw} we turn from Sobolev spaces to Lebesgue spaces. First, in \S \ref{sec: Iwasaw} we will establish 
spectral estimates for ergodic averages in actions of general non-amenable closed subgroups which are non-amenably embedded --- a term we will define and explain there. In \S \ref{sec:erg-Iwasaw} we use these estimates  and establish exponentially fast maximal, mean and pointwise ergodic theorems for these averages acting in Lebesgue spaces. The proof will in fact only depend on a mild regularity assumption on the averages, an assumption that is always satisfied when the subgroup is algebraic and the averages are defined by a homogeneous proper non-negative polynomial.  

Let $H\subset G \subset \hbox{SL}_d(\mathbb{R})$ be closed 
noncompact subgroup  with a left Haar measure $\rho$.
For an arbitrary measure-preserving action of $H$ on a probability space $(Y,\mu)$, 
we consider the family of averaging operators $\pi_Y(\beta_t):L^p(Y)\to L^p(Y)$
defined by 
$$
\pi_Y(\beta_t)F(y)=\frac{1}{\rho(H_t)} \int_{H_t} F(h^{-1} y)\, d\rho(h),\quad F\in L^p(Y),
$$
where $H_t$ are the balls associated with a proper non-negative homogeneous polynomial, as defined in Section \ref{sec:volume}.
We will also consider the sets $H_t[g_1,g_2]\subset H$, defined via the embedding $H\subset G$
(see \eqref{eq:ss_t}).

\subsection{Ergodic theorems in the presence  of subexponential  growth}\label{sec:erg-poly}


\begin{thm}\label{th:erg_sub}
Keeping the notation of the preceding paragraph,  assume that $H$ and $G$ are almost algebraic, the sets $H_t\subset H$  have subexponential volume growth and that $H$ 
acts on an arbitrary probability measure space $(Y,\mu)$ preserving the measure. Then the averages $\pi_Y(\beta_t)$ satisfy the following. 
\begin{enumerate}
\item[(i)] {\rm Weak-type $(1,1)$-maximal inequality. }
For every $F\in L^1(Y)$ and $\delta>0$,
$$
\mu\left(\left\{\sup_{t\ge t_0} |\pi_Y(\beta_t)F|>\delta \right\}\right)\ll \frac{\|F\|_{L^1(Y)}}{\delta}.
$$
\item[(ii)] {\rm Strong maximal inequality.}
For $1<p\le \infty$ and every $F\in L^p(Y)$,
$$
\left\|\sup_{t\ge t_0} |\pi_Y(\beta_t)F|\right\|_{L^p(Y)}\ll_p \|F\|_{L^p(Y)}.
$$
\item[(iii)] {\rm Mean and pointwise ergodic theorem.}
For every $1\le p<\infty$ and $F\in L^p(Y)$,  
the averages $\pi_Y(\beta_t)F$ converges almost everywhere and in $L^p$-norm as $t\to\infty$. In the ergodic case, the limit is $\int_Y Fd\mu$. 
\end{enumerate}
The same results hold without change for each of the families $\pi_Y(\beta_t^{g_1,g_2})$  supported on $H_t[g_1,g_2]$, for any $g_1,g_2\in G$. 
\end{thm}

Before starting the proof, let us observe that by Lemma \ref{l:growth},  whenever the sets $H_t$ have subexponential volume growth, the group $H$ is in fact 
amenable, so that we can use classical methods from the ergodic theory
of amenable groups (see, for instance, \cite{n1} for a recent survey, and \cite{AAB} for a detailed discussion).
We say that a family of subsets $\{B_t\}$ of $H$ is {\it asymptotically invariant, or uniform F\o lner} 
if for every compact subset $Q$ of $H$,
\begin{equation}\label{eq:folner}
\lim_{t\to\infty} \frac{\rho(Q B_t \triangle B_t)}{\rho(B_t)}=0.
\end{equation}
We say that a family $\{B_t\}$ is {\it regular} if
\begin{equation}\label{eq:regular}
\rho(B_t\cdot B_t^{-1})\ll \rho(B_t).
\end{equation}

\begin{prop}\label{p:folner}
Assume that the sets $H_t\subset H$ have subexponential volume growth.
Then the family $\{H_t\}$ is uniform F\o lner and regular.
\end{prop}

\begin{proof}
Since $H$ is of subexponential type, by Theorem \ref{th:book}(i),
\begin{equation}\label{eq:sub_exp}
\rho(H_t)\sim c_b\, t^b\quad\hbox{ as $t\to \infty$}
\end{equation} 
for some $c_b>0$ and $b\in \mathbb{N}$.

Without loss of generality, we may assume that the compact set $Q$ in (\ref{eq:folner}) 
contains the identity. 
There exists $c>0$ such that $H_t Q\subset H_{t+c}$, so that
$$
\limsup_{t\to\infty} \frac{\rho(Q H_t \triangle H_t)}{\rho(H_t)}\le
\limsup_{t\to\infty} \frac{\rho(H_{t+c}- H_t)}{\rho(H_t)}=0
$$
by \eqref{eq:sub_exp}. This proves that $\{H_t\}$ is uniform F\o lner. Clearly by property (CA2) namely coarse admissibility of $H_t$, the families $H_t[g_1,g_2]$ are also F\o lner. 

To prove the second claim, we observe that by 
for every $h,h'\in H$,
$$
P(h\cdot h')\ll \left(\max_{ij} |h_{ij}|\right)^{\deg(P)}\left(\max_{ij} |h'_{ij}|\right)^{\deg(P)}
\ll P(h)P(h'),
$$
and
$$
P(h^{-1})\ll \left(\max_{ij} |h_{ij}|\right)^{(d-1)\deg(P)}\ll P(h)^{d-1}.
$$
Therefore, there exists $c>0$ such that
$$
H_t\cdot H_t^{-1}\subset H_{dt+\log(2c)}.
$$
Hence, it follows from \eqref{eq:sub_exp} that $\{H_t\}$ is a regular F\o lner family, and similarly the same holds for each family $H_t[g_1,g_2]$.
\end{proof}

\begin{proof}[Proof of Theorem \ref{th:erg_sub}]
Since the family of sets $H_t$ is uniform F\o lner and regular.
The theorem is a partial case of \cite[Th.~6.6]{n1}.
\end{proof}

\subsection{Quantitative mean ergodic theorem in Sobolev spaces}\label{sec:erg-sobolev}

 In the present subsection we turn to establish a quantitative mean ergodic theorem for the operators $\pi_Y(\beta_t)$ in Sobolev spaces. This result will be used \S \ref{sec:pointwise-sobolev} below in the proof of the quantitative pointwise ergodic theorem in Sobolev spaces. We begin by considering a general closed subgroup $H$, contained in a semisimple Lie group $G$ acting smoothly on a manifold $Y$ with a strong spectral gap, namely each of  the simple factors of $G$ has a spectral gap. This  condition is of course necessary in order to obtain norm decay along subgroups.  
\begin{thm} {\rm Quantitative mean ergodic theorem in Sobolev spaces.}\label{th:mean_quant_H}
Assume that 
\begin{itemize}
\item the group $H$ is an arbitrary closed subgroup of a connected semisimple Lie group $G$
with finite centre, and $G$ acts on a manifold $Y$ preserving a probability measure $\mu$,
\item the representation of every simple factor of $G$ on $L_0^2(Y)$ is isolated from the trivial representation.
\end{itemize}
Then there exist $l\in\mathbb{N}$ and $t_0>0$ such that for every $1<p<\infty$, 
a compact domain $B$ of $Y$, and $F\in L^p_l(B)$, the following estimate holds with $\kappa_p>0$, for all $t\ge t_0$,
$$
\left\|\pi_Y(\beta_t)F-\int_Y F\, d\mu\right\|_{L^p(Y)}\ll_{p,B} \rho(H_t)^{-\kappa_p}\,
\|F\|_{L^p_l(B)}
$$
Furthermore, the rate of convergence applies to each famiy $\pi_Y(\beta_t^{g_1,g_2})$  supported on the
sets $H_t[g_1,g_2]$, uniformly when $g_1,g_2\in G$ vary in compact sets in $G$. 
\end{thm}

We remark that the mean ergodic Theorem \ref{th:mean_quant_H} above,  in the case where the  subgroup $H$ is either amenable, or non-amenable but amenably embedded in $G$, is the best possible result of its kind. Indeed, in these cases, while quantitative mean ergodic theorems will presently be shown to hold in Sobolev spaces,  they definitely do not hold in Lebesgue spaces : the norm of the operators $\pi_Y^0(\beta_t)|_{L^2(Y)\to L^2(Y)}$ is either identically $1$ (when $H$ is amenable) or converges to $1$ (when $H$ is amenably embedded). 

 As noted already, under the assumption that $H$ is non-amenable and not embedded amenably, we will derive in Theorem \ref{thm:non-amenable-erg} in \S \ref{sec:erg-Iwasaw} below the stronger exponentially fast mean and pointwise ergodic theorems in Lebesgue space $L^p$, $ 1 < p < \infty$, going beyond the results in Sobolev spaces.

\begin{proof}
We first consider the case when $p=2$. 
By \cite[2.4.3]{KM1}, there exist $l\in\mathbb{N}$ and $\kappa>0$ such that
for every $F_1,F_2\in L^2_l(B)$ with zero integrals, we have
$$
|\left<\pi_Y(g)F_1,F_2\right>|\ll \|g\|^{-\kappa} \|F_1\|_{2,l}\|F_2\|_{2,l},
$$
where $\|\cdot\|_{2,l}$ denotes the Sobolev norm as defined in \cite{KM1}.
While these Sobolev norms are different from the Sobolev norms
$\|\cdot\|_{L^2_l(B)}$ that we use in our paper, 
it is clear that $\|\cdot\|_{2,l}\ll_B \|\cdot\|_{L^2_l(B)}$
on $L^2_l(B)$.

It follows from the above estimate
that for every $F\in L^2_l(B)$ with zero integral,
\begin{align*}
\|\pi_Y(\beta_t)F\|_{L^2(Y)}^2
= &\frac{1}{\rho(H_t)^2} \int_{H_t\times H_t}\left<\pi_Y(h_1)F,\pi_Y(h_2)F\right>\, d\rho(h_1)d\rho(h_2)\\
= &\frac{1}{\rho(H_t)^2} \int_{H_t\times H_t}\left<\pi_Y(h_2^{-1}h_1)F,F\right>\, d\rho(h_1)d\rho(h_2)\\
\le&\frac{1}{\rho(H_t)^2} \int_{H_t\times H_t}P(h_2^{-1}h_1)^{-\kappa} \|F\|_{L^2_l(B)}^2\, d\rho(h_1)d\rho(h_2).
\end{align*}
We observe that since the polynomial $P$ is proper on $\hbox{Mat}_d(\mathbb{R})$, we have
$$
\inf\{P(h):\, h\in H\}>0,
$$
and there exists $a>0$ such that
\begin{align*}
\rho(H_s)\ll e^{as}\,.
\end{align*}
This simple estimate follows for an algebraic subgroup $H$  from the much sharper asymptotic result stated in Theorem \ref{th:book}(i), but holds true for any closed subgroup $H$ whatsoever. 

Therefore, we deduce that
\begin{align*}
&\int_{H_t\times H_t}P(h_2^{-1}h_1)^{-\kappa} \, d\rho(h_1)d\rho(h_2)\\
=&\int_{(h_1,h_2)\in H_t\times H_t: P(h_2^{-1}h_1)< e^s} P(h_2^{-1}h_1)^{-\kappa} \,
d\rho(h_1)d\rho(h_2)\\
&+\int_{(h_1,h_2)\in H_t\times H_t: P(h_2^{-1}h_1)\ge e^s} P(h_2^{-1}h_1)^{-\kappa} \, d\rho(h_1)d\rho(h_2)\\
\le &\int_{(h,h_2)\in H\times H_t: P(h)< e^s} P(h)^{-\kappa} \, d\rho(h)d\rho(h_2)
+e^{-\kappa s} \rho(H_t)^2\\
\ll&\, \rho(\{h\in H:\,P(h)<e^s\})\rho(H_t)+ e^{-\kappa s} \rho(H_t)^2\\
\ll&\, e^{as}\rho(H_t)+ e^{-\kappa s} \rho(H_t)^2.
\end{align*}
Now taking $s$ such that $e^s=\rho(H_t)^{1/(a+\kappa)}$,
we conclude that for all $t\ge t_0$,
$$
\|\pi_Y(\beta_t)F\|_{L^2(B)}^2\ll 
\rho(H_t)^{-\kappa/(a+\kappa)} \|F\|_{L^2_l(B)}^2,\quad F\in L^2_l(B).
$$
This proves the theorem for $p=2$, for the family $\pi_Y(\beta_t)$. The case of the family $\pi_Y(\beta_t^{g_1,g_2})$,  as well as the uniformity as $g_1,g_2$ vary in compact sets in $G$ 
is similar. 
  
In order to complete the proof in general, we observe that the linear operator
$$
\mathcal{A}_t(F)=\pi_Y(\beta_t)F-\int_Y F\, d\mu
$$
satisfies the estimates
\begin{align*}
&\|\mathcal{A}_t\|_{L^{1}_{l}(B)\to L^{1}(Y)}\ll 1,\\
&\|\mathcal{A}_t\|_{L^{2}_{l}(B)\to L^{2}(Y)}\ll \rho(H_t)^{-\kappa/(a+\kappa)},\\
&\|\mathcal{A}_t\|_{L^{\infty}_{l}(B)\to L^{\infty}(Y)}\ll 1 .
\end{align*}
Hence, the general case follows from Theorem \ref{th:interpolation} below.
\end{proof}

\begin{thm}[\cite{bs}]\label{th:interpolation}
Let $1\le p_1\le p_2\le \infty$, $l\in \mathbb{N}$, $B$ a compact domain in $Y$, and let
$$
\mathcal{A}:L^{p_1}_{l}(B)+L^{p_2}_{l}(B)\to L^{p_1}(Y)+L^{p_2}(Y)
$$
be a linear operator such that
\begin{align*}
&\mathcal{A}(L^{p_1}_{l}(B))\subset L^{p_1}(Y),\quad 
\|\mathcal{A}\|_{L^{p_1}_{l}(B)\to L^{p_1}(Y)}\le M_1,\\
&\mathcal{A}(L^{p_2}_{l}(B))\subset L^{p_2}(Y),\quad
\|\mathcal{A}\|_{L^{p_2}_{l}(B)\to L^{p_2}(Y)}\le M_2.
\end{align*}
Then for every $p$ such that $\frac{1}{p}=\frac{1-\theta}{p_1}+ \frac{\theta}{p_2}$
with $\theta\in (0,1)$,
\begin{align*}
\mathcal{A}(L^{p}_{l}(B))\subset L^{p}(Y),\quad
\|\mathcal{A}\|_{L^{p}_{l}(B)\to L^{p}(Y)}\ll M_1^{1-\theta}M_2^\theta.
\end{align*}
\end{thm}

\begin{proof} 
This theorem is a consequence for the results on interpolation
of linear operators that can be found in \cite[Ch.~5]{bs}.
Indeed, using a partition of unity, one can reduce the proof
to the case when the support of $\phi\in L^{p}_{l}(B)$ is contained in
a single coordinate chart.
Then by the DeVore--Scherer Theorem (see \cite[Cor.~5.13]{bs}),
the interpolation space $(L^{1}_{l}(\mathbb{R}^d),L^{\infty}_{l}(\mathbb{R}^d))_{1-1/p,p}$
is precisely $L^{p}_{l}(\mathbb{R}^d)$, and by the reiteration theorem
\cite[Th.~5.12]{bs}, 
$$
L^{p}_{l}(\mathbb{R}^d)=(L^{p_1}_{l}(\mathbb{R}^d),L^{p_2}_{l}(\mathbb{R}^d))_{\theta,p}
$$
where $\theta$ is given by $\frac{1}{p}=\frac{1-\theta}{p_1}+ \frac{\theta}{p_2}$.
Similarly,
$$
L^{p}(\mathbb{R}^d)=(L^{p_1}(\mathbb{R}^d),L^{p_2}(\mathbb{R}^d))_{\theta,p}
$$
(see \cite[Th.~1.9]{bs}).
Therefore, Theorem \ref{th:interpolation} is a consequence of \cite[Cor.~1.12]{bs}.
\end{proof}

\subsection{Quantitative pointwise ergodic theorem in Sobolev spaces}\label{sec:pointwise-sobolev}
Let us now note the following general fact based on the representation theory of a semisimple Lie group $G$, which applies 
to any of its closed subgroups, not just the algebraic ones. 

\begin{thm}\label{th:erg_exp}
Assume that 
\begin{itemize}
\item $G$ is a connected semisimple Lie group 
with finite centre acting smoothly on a manifold $Y$ preserving a probability measure $\mu$,
\item the representation of every simple factor of $G$ on $L^2_0(Y)$ is isolated from the trivial representation.
\item $H$ is an arbitrary closed subgroup of $G$,  and the restricted sets  $H_t$  have 
exponential volume growth, namely $\rho(H_t)\gg e^{at}$ for some $a > 0$. 
\end{itemize}
Then there exist $l\in\mathbb{N}$ and $t_0\ge 0$ such that
\begin{enumerate}
\item[(i)] {\rm Strong maximal inequality.} If the sets $H_t$ satisfy the rough monotonicity property,
that is, $\rho(H_{\lfloor t\rfloor+1})\ll \rho(H_t)$, then 
for every $1<p\le\infty$, a compact domain $B$ of $Y$, and $F\in L^p_l(B)^+$,
$$
\left\|\sup_{t\ge t_0} \pi_Y(\beta_t)F\right\|_{L^p(Y)}\ll_{p,B} \|F\|_{L_l^p(B)}.
$$

\item[(ii)] {\rm Quantitative pointwise theorem.}
If the sets $H_t$ satisfy in addition the Holder regularity property as stated in equation (\ref{holder}), then  
for every $1<p<\infty$, a compact domain $B$ of $Y$,
a bounded function $F\in L_l^p(B)$, and almost every $y\in Y$,
\begin{align*}
\left|\pi_Y(\beta_t)F(y)-\int_Y F\, d\mu \right|\le 
 C_p(y,F) e^{-\delta_p t},\quad t\ge t_0,
\end{align*}
where $\delta_p>0$ and the estimator $C_p(y,F)$ satisfies
\begin{align*}
\|C_p(\cdot, F)\|_{L^p(Y)}\ll_{p,B} \|F\|_{L_l^p(B)}+ \|F\|_{\infty}.
\end{align*}

\end{enumerate}
When $H$ is  a connected almost algebraic subgroup, the sets $H_t$ are indeed roughly monotone and H\"older-regular, and the foregoing 
assertions apply to each family $\pi_Y(\beta_t^{g_1,g_2})$, uniformly as $g_1,g_2$ vary in compact sets in $G$.  
\end{thm}

\begin{rem}{\rm We note that the maximal inequality just stated can be improved to an exponential-maximal inequality. The proof uses the analytic interpolation theorem, the estimates established below, and the interpolation results stated in Section \S \ref{sec:erg-sobolev}, in a manner analogous to \cite{mns}. Such an argument renders superfluous the assumption in Theorem \ref{th:erg_exp}(ii) that the function is bounded.}
\end{rem} 

\begin{proof}[Proof of Theorem \ref{th:erg_exp}(i)]
This estimate clearly holds for $p=\infty$, so let us assume that $p<\infty$.
Since $H_t$ have exponential growth, 
it follows from Theorem \ref{th:mean_quant_H} that there exists $\delta_p>0$
such that for every $F\in L^p_l(B)$, 
$$
\left\|\pi_Y(\beta_t)F-\int_Y F\, d\mu \right\|_{L^p(Y)}\ll e^{-\delta_p t}
\|F\|_{L_l^p(B)},\quad t\ge t_0.
$$
Hence, for the function $C(y,F)$ defined by
$$
C(y,F):=\sum_{n\in \mathbb{N}:n\ge t_0} \left|\pi_Y(\beta_n)F(y)-\int_Y F\, d\mu \right|,
$$
we deduce from the triangle inequality that
$$
\|C(\cdot,F)\|_{L^p(Y)}\ll \|F\|_{L_l^p(B)}.
$$
Therefore,
$$
\left\|\sup_{n\in\mathbb{N}:\,n\ge t_0}\left|\pi_Y(\beta_n)F-\int_Y F\, d\mu \right| \right\|_{L^p(Y)}\ll \|F\|_{L_l^p(B)}.
$$
Then
\begin{align*}
\left\|\sup_{n\in\mathbb{N}:\,n\ge t_0}\left|\pi_Y(\beta_n)F\right| \right\|_{L^p(Y)}
\le& \left\|\sup_{n\in\mathbb{N}:\,n\ge t_0}\left|\pi_Y(\beta_n)F-\int_Y F\, d\mu \right|
\right\|_{L^p(Y)} + \left|\int_Y F\, d\mu\right|\\
 \ll& \|F\|_{L_l^p(B)}.
\end{align*}

Finally, to complete the proof, we assume that $F\ge 0$. Then it follows from 
Theorem \ref{th:book}(i) that when $H$ is a connected almost algebraic group $\rho(H_{\lfloor t\rfloor+1})\ll \rho(H_t)$ for all $t\ge t_0$. Whenever this rough monotonicity 
property holds, we have 
\begin{align*}
\pi_Y(\beta_t)F(y)&=\frac{1}{\rho(H_t)}\int_{H_t} F(h^{-1}y)\, d\rho(y)
\le \frac{1}{\rho(H_t)}\int_{H_{\lfloor t\rfloor+1}} F(h^{-1}y)\, d\rho(y)\\
&\ll \pi_Y(\beta_{\lfloor t\rfloor+1})F(y).
\end{align*}
In conjunction with the previous estimate, this completes the proof.
\end{proof}

\begin{proof}[Proof of Theorem \ref{th:erg_exp}(ii)]
As already noted, it follows from Theorem \ref{th:mean_quant_H} that for some $\delta>0$,
\begin{equation}\label{eq:mean_est}
\left\|\pi_Y(\beta_t)F-\int_Y F\, d\mu\right\|_{L^p(Y)}\ll e^{-\delta t}
\|F\|_{L^p_l(B)},\quad t\ge t_0.
\end{equation}
We take an increasing sequence $\{t_i\}_{i\ge 0}$ that contains
all positive integers greater than $t_0$ and has spacing $\lfloor e^{p\delta n/4}\rfloor^{-1}$
on the intervals $[n,n+1]$, $n\in \mathbb{N}$. Then
\begin{equation*}
t_{i+1}-t_i\le e^{-p \delta \lfloor t_i\rfloor/4}
\end{equation*}
for all $i\ge 0$.
It follows from (\ref{eq:mean_est}) that
\begin{align*}
&\int_Y\left(\sum_{i\ge 0} e^{p\delta t_i/2}
  \left|\pi_Y(\beta_{t_i})F(y)- \int_Y F\, d\mu\right|^p\right)\, d\mu(y)\\
\ll& \sum_{i\ge 0} e^{-p\delta t_i/2}\|F\|^p_{L^p_l(B)}
\le \sum_{n\ge \lfloor t_0 \rfloor } e^{-p\delta n/2} \lfloor e^{p\delta n/4}\rfloor 
\|F\|^p_{L^p_l(B)}
\ll \|F\|^p_{L^p_l(B)}.
\end{align*}
Hence, if we set
$$
C(y,F):=\left(\sum_{i\ge 0} e^{p\delta t_i/2}
  \left|\pi_Y(\beta_{t_i})F(y)- \int_Y F_{t_i}\, d\mu\right|^p\right)^{1/p},
$$
then 
\begin{align*}
\left|\pi_Y(\beta_{t_i})F(y)- \int_Y F\, d\mu\right|\le C(y,F)e^{-p\delta t_i/2}
\end{align*}
for all $i\ge 0$, and 
$$
\|C(\cdot,F)\|_{L^p(Y)}\ll \|F\|_{L^p_l(B)}.
$$
For every  $t\ge t_0$, there exists $t_i<t$ such that
$$
t-t_i\ll e^{-p\delta  \lfloor t_i\rfloor /4}\ll e^{-p\delta  t/4}.
$$
Then
\begin{align*}
\left|\pi_Y(\beta_{t})F(y)- \int_Y F\, d\mu\right| \le &
\left|\pi_Y(\beta_{t})F(y)-\pi_Y(\beta_{t_i})F(y)\right|\\
&+ \left|\pi_Y(\beta_{t_i})F(y)- \int_Y F\, d\mu\right|,
\end{align*}
and the following computation completes the proof
\begin{align*}
&\left|\pi_Y(\beta_{t})F(y)-\pi_Y(\beta_{t_i})F(y)\right|\\
=& \left| \frac{1}{\rho(H_t)} \int_{H_t} F(h^{-1}y)\,d\rho(y)
-\frac{1}{\rho(H_{t_i})} \int_{H_{t_i}} F(h^{-1}y)\,d\rho(y)\right|\\
=& \left| \frac{1}{\rho(H_t)} \int_{H_t} F(h^{-1}y)\,d\rho(y)
-\frac{1}{\rho(H_{t})} \int_{H_{t_i}} F(h^{-1}y)\,d\rho(y)\right|\\
&+ \left| \frac{1}{\rho(H_t)} \int_{H_{t_i}} F(h^{-1}y)\,d\rho(y)
-\frac{1}{\rho(H_{t_i})} \int_{H_{t_i}} F(h^{-1}y)\,d\rho(y)\right|\\
\le& 2\frac{\rho(H_t-H_{t_i})}{\rho(H_t)} \|F\|_\infty\ll 
(e^{-p\delta  t/4})^\theta  \|F\|_\infty,
\end{align*}
where the last estimate follows from our assumption that $H_t$ is H\"older-regular. For a connected almost algebraic group $H$, Theorem \ref{th:book}(ii) shows that 
this property does indeed hold.

Finally, the proof for each family  $\pi_Y(\beta_t^{g_1,g_2})$  supported on $H_t[g_1,g_2]$ is similar, and the uniformity as $g_1,g_2$ vary in compact sets in $G$ follows from the uniform norm estimate established in Theorem \ref{th:mean_quant_H}. 
\end{proof}

For future reference below we also recall the following result about ergodic theory of semisimple groups
which is a variation on \cite[Th.~4.2]{gn}:

\begin{thm}\label{th:erg_semi}
Assume that $H\subset G\subset \hbox{\rm SL}_d(\RR)$, $G$ is an algebraic subgroup and $P$ a homogeneous polynomial. Assume also that 
\begin{itemize}
\item the group $H$ is connected and semisimple, 
and it acts on the probability space $(Y,\mu)$ preserving the measure,
\item the representation of every simple factor of $H$ on $L^2_0(Y)$ is isolated from the trivial representation.
\end{itemize}
Then, for the operators $\pi_Y(\beta_t)$ supported on the restricted sets $H_t$, we have for $t\ge t_0$,
\begin{enumerate}
\item[(i)] {\rm Strong exponential maximal inequality}.
For every $1<p < q\le \infty$ and $F\in L^p(Y)$, with $\delta^\prime_{p,q}> 0$
$$
\left\|\sup_{t\ge t_0}e^{\delta^\prime_{p,q}t} |\pi^0_Y(\beta_t)F|\right\|_{L^p(Y)}\ll_p \|F\|_{L^q(Y)}\,.
$$

\item[(ii)] {\rm Quantitative mean ergodic theorem.}
For every $1\le p\le q\le \infty$ with $(p,q)\ne (1,1)$ and $(p,q)\ne (\infty,\infty)$,
and $F\in L^q(Y)$, with $\delta_{p,q}>0$ 
$$
\left\|\pi_Y(\beta_t)F-\int_YF\, d\mu\right\|_{L^p(Y)}\ll_{p,q} e^{-\delta_{p,q}t}\|F\|_{L^q(Y)}\,.
$$
\end{enumerate}
Furthermore, the same conclusion holds without change for the operators $\pi_Y(\beta_t^{g_1,g_2})$ supported on $H_t[g_1,g_2]$, as $g_1$, $g_2$ vary over compact sets 
in $G$. 
\end{thm}
\begin{proof} 
For any given family $\pi_Y(\beta_t^{g_1,g_2})$ the result is a direct consequence of  \cite[Th.~4.2]{gn}. 
The main ingredient in the proof is the strong spectral gap, which implies the representation $\pi_Y^0$ 
on $L^2_0(Y)$ restricted to $H$ is strongly $L^v$, for some $v < \infty$. Using the transfer principle, the Kunze-Stein phenomenon, and coarse admissibility of $H_t$, there is a uniform  norm bound  of the operators, as $g_1,g_2$ vary in a compact set, namely 
 $\norm{\pi_Y^0(\beta_t^{g_1,g_2})}\ll e^{-\kappa_v t}$, with $\kappa_v > 0$.  This implies that  the proof of   \cite[Th.~4.2]{gn} applies uniformly as $g_1,g_2$ vary in a compact set. 
 \end{proof}

\subsection{Spectral and volume estimates on non-amenable algebraic subgroups}\label{sec: Iwasaw}

Let $H$ be a closed almost connected subgroup of $\hbox{SL}_d(\mathbb{R})$.
We fix a non-negative proper homogeneous polynomial $P$ on $\hbox{Mat}_d(\mathbb{R})$
and set
$$
H_t=\{h\in H:\, \log P(h)\le t\}.
$$
We recall that $\beta_t$ denote the Haar-uniform probability measure on $H_t$. 
 
We denote by $R$ the amenable radical of $H$, that is, the maximal closed connected
normal amenable subgroup of $H$. 
We denote by $\rho_R$ and $\rho_H$ the corresponding Haar measures.

\begin{defi}\label{NA embedding}
{\rm 
We say that $H$ is {\it non-amenably embedded} (w.r.t. the gauge function $P$) 
if 
\begin{equation}\label{eq:nonamean}
\limsup_{t\to\infty} \frac{\log \rho_R(R\cap H_t)}{\log \rho_H(H_t)}<1.
\end{equation}
}
\end{defi}
 
We note that if $H$ is non-amenably embeded w.r.t. one homogeneous polynomial as above, then it is non-amenably embedded w.r.t. all of them, so that this notion is independent of the homogeneous polynomial chosen to verify it. 
 
Of course, if $H$ is non-amenably embedded (w.r.t. any gauge function $P$ as above), then $H$ is a non-amenable group. 
Let us recall the following definition \cite[Ch.5]{gn}

\begin{defi} {\it Groups with an Iwasawa decomposition.}
{\rm 
\begin{enumerate}\item[(i)] 
An lcsc group $H$ has an Iwasawa decomposition 
if it has two closed
amenable subgroups $K$ and $Q$, with $K$ compact and $H=KQ$.

\item[(ii)] The {\it Harish-Chandra $\Xi$-function} associated with the 
Iwasawa decomposition $H=KQ$ of the unimodular group $H$ 
is given by 
$$\Xi_H(h)=\int_K \delta^{-1/2}(hk)dk$$
where $\delta$ is the left modular function of 
$Q$, extended to a left-$K$-invariant
function on $H=KQ$. (Thus, if $m_Q$ is left Haar measure on $Q$, 
$\delta(q) m_Q$ is right invariant, and 
$dm_H=dm_K \delta(q) dm_Q$.)
\end{enumerate}
}
\end{defi}

We begin by stating the following basic spectral estimates for 
Iwasawa groups, which follows from \cite[Ch. 5, Prop. 5.9]{gn}.


\begin{thm}\label{Kfinite estimate}
Let $H$ be a unimodular lcsc 
group with an Iwasawa decomposition,  
and $\pi$ a strongly continuous
unitary representation of $H$. 
Assume that the sets $H_t$ are coarsely admissible (i.e, satisfy condition (CA2) from \S\ref{sec:CA}).
Then there exists $c>0$ such that for every $t\ge t_0$, the following estimates hold.
\begin{enumerate}
\item[(i)]
 If $\pi$  
is weakly contained in the regular representation, and in particular if $\pi$ is the regular
representation $\text{\rm reg}_H$ itself, then 
$$\norm{\pi(\beta_t)}\ll  \frac{1}{\vol (H_{t+c})}\int_{H_{t+c}} \Xi_H(h)\,d\rho_H(h)\,.$$
\item[(ii)] If $\pi^{\otimes 2N}$ is weakly contained in the regular representation of $H$, then
$$\norm{\pi(\beta_t)}\ll \norm{\text{\rm reg}_H(\beta_{t+c})}^{\frac{1}{2N}}\,.$$
\end{enumerate}
\end{thm}

We now apply the previous general estimates to the case of subgroups of $\hbox{SL}_d(\RR)$.

\begin{prop}\label{prop:vol-decay}
Let $H\subset  \hbox{\rm SL}_d(\RR)$ be closed, unimodular and almost connected subgroup, and  suppose
that the group $H$ is nonamenably embedded. 
Assume that the sets $H_t$ are coarsely admissible (i.e, satisfy condition (CA2) from \S\ref{sec:CA})
and satisfy $\rho_H(H_t)\sim c_b\, e^{at}t^b$ as $t\to\infty$, with $c_b,a >0$. 
Then
\begin{enumerate}
\item[(i)] For every $p>0$ and $t\ge t_0$,
$$
\int_{H_t} \Xi_H(h)^p\,d\rho_H(h)\ll_p \rho_H(H_t)^{1-\delta_p},
$$
where $\delta_p>0$.
\item[(ii)] The convolution norm of $\text{\rm reg}_H(\beta_t)$ as operators on $L^2(H)$ satisfies the decay estimate 
$$\norm{\text{\rm reg}_H(\beta_t)}_{L^2(H)\to L^2(H)} \ll \vol (H_t)^{-\kappa}\,,$$
with $\kappa > 0$.
\end{enumerate} 

Furthermore, (ii) applies to each family  $\pi_Y(\beta_t^{g_1,g_2})$  supported on $H_t[g_1,g_2]$, uniformly 
as $g_1,g_2$ vary in compact sets in $G$. 
In particular, the assertions above hold when $H$ is an almost algebraic group which is non-amenably
embedded in a semisimple group $G\subset \hbox{\rm SL}_d(\RR)$. 
\end{prop}

\begin{proof}
Since $H$ has finitely many connected components, without loss of generality,
we may assume that $H$ is connected. Let $H=RL$ be the decomposition of $H$ 
where $L$ is a connected semisimple subgroup without compact factors, and $R$ is the amenable radical. The existence of such a decomposition is an
 immediate consequence of the Levi decomposition. 
A left Haar measure $\rho_H$ on $H$ is given by product of a left Haar measure $\rho_R$ on $R$
and a Haar measure $\rho_L$ on $L$. We can further identify an Iwasawa decomposition of $H$ in the form $H=KQ$, where $K\subset L$ is a maximal compact subgroup, and 
$Q=Q_L R$, where $Q_L\subset L$ is a minimal parabolic subgroup of $L$. 

When $H$ is non-amenably embedded in $G$, by  \eqref{eq:nonamean}
there exists $\kappa>0$ such that 
\begin{equation}\label{eq:nonamean2}
\rho_R(R\cap H_t)\le \rho_H(H_t)^{1-\kappa}
\end{equation}
for all sufficiently large $t$. 

By our assumption, there exist $c_b>0$, $a\in \mathbb{Q}_{> 0}$ and $b\in \mathbb{Z}_{\ge 0}$
such that 
\begin{equation}\label{eq:vol1}
\rho_H(H_t)\sim c_b\, e^{at}t^b\quad \hbox{as $t\to\infty$.}
\end{equation}
For a non-compact almost algebraic subgroup $H$ this follows from Theorem \ref{th:book}.

Let $L(s):=\{l\in L;\, d(eK,lK)\le s\}$ denote the ball of radius $s$ with  respect to the Cartan-Killing metric
on the symmetric space $L/K$ of $L$. We use the estimate
\begin{equation}\label{eq:vol2}
\rho_L(L(s))\ll e^{a' s}, \quad s\ge 0,
\end{equation}
with $a'>0$.

Since the coordinates of the matrices in the set $L(s)=L(s)^{-1}$ are bounded by $e^{a''s}$ with some $a''>0$,
up to a multiplicative constant, arguing as in Lemma \ref{l:ccc_c}, we conclude that there exists $c>0$ such that
for every $s>0$,
$$
H_t\cdot L(s)^{-1}\subset H_{t+cs}.
$$
Let $\delta>0$. Using \eqref{eq:nonamean2}, \eqref{eq:vol1} and \eqref{eq:vol2}, we deduce that
\begin{align*}
\rho_H(RL(\delta t)\cap H_t)&\le \rho_H((R\cap H_{t+c\delta t})L(\delta t))
=\rho_R(R\cap H_{t+c\delta t})\rho_L(L(\delta t))\\
&\le \rho_H( H_{t+c\delta t})^{1-\kappa}\rho_L(L(\delta t))\ll e^{\theta t},
\end{align*}
where $\theta=(1+c\delta)(1-\kappa)a+\delta a'$ and $t$ is sufficiently large.
Hence, taking $\delta$ sufficiently small, we obtain $\theta<(1-\kappa/2)a$ and hence
\begin{equation}\label{eq:Bbbb1}
\rho_H(RL(\delta t)\cap H_t)\le \rho_H(H_t)^{1-\kappa/2}
\end{equation}
for all sufficiently large $t$.

To conclude the proof of Proposition \ref{prop:vol-decay}, we recall the well-known estimate on the  Harish-Chandra function $\Xi_L$ 
on the semisimple Lie group $L$ (see e.g. \cite[\S 4.6]{GV}): for every $a$ in the positive Weyl chamber,
$$
\Xi_L(a)  \ll e^{-\rho_L(\log a)}(1+\norm{\log a})^d,
$$
where $\rho_L$ denotes the half-sum of positive roots.
Given $l\in L$,
let its Cartan decomposition be given by $l=kak^\prime$, where $k,k^\prime\in K$
and $a$ in the positive Weyl chamber. Note that
$$
d(K,lK)=d(K,aK)=\|\log a\|.
$$
Since there exists $\eta'>0$ such that
$$
\rho_L(\log a)\ge \eta^\prime \norm{\log a}
$$
for $a$ in the positive Weyl chamber, it follows that there exists $\eta>0$ such that for $l \in L$,
\begin{equation}\label{eq:Bbbb2}
\Xi_L(l)=\Xi_L(a)  \ll e^{-\rho_L(\log a)}(1+\norm{\log a})^d   \ll e^{- \eta\, d(K,aK)}= e^{- \eta\, d(eK,lK)}\,.
\end{equation}
Furthermore, the Harish-Chandra $\Xi_H$-function of the group $H$ satisfies 
$$
\Xi_H (rl)=\Xi_H(lr)=\Xi_L(l)
$$
for all $l\in L$ and $r\in R$. Indeed, every element in the 
amenable radical $R$ acts trivially on the homogeneous space $H/Q$, since $R$ is a normal subgroup of $H$
contained in $Q$, so that $rhQ=hr^\prime Q=hQ$.   Hence,
$$
\Xi_H(h)=\int_{H/Q}\sqrt{r_m(h,yQ)}dm(yQ),
$$
where $r_m(h,yQ)$ is the Radon-Nikodym derivative of the
unique $K$-invariant probability measure $m$ on $H/Q$,  which is invariant under left and right translations by $r\in R$. 

Using bounds \eqref{eq:Bbbb1} and \eqref{eq:Bbbb2}, we obtain that
for all sufficiently large $t$,
$$
\int_{RL(\delta t)\cap H_t} \Xi_H(h)^p\,d\rho_H(h)\ll_p \rho_H(H_t)^{1-\kappa/2},
$$
and 
$$
\int_{H_t-RL(\delta t)} \Xi_H(h)^p\,d\rho_H(h)\ll_p e^{-p\eta \delta t}\rho_H(H_t).
$$
This implies (i), and (ii) follows from Theorem \ref{Kfinite estimate}. Clearly, when (CA1)
is satisfied, the validity of the norm estimate stated in (ii) implies its validity for the families $H_t[g_1,g_2]$, uniformly as $g_1,g_2$ vary over compact sets in $G$.  
\end{proof}

\subsection{Quantitative ergodic theorems for non-amenable algebraic subgroups}\label{sec:erg-Iwasaw}

We can now state the following general ergodic theorem in Lebesgue $L^p$-spaces, which applies in particular to algebraic subgroups $H$ of a semisimple Lie group $G$. 

\begin{thm}\label{thm:non-amenable-erg}
Assume that $H\subset \hbox{\rm SL}_d(\RR)$ is a unimodular non-amenably embedded closed subgroup, 
and that the restricted sets $H_t$ satisfy the H\"older property as in Theorem \ref{th:book}(ii)
and have  volume asymptotic $c_b\, t^{b}e^{at}$ with $c_b,a>0$. 
Assume that $H\subset G$, where $G$ is semisimple, and $G$ acts on a probability space $Y$ preserving an
ergodic probability measure, such that the representation of $G$ in $L^2_0(Y)$ has a strong spectral
gap. Then 


\begin{enumerate}
\item The family $\pi_Y(\beta_t)$ satisfies the $(L^p, L^r)$-exponentially fast 
 mean ergodic theorem for  $1 < r\le p < \infty$,  namely there exists 
$\delta_{p,r}>0$ such that for every $f\in L^p(Y)$,
$$
\left\|\pi_Y(\beta_t)f-\int_Y f\,d\mu\right\|_{L^r(Y)}\ll_{p,r} 
e^{-\delta_{p,r} t}\|f\|_{L^p(Y)}
$$
for all $t\ge t_0$.
\item The family $\pi_Y(\beta_t)$ satisfies the $(L^p,L^r)$-exponential strong 
maximal inequality  for  $1< r <  p< \infty$, namely there exist $t_0>0$ and
$\delta_{p,r}>0$ such that for every $f\in L^p(Y)$,
$$
\left\|\sup_{t\ge t_0} e^{\delta_{p,r} t}
 \left|\pi_Y(\beta_t)f-\int_Y f\,d\mu\right|\right\|_{L^r(Y)}
\ll_{p,r}\|f\|_{L^p(Y)}.
$$
\item The family $\beta_t$ satisfies the {\it $(L^p,L^r)$-exponentially fast 
pointwise ergodic theorem} for 
$1<r <  p< \infty$, namely there exists $\delta_{p,r}>0$ such that for every $f\in L^p(Y)$ and $t\ge t_0$,
$$
\left|\pi_Y(\beta_t)f(y)-\int_Y f\,d\mu\right|\le B_{p,r}(y,f)
e^{-\delta_{p,r} t}
\quad\hbox{ for $\mu$-a.-e. $y\in Y$}
$$
with the estimator $B_{p,r}(y,f)$ satisfying the norm estimate 
$$
\|B_{p,r}(\cdot,f)\|_{L^r(Y)}\ll_{p,r}\|f\|_{L^p(Y)}.
$$

\end{enumerate} 

Furthermore, the same results hold without change for the operators $\pi_Y(\beta_t^{g_1,g_2})$ supported on the sets $H_t[g_1,g_2]$, uniformly when  $g_1,g_2$ vary over compact sets in $G$. 

In particular,  the conclusions hold when $H$ is an almost algebraic non-amenably embedded subgroup.

\end{thm}

\begin{proof}
Since the representation $\pi^0_Y$ of $G$ in $L_0^2(Y)$ has a strong spectral gap, it follows that for
some even $k$, $\left(\pi_Y^0\right)^{\otimes k}$ is isomorphic to a subrepresentation of $\infty
\cdot\text{\rm reg}_G$, i.e. to a subrepresentation of a multiple of the regular representation of $G$
(see the discussion in \cite[Ch. 5]{gn} for more details). It therefore follows that the restriction of
$\pi_Y^0$ to the closed subgroup $H$ has the same property, namely that $\left(\pi_Y^0|_H\right)^{\otimes
  k}\subset \infty \cdot \text{\rm reg}_H$. By Theorem \ref{Kfinite estimate} and 
 Proposition \ref{prop:vol-decay} it follows that the exponentially fast mean ergodic theorem holds in $L^2$, uniformly for $\pi_Y(\beta_t^{g_1,g_2})$ as $g_1,g_2$ vary over compact sets in $G$.  Using interpolation, it also holds as stated in part (i) of Theorem \ref{thm:non-amenable-erg}.  The fact that under the regularity conditions stated in the Theorem, together with the norm decay established in part (i), the assertions of part (ii) and part (iii) follows is proved in detail in \cite[Ch. 5]{gn} in the proof of Theorem 5.7. 
\end{proof}

\section{Completion of the proof of the main theorems}\label{sec:proof}

We write the algebraic homogeneous space $X$ of $G$ as a factor space
$X\simeq H\backslash G$
where $H$ is an almost algebraic subgroup of $G$.
The main theorems stated in the introduction will be deduced
from the ergodic theory for the action of $H$ on $Y\simeq G/\Gamma$ developed 
in Section \ref{sec:alg} combined with the ergodic-theoretic duality results developed
in Sections \ref{sec:max_ineq}--\ref{sec:quant_pointwise}, and with the volume regularity properties 
established in Section \ref{sec:volume}.

\subsection{Regularity properties of the sampling sets}
We apply the results of Sections \ref{sec:max_ineq}--\ref{sec:volume} to the sets
$$
G_t=\{g\in G:\, \log P(g)\le t\}\quad\hbox{and}\quad
H_t[g_1,g_2]=\{h\in H:\, \log P(g_1^{-1}hg_2)\le t\},
$$
where $P$ is either a non-negative proper homogeneous polynomial or norm on $\hbox{Mat}_d(\RR)$.
Let us verify that these sets
satisfy the regularity properties used in Sections \ref{sec:max_ineq}--\ref{sec:quant_pointwise}. 
When $P$ is a non-negative proper homogeneous polynomial,
we use the results established in Section \ref{sec:volume}.
Property (CA1) follows from Lemma~\ref{l:ccc_c}(i),
and (\ref{eq:ha1}) of property (HA1) (and, in particular, (A1)) follows from Lemma~\ref{l:ccc_c}(i).
To complete verification of (HA1) we observe that a 
function $\chi_\vre$ satisfying (\ref{eq_psi_e}) can be constructed by identifying 
neighbourhoods of the identity in $H$ with neighbourhoods of the origin in the Euclidean space
and taking $\chi_\vre(x)=\vre^{-\dim(H)} \chi(\vre x)$ for a fixed
$\chi\in C_c^l(\mathbb{R}^{\dim(H)})$. This shows that (\ref{eq_psi_e}) holds with
$\kappa=(l+\dim(H)(1-1/q)$. Since property (\ref{eq:cover}) follows from the corresponding
property of the Euclidean space, we conclude that (HA1) holds.
Property (CA2) follows from Theorem \ref{th:book}(i).
Properties (A2), (A2$^\prime$), (HA2), (H2$^\prime$) with $1\le r<\infty$
are established in Proposition \ref{p:reg}. Property (A3) is a consequence of 
Proposition \ref{p:alpha}. When $P$ is a norm, properties (CA1), (A1), and (HA1)
directly follow from norm properties. The argument of \cite[Prop.~7.3]{gn} gives
the estimate
$$
\rho(H_{t+\vre}[u,v])-\rho(H_{t}[u,v])\le c\,\vre\rho(H_t)
$$
for all $t\ge t_0$ and $\vre\in (0,1)$, where $c$ is uniform over $u,v$ in compact sets.
This implies conditions (CA1), (A2), (A2$^\prime$), (HA2), (HA2$^\prime$).
Condition (A3) follows from the asymptotic formula for $\rho(H_t)$ established in \cite{GW,Mau}.
Finally, Properties (S) and (HS) are standard in the theory of homogeneous
spaces of Lie groups. Therefore, we conclude that the results established in
Sections \ref{sec:max_ineq}--\ref{sec:quant_pointwise} apply in our setting. 

\subsection{The limiting density}
It follows from Theorem \ref{th:book}(i) that for some $c_b>0$,
$a\in\mathbb{Q}_{\ge 0}$, and $b\in\mathbb{Z}_{\ge 0}$,
\begin{equation}\label{eq:asssss}
\rho(H_t)= c_b\, e^{at}t^b+O(e^{at}t^{b-1}).
\end{equation}
We recall that  by Lemma \ref{l:growth}, $a=0$ if and only the Zariski closure of $H$ is
an almost direct product of a compact subgroup and an abelian diagonalisable subgroup
as in Theorem \ref{th:main1}. The quantity $V(t):=e^{at}t^b$ is the correct normalisation for our averages.

Recall that we defined in (\ref{eq:nu_x})) 
$$
 d\nu_x(y)=\left(\lim_{t\to\infty} \frac{\rho(H_t[\mathsf{s}(x),\mathsf{s}(y)])}{\rho(H_t)}\right)
d\xi(y)
$$
Here the limit exists and is positive and continuous by Proposition \ref{p:alpha},
and the measure $\xi$ is defined by (\ref{eq:measure}).
One can verify that the measures $\nu_x$ are canonically defined, i.e.,
they are independent of a choice of the section $\mathsf{s}$ and the Haar measure $\rho$ on $H$.

We denote $\tilde{\nu}_x$, $x\in X$, the family of measures on $X$ defined similarly by using  the alternative normalization 
\begin{equation}\label{eq:v_x_x}
d\tilde{\nu}_x(y)=\left(\lim_{t\to\infty} \frac{\rho(H_t[\mathsf{s}(x),\mathsf{s}(y)])}{e^{at}t^b}\right)
d\xi(y).
\end{equation}
so that $\tilde{\nu}_x=c_b\, \nu_x$. 

If the group $H$ is of subexponential type, then
it follows from Proposition \ref{p:alpha_sub} that 
\begin{equation}\label{eq:vv_x}
\tilde{\nu}_x=c_b\,\xi, 
\end{equation}
and since $H$ is unimodular,
this gives the unique (up to scalar) $G$-invariant measure on $X$.

\subsection{Proof of Theorem \ref{th:main1}, Theorem \ref{th:main2} and Theorem \ref{th:main3}}
Let us now compare the averages 
$$
\pi_{X}(\tilde\lambda_t)\phi(x)=\frac{1}{e^{at} t^b}\sum_{\gamma\in \Gamma_t} \phi(x\gamma)
$$
 with the averages
$$
\pi_{X}(\lambda_t)\phi(x)=\frac{1}{\rho(H_t)}\sum_{\gamma\in \Gamma_t} \phi(x\gamma)
$$
which formed the subject of the discussion in  Sections \ref{sec:max_ineq}--\ref{sec:quant_pointwise}.
It follows from (\ref{eq:asssss}) that
$$
|\pi_{X}(\tilde\lambda_t)\phi|\ll |\pi_{X}(\lambda_t)\phi|.
$$
Therefore,
Theorem \ref{th:main1}(i) follows from Theorem \ref{th_max_ineq}(ii) combined with Theorem \ref{th:erg_sub}(ii),
Theorem \ref{th:main2}(i) follows from Theorem \ref{th_max_ineq}(ii) combined with Theorem
\ref{th:erg_exp}(i), and Theorem \ref{th:main3}(i) follows from Theorem \ref{th_max_ineq}(ii) combined 
with Theorem \ref{th:erg_semi}(i), provide that the group $H$ is connected.
In general, its connected component $H^0$ has finite index in $H$, and we can apply
the previous argument to the finite cover $H^0\backslash G$ which implies 
Theorem \ref{th:erg_semi}(i) for $H\backslash G$.

Similarly, by (\ref{eq:asssss}), 
$$
|\pi_{X}(\tilde\lambda_t)\phi(x)-c_b\,\pi_{X}(\lambda_t)\phi(x)|
\ll t^{-1} |\pi_{X}(\lambda_t)\phi(x)|, 
$$
and by Theorem \ref{th_max_ineq}(i),
\begin{equation}\label{eq:lllll1} 
\|\pi_{X}(\tilde\lambda_t)\phi-c_b\,\pi_{X}(\lambda_t)\phi\|_{L^p(D)}
\ll t^{-1} \|\phi\|_{L^p(D)}.
\end{equation}
Hence, Theorem \ref{th:main1}(ii) follows from Theorem \ref{th:mean} combined with Theorem
\ref{th:erg_sub}(iii), and
Theorem \ref{th:main1}(iv) follows from Theorem \ref{th:pointwise} combined with Theorem   
\ref{th:erg_sub}(iii).

Combining Theorem \ref{th:mean} with Theorem \ref{th:mean_quant_H}, 
we deduce that Theorem \ref{th:main2}(ii) holds for $\phi\in L^p_l(D)^+$ with $p>1$.
Since it sufficient to prove convergence for a dense family of functions
(see the proof of Theorem \ref{th:mean}), this implies the claim of Theorem \ref{th:main2}(ii).
Theorem \ref{th:main2}(iii) follows from Theorem \ref{th:pointwise} combined with Theorem
\ref{th:erg_exp}(ii).

To prove Theorem \ref{th:main1}(iii),
we observe that by Theorem \ref{th_dual_Ger} and Theorem \ref{th:mean_quant_H},
for some $\delta>0$,
\begin{equation}\label{eq:lllll2}
\left\| \pi_X(\lambda_t)\phi(x)- \pi_X( \lambda^G_t)\phi
\right\|_{L^p(D)}\ll 
t^{-\delta} \|\phi\|_{L_l^{q}(D)},
\end{equation}
where 
\begin{align*}
\pi_X(\lambda^G_t)\phi(x) 
&=
\frac{1}{\rho(H_t)}\int_{G_t} \phi(\mathsf{p}_X(\mathsf{s}(x)g))\, dm(g)\\
&=\frac{1}{\rho(H_t)}\int_{(y,h):\, \mathsf{s}(x)^{-1}h\mathsf{s}(y)\in G_t} \phi(\mathsf{p}_X(h\mathsf{s}(y)))\, d\rho(h)d\xi(y)\\
&=\int_X\phi(y)\frac{\rho(H_t[\mathsf{s}(x),\mathsf{s}(y)])}{\rho(H_t)}\,d\xi(y).
\end{align*}
It follows from Proposition \ref{p:alpha_sub} that
\begin{equation}\label{eq:lllll3}
\left|\pi_X(\lambda^G_t)\phi(x)-\int_X\phi\, d\xi\right| \ll t^{-1}\|\phi\|_{L^1(D)}
\end{equation}
uniformly as $x$ varies in compact sets. Therefore, combining estimates (\ref{eq:lllll1}),
(\ref{eq:lllll2}), and (\ref{eq:lllll3}), we deduce Theorem \ref{th:main1}(iii).

Theorem \ref{th:main2}(iv) is deduced from Theorem \ref{th_dual_pointwise} combined with Theorem \ref{th:mean_quant_H}.
 It follows from Theorem \ref{th_dual_pointwise} that 
for some $\delta>0$ and almost every $x\in D$,
$$
\left| \sum_{\gamma\in\Gamma_t} \phi(x\gamma)- \int_{G_t}\phi(xg)\,dm(g) \right| \ll_{\phi,x} e^{-\delta
  t}\rho(H_t)
\ll e^{(a-\delta)t}t^b.
$$
Since by Proposition \ref{p:infinity},
$$
\int_{G_t} \phi(x g)\, dm(g)=e^{a t}\left(\sum_{i=0}^b c_i(\phi,x) t^i \right) +O_{\phi,x}\left(e^{(a-\delta)t}\right),
$$
this implies Theorem \ref{th:main2}(iv).

 Theorem \ref{th:main3}(ii) is deduced similarly from
Theorem \ref{th_dual_pointwise} and Theorem \ref{th:erg_semi}(ii), when $H$ is connected.
In general, we reduce the argument to the action on the space $H^0\backslash G$
which is a finite cover of $X$. This implies Theorem \ref{th:main3}(ii) in general. Theorem \ref{th:main3}(iii) is derived 
from Theorem \ref{th_dual_pointwise} combined with Theorem \ref{th:erg_semi} and  Theorem \ref{thm:non-amenable-erg} in the same manner as 
Theorem \ref{th:main2}(iv), using that the norm estimates of $\pi_Y(\beta_t)$ are available in Lebesgue
spaces, rather than just Sobolev spaces, uniformly as 
$g_1,g_2$ vary in a compact set. 

\subsection{Proof of Theorem \ref{th:main5}}
First, let us note that to prove each of the three statement in Theorem \ref{th:main5}, it suffices to prove them for non-negative functions in the function space under consideration. 
Thus we can assume that the function $\phi$ is non-negative, when convenient. 

We begin by proving the mean ergodic theorem stated in Theorem \ref{th:main5}(i). By Theorem \ref{th:erg_semi}(ii), Theorem \ref{th_dual_Ger} (applying the case $l=0$), and Theorem \ref{thm:non-amenable-erg}, we conclude that for some $\delta > 0$,
\begin{equation}\label{eq:lllll2_1}
\left\| \pi_X(\lambda_t)\phi(x)- \pi_X( \lambda^G_t)\phi (x)
\right\|_{L^p(D)}\ll 
e^{-\delta t} \|\phi\|_{L^{q}(D)}.
\end{equation}
Since 
$$\pi_X(\lambda^G_t)\phi(x) =
\int_X\phi(y)\frac{\rho(H_t[\mathsf{s}(x),\mathsf{s}(y)])}{\rho(H_t)}\,d\xi(y),
$$
and the density of $\nu_x$ with respect to $\xi$ is given by 
$$ \lim_{t\to \infty}\frac{\rho(H_t[\mathsf{s}(x),\mathsf{s}(y)])}{\rho(H_t)}=\Theta(\mathsf{s}(x),\mathsf{s}(y)),$$
it suffices to estimate
\begin{equation}\label{kernel}
\abs{ \pi_X(\lambda^G_t)\phi(x)- \int_D \phi\, d\nu_x}=\abs{\int_X \phi(y)\left( \frac{\rho(H_t[\mathsf{s}(x),\mathsf{s}(y)])}{\rho(H_t)}-\Theta(\mathsf{s}(x),\mathsf{s}(y))\right)d\xi(y)}.
\end{equation}

Since we assume that the homogeneous polynomial in question is a norm, $H$ is semisimple,  and the volume growth of $H_t$ is purely exponential, we can appeal to \cite[Thm. 3]{Mau}, where the following regularity property is established for the volume of $H_t[g_1,g_2]$  :
$$\rho(H_t[g_1,g_2])=c(g_1,g_2)e^{at} +O(e^{(a-\delta)t})\,,$$
with $\delta> 0$ independent of $g_1,g_2$, and $c(g_1,g_2)$ and the implied constant uniform as $g_1,g_2$ vary in compact sets in $G$. 
This immediately implies an exponential decay estimate of the kernel that appears in equation (\ref{kernel}) and the quantitative mean ergodic theorem follows.

The exponential-maximal inequality, namely, the estimate for the quantity 
$$
\sup_{t\ge t_0}e^{\delta_{p,w} t}\abs{\pi_X(\lambda_t)\phi(x)-\int_X \phi\, d\nu_x}$$
stated in Theorem \ref{th:main5}(ii)
follow from the exponential decay estimate just established on the norms
$\norm{\pi_X(\lambda_t)\phi(x)-\int_X\phi\, d\nu_x}_{L^p(D)}$. This follows from the same argument 
as already used in the proof of Theorem \ref{th_dual_pointwise}. 

The exponentially fast pointwise ergodic theorem stated in Theorem \ref{th:main5}(iii) follows directly from Theorem \ref{th_dual_pointwise} (applying the case $l=0$), together 
with the pointwise estimate arising from equation (\ref{kernel}) using the volume asymptotics just cited.


Finally, to prove the statement in Remark  \ref{b> 0} that when $H$ is semisimple,
but the growth is not necessarily purely exponential (namely $b > 0$), the arguments cited above of
\cite{GW,Mau} yield 
$$\rho(H_t[g_1,g_2])=c(g_1,g_2)e^{at}t^b +O(e^{at}t^{b-\delta})\,,$$
uniformly as $ g_1,g_2$ vary  in compact sets.

This implies a rate of decay estimate for the kernel that appears in equation (\ref{kernel}), and 
repeating the foregoing arguments using this  estimate, we deduce that the the quantitative mean and pointwise ergodic theorems with speed $t^{-\eta_p}$.

\begin{rem}{\rm 
We note that the argument establishing the volume asymptotics used above in \cite{GW} and \cite{Mau} use the triangle inequality for norms in an 
essential way, as well as the explicit formulas for the invariant measure on semisimple
groups. On the other hand, the argument with resolution of singularities used in Theorem \ref{th:book}
applies to general groups. A very similar estimate holds, but its uniformity in $x,y$
is not clear. Whenever such uniformity is established, equation (\ref{kernel}) will provide quantitative results as in Theorem \ref{th:main5}. }
\end{rem}

\section{Examples and applications}\label{sec:examples}

We now turn to discuss some examples in detail. We have formulated most of the results in general, but remind the reader that in examples 11.1, 11.4, 11.5 and 11.7 below, if we choose  the homogeneous polynomial $P$ to be norm with purely exponential volume growth of balls, then the stronger results of Theorem \ref{th:main5} apply.

\subsection{Quadratic surfaces}\label{sec:quad}
We discuss the action on the de-Sitter space mentioned in the introduction
(see (\ref{eq:quad0})).  
We observe that the group $G:=\hbox{SO}_{d,1}(\mathbb{R})^0$
acts transitively on $X$, and $X\simeq H\backslash G$
where 
$$
H=\hbox{Stab}_G(e_1)=\left(
\begin{tabular}{ll}
1 & 0\\
0 & $\hbox{SO}_{1,d-1}(\mathbb{R})^0$ 
\end{tabular}
\right).
$$
When $d=2$, the group $H$ is a one-dimensional
$\mathbb{R}$-diagonalisable subgroup. Hence, we are in the setting of 
Theorem \ref{th:main1} in this case. When $d\ge 3$, $H$ is a simple almost algebraic group,
and the representation of $H$ on $L^2_0(G/\Gamma)$ is isolated
from the trivial representation. We will now see that we are in fact in the setting of
Theorem \ref{th:main5}, namely for our choice of norm  the volume growth is purely exponential.

Let us compute the normalization factor $V(t)$ and the limit measures.
Let 
$$
K_0=\left(
\begin{tabular}{ccc}
1 & 0 & 0\\
0 & $\hbox{SO}_{d-1}(\mathbb{R})$ & 0\\
0 & 0 & 1
\end{tabular}
\right)\quad\hbox{and}\quad
B^+=\left\{ b_s=
\left(
\begin{tabular}{ccc}
id & 0 & 0\\
0 & $\cosh s$ & $\sinh s$ \\
0 & $\sinh s$  & $\cosh s$
\end{tabular}
\right)\right\}_{s\ge 0}.
$$
Then we have the Cartan decomposition 
$$
H=K_0B^+K_0,
$$
and a Haar measure
with respect to this decomposition is given by
$$
d\rho(k_1, s, k_2)= dk_1\, (\sinh s)^{d-2}ds \, dk_2,\quad (k_1,s,k_2)\in K_0\times \mathbb{R}_{\ge
  0}\times K_0,
$$
where $dk_1$ and $dk_2$ denote the probability Haar measures on $K_0$.
Since
$$
\|k_1b_s k_2\|=\|b_s\|=e^s+O(1),
$$ 
and it follows that (up to the factor $2^{2-d}$)
$$
\rho(H_t)\sim \left\{ 
\begin{tabular}{l}
$t$\quad\quad\quad \hbox{when $d=2$},\\
$e^{(d-2)t}$\quad\hbox{when $d\ge 3$}.
\end{tabular}
\right.
$$
This gives the normalization factor $V(t)$.

To compute the limit measure, we use that $H$ is a symmetric subgroup of $G$.
We have the Cartan decomposition for $G$ with respect to $H$ given by
$$
G=HAK,
$$
where
$$
K=\left(
\begin{tabular}{cc}
1 & 0\\
0 & $\hbox{SO}_{d}(\mathbb{R})$
\end{tabular}
\right)\quad\hbox{and}\quad
A=\left\{ a_r=
\left(
\begin{tabular}{ccc}
$\cosh r$ & 0 & $\sinh r$ \\
0  & id & 0\\
$\sinh r$  & 0 & $\cosh r$
\end{tabular}
\right)\right\}_{r\in \mathbb{R}}.
$$
We note that the $A$-component of this decomposition is unique, and 
when $r\ne 0$, the $K$-component is unique modulo $K_0$.
Therefore, $X\simeq H\backslash G$ can be identified (up to measure zero)
with $A\times K_0\backslash K$. 
More explicitly, the identification is given by 
the polar coordinates (\ref{eq:polar}).
We use this identification to give
the section $\mathsf{s}:X\to G$. A Haar measure on $G$ with respect to the Cartan
decomposition is given by
$$
dm(h,r,\omega)=d\rho(h)\, (\cosh r)^{d-1}dr\, d\omega, \quad (h,r,\omega)\in H\times \mathbb{R}\times K_0\backslash K,
$$
where $d\omega$ denotes the probability
Haar measure on $K_0\backslash K$. We normalise $m$, so that $m(G/\Gamma)=1$.
It follows that the measure $\xi$ appearing in (\ref{eq:measure}) is equal up to a constant to
$$
d\xi(r,\omega)=(\cosh r)^{d-1}dr\, d\omega, \quad (r,\omega)\in \mathbb{R}\times K_0\backslash K.
$$
In fact, this is a Haar measure on $X$. 
Using that the norm is $K$-invariant, and $A$ commutes with $K_0$,
we deduce that for $x_1=e_1 a_{r_1}\omega_1$, $h=k_1b_sk_2$, $x_2=e_1
a_{r_2}\omega_2$, we have 
\begin{align*}
\|\mathsf{s}(x_1)^{-1} h \mathsf{s}(x_2)\|&= \|\omega_1^{-1} a_{r_1}^{-1} k_1b_sk_2 a_{r_2}\omega_2\|
= \|a_{r_1}^{-1} b_s a_{r_2}\|\\
&=c(r_1,r_2)e^s+O_{r_1,r_2}(1),
\end{align*}
where
$$
c(r_1,r_2)=\left(1+(\sinh r_1)^2\right)^{1/2}\left(1+(\sinh r_2)^2\right)^{1/2}.
$$
This implies that
$$
\lim_{t\to\infty} \frac{\rho(H_t[\mathsf{s}(x_1),\mathsf{s}(x_2)])}{V(t)}
=c(r_1,r_2)^{-(d-2)},
$$
and the limit measure is given by
$$
d\nu_v(r,\omega)=
\left(1+v_d^2\right)^{-(d-2)/2}\left(1+(\sinh r)^2\right)^{-(d-2)/2}(\cosh r)^{d-1}dr\, d\omega.
$$

Since for $d\ge 3$ the group $H$ is simple and has a spectral gap in $G/\Gamma$, and we have chosen a
norm such that $H_t$ has purely exponential volume growth, Theorem \ref{th:main5} applies. We conclude
that for every $\phi\in L^p(X)$, $p>1$,
 with compact support
and almost every $v\in X$, the following quantitative pointwise convergence theorem holds :
\begin{equation}\label{eq:quad1_1}
\frac{1}{e^{(d-2)t}} \sum_{\gamma\in\Gamma_t}
  \phi(v\gamma)= \frac{c_{d}(\Gamma)}{\left(1+v_d^2\right)^{(d-2)/2}} \int_{X} \phi(r,\omega)\, 
\frac{(\cosh r)^{d-1}dr\, d\omega}{\left(1+(\sinh r)^2\right)^{(d-2)/2}}
+O_{p,v,\phi}(e^{-\delta_p t})
\end{equation}
for some $c_d(\Gamma)>0$, and with a fixed $\delta_p>0$, independent of $v$ and $\phi$ .

\subsection{Projective spaces}\label{sec:proj}
We return to the action on the projective space discussed in the Introduction
(see (\ref{eq:projective})).
We observe that $\mathbb{P}^{d-1}(\mathbb{R})$ is a homogeneous space of $G:=\hbox{SL}_d(\mathbb{R})$,
and $\mathbb{P}^{d-1}(\mathbb{R})\simeq H\backslash G$ where
$$
H:=\left(\begin{tabular}{cc} $\star$ & 0\\ $\star$ & $\star$ \end{tabular}\right)\subset G
$$
is the maximal parabolic subgroup of $G$. Since $G=HK$ with $K:=\hbox{SO}_d(\mathbb{R})$,
a Haar measure on $G$ is given by
$$
\int_G f\, dm(g)=\int_{(K\cap H)\backslash K}\int_H f(hk)\, d\rho(h)d\xi(k),\quad f\in L^1(G),
$$
where $\rho$ is the left Haar measure on $H$ and $\xi$ a Haar measure on $(K\cap H)\backslash K$.
We normalise $\xi$ to be the probability measure and normalise $\rho$ so that $m(G/\Gamma)=1$.
Then under the identification $\mathbb{P}^{d-1}(\mathbb{R})\simeq (K\cap H)\backslash K$,
the measure $\xi$ is the measure appearing in (\ref{eq:measure}). 
It follows from $K$-invariance of $\|\cdot\|$ and \cite[Appendix~1]{drs} that
$$
\rho(H_t)=m(G_t)\sim c\, e^{(d^2-d)t}\quad\hbox{as $t\to\infty$}
$$
with $c>0$. Hence, the correct normalisation factor in (\ref{eq:projective}) is $V(t)=e^{(d^2-d)t}$.
By (\ref{eq:v_x_x}), the limit measure is given by
$$
\nu_v(u)=\Theta(v,u)d\xi(u),\quad u,v\in (K\cap H)\backslash K,
$$
where
\begin{align*}
\Theta(v,u)&=\lim_{t\to\infty}\frac{\rho(H_t[v,u])}{e^{(d^2-d)t}}=\Theta(e,e),
\end{align*}
by the $K$-invariance of the norm. This completes verification of (\ref{eq:projective}).
It is clear that the above argument applies to other compact homogeneous spaces of
$\hbox{SL}_d(\mathbb{R})$ such as the Grassmann varieties and the flag variety.

The limit formula (\ref{eq:projective}) but without an error estimate was
obtained in \cite{G1}.

\subsection{Spaces of frames}\label{sec:frame}
Let $\Gamma$ be a lattice in $\hbox{SL}_d(\mathbb{R})$ and $\Gamma_t=\{\gamma\in\Gamma:\,
\log\|\gamma\|\le t\}$
denote the norm balls with respect to the standard Euclidean norm
$\|\gamma\|=\left(\sum_{i,j=1}^d \gamma_{ij}^2\right)^{1/2}$.
We consider the action of $\Gamma$ on the space $X=\prod_{i=1}^k \mathbb{R}^d$
with $k<d$.  We demonstrate that one can analyse the asymptotic distribution  of the
averages $\sum_{\gamma\in \Gamma_t} \phi(x\gamma)$
on $X$ with a help of Theorem \ref{th:main2}. This question was studied in \cite{G2}
(and for the two dimensional case in \cite{L1,n,lp,mw}). Although the method of \cite{G2}
allows to compute the asymptotics of $\sum_{\gamma\in \Gamma_t} \phi(x\gamma)$, it
is not capable to give a rate of convergence.
Theorem \ref{th:main2}(iv) implies
that  for any nonnegative continuous subanalytic function
$\phi\in L_l^1(\mathbb{R}^d)$ with compact support, and for almost every $v\in X$,
there exists $\delta>0$ such that
\begin{equation}\label{eq:frames}
\frac{1}{e^{(d-1)(d-k)t}} \sum_{\gamma\in\Gamma_t}
  \phi(v\gamma)= \frac{c_{d,k}(\Gamma)}{\vol(v)^{d-1}}\int_{X} \phi(w)\, \frac{dw}{\vol(w)}
  +O_{\phi,v}(e^{-\delta t}),
\end{equation}
where $c_{d,k}(\Gamma)>0$, $\vol(v)$ denotes the Euclidean volume of the $k$-dimensional parallelepiped
spanned by the tuple of vectors in $v$, and $dw$ denotes the measure on $X$
which is the product of the Lebesgue measures on $\mathbb{R}^d$. 

To deduce (\ref{eq:frames}) from Theorem \ref{th:main2} we observe that the subset of $X$
consisting of linearly independent vectors is a single orbit of the group $G:=\hbox{SL}_d(\mathbb{R})$
which has full measure on $X$. Therefore, up to measure zero $X\simeq H\backslash G$ where
$$
H:=\left(\begin{tabular}{cc} $id$ & 0\\ $\star$ & $\star$ \end{tabular}\right)\subset G,
$$
and we are in the setting of Theorem \ref{th:main2}. It remains to compute the normalisation
factor $V(t)$ and the limit measure, and that has already been done in \cite{G2} (see \cite[Theorem~3]{G2}).

\subsection{Dense projections}\label{sec:projections}
Let $H\subset \hbox{SL}_n(\mathbb{R})$ and $L\subset \hbox{SL}_m(\mathbb{R})$
be connected semisimple groups, $G=H\times L$, and let $\Gamma$ be 
an lattice in $G$ such that its image under the natural projection map $\pi: G\to L$
is dense. We investigate the distribution of $\pi(\Gamma)$ in $L$.
We fix Euclidean norms on $\hbox{Mat}_n(\mathbb{R})$ and $\hbox{Mat}_m(\mathbb{R})$
and the set 
$$
\Gamma_t=\{\gamma=(h,\ell):\, \|(h,\ell)\|:=\sqrt{\|h\|^2+\|\ell\|^2}<e^t\}.
$$
Let us assume that the representation of every simple factor of $H$ on $L^2_0(G/\Gamma)$ 
is isolated from the trivial representation. This is known to be the case when $G$ has no compact factors
and also when $\Gamma$ is a congruence subgroup (see \cite{ks}).
In this case, Theorem \ref{th:main3} implies pointwise almost sure convergence  with respect to a Haar measure $\lambda$ on $L$.
Namely, there exist $a\in \mathbb{Q}_{>0}$, and $b\in \mathbb{Z}_{\ge 0}$
such that for every non-negative continuous subanalytic
function $\phi$ on $L$ with compact support and for almost every $x\in L$,
we have the asymptotic expansion 
\begin{equation}\label{eq:dense}
\frac{1}{e^{at}t^b}\sum_{\gamma\in\Gamma_t}
  \phi(x\pi(\gamma))=\int_L\phi\, d\lambda+\sum_{i=1}^b c_i(\phi,x)t^{-i}+O_{\phi,x}(e^{-\delta t})
\end{equation}
with $\delta>0$.

To deduce formula (\ref{eq:dense}) from Theorem \ref{th:main5}, all we need to do is to identify the limit measure.  We choose the section
$\mathsf{s}(\ell)=(e,\ell)$. Then the measure $\xi$ in (\ref{eq:measure}) is equal to $\lambda$.
Since for $h\in H$ and $l_1,l_2$ in a compact subset of $H$,
$$
\|(e,\ell_1^{-1})\cdot (h,e)\cdot (e,\ell_2)\|=\|(h,\ell_1^{-1}\ell_2)\|=\|h\|+O(1),
$$
it follows that
$$
\rho(H_t[\mathsf{s}(\ell_1),\mathsf{s}(\ell_2)])\sim \rho (H_t)\quad\hbox{as $t\to\infty$.}
$$
Hence, by (\ref{eq:v_x_x}), the limit measure is a Haar measure on $L$.

{\it Quantitative equidistribution.} Let us note that under our assumption here Theorem \ref{th:main5} holds as well, so that in particular, the quantitative mean ergodic theorem is valid.  In \cite{gn3} we apply this fact to $\tilde{\lambda}_t$  and derive that quantitative equidistribution holds in this case. Namely,  for H\"older functions convergence holds for
{\it every} $\ell\in L$ with a fixed rate, and with the implied constant uniform over $\ell$ in compact sets. When the volume growth is purely exponential, the rate of equidistribution  is $e^{-\delta t}$, and otherwise the rate is $t^{-\eta}$. 
Previously, the problem of distribution of dense projections was investigated in
\cite[Sec.~1.5.2]{GW}, but the method of \cite{GW} does not yield any error term.

\subsection{Values of quadratic form}\label{sec:quad2}

Let $Q$ be a nondegenerate indefinite quadratic form in $d$ variables with $d\ge 3$,
signature $(p,q)$.
Given a tuple of vectors $v=(v_1,\ldots,v_d)$ in $\mathbb{R}^d$, we denote by
$\bar Q(v)$ the corresponding Gram matrix:
$$
\bar Q(v):=(Q(v_i,v_j))_{i,j=1,\ldots d}\in \hbox{Mat}_d(\mathbb{R}).
$$
We denote by $\mathcal{F}_d$ the set of unimodular frames, namely the set of $d$-tuples of  
vectors $v=(v_1,\ldots,v_d)$ in $\mathbb{R}^d$ satisfying  $\det(v_1,\ldots,v_d)=1$.  Let 
$\mathcal{F}_d(\mathbb{Z})$ denote the subset of unimodular frames  with integral coordinates.

Note that for any unimodular frame, the representation of the quadratic form $Q$ as a matrix w.r.t. the frame has the same determinant, which we will denote by $\Delta$. 
Thus for $v\in \mathcal{F}_d$, we have $\bar Q(v)\in \mathcal{Q}_{p,q}(\Delta)$
where $\mathcal{Q}_{p,q}(\Delta)$ denotes the set of symmetric matrices
with signature $(p,q)$ and determinant $\Delta$.
We also use the same notation for the corresponding set of quadratic forms.
It is known that for almost all quadratic forms $Q$ in the space of nondegenerate
quadratic forms of given dimension, the set $\bar Q(\mathcal{F}_d(\mathbb{Z}))$
is dense in $\mathcal{Q}_{p,q}(\Delta)$.  The distribution of $\bar Q(\mathcal{F}_d(\mathbb{Z}))$
in $\mathcal{Q}_{p,q}(\Delta)$ was investigated in \cite[Sec.~1.5.1]{GW}, and here we show
that the asymptotic formula from \cite{GW} holds with an exponentially decaying error term for almost all quadratic forms
$Q\in \mathcal{Q}_{p,q}(\Delta)$.

We observe that the group $G:=\hbox{SL}_d(\mathbb{R})$
acts transitively on $\mathcal{Q}_{p,q}(\Delta)$ by 
$$
x\mapsto {}^t g x g,\quad x\in \mathcal{Q}_{p,q}(\Delta), \; g\in G,
$$
and the space $\mathcal{Q}_{p,q}(\Delta)$ can be identified with $H\backslash G$ where $H\simeq \hbox{SO}_{p,q}(\mathbb{R})$.
Moreover, with respect to this action,
$$
\bar Q(\mathcal{F}_d(\mathbb{Z}))=x_Q\cdot \Gamma,
$$
where $x_Q$ denotes the matrix corresponding to $Q$, and $\Gamma=\hbox{SL}_d(\mathbb{Z})$.

Using Theorem \ref{th_dual_pointwise} we deduce that for $\phi\in L^s(D)^+$, $s>1$, and some $\delta > 0$ (independent of $\phi$ and $v$) and for almost all  $Q$
$$
\sum_{v\in \mathcal{F}_d(\mathbb{Z}):\, \sum_i\|v_i\|^2< e^t}\phi(Q(v))=\int_{v\in \mathcal{F}_d:\, \sum_i\|v_i\|^2<
  e^t}\phi(Q(v)) \,dm(v)+O_{s,\phi,Q}\left(e^{(p(q-1)-\delta) t}\right).
$$
Here $m$ denote the $G$-invariant measure on $\mathcal{F}_d\simeq G$, normalized so that $m(G/\Gamma)=1$. 
 The main term in the volume growth of the sets $H_t$ has been shown in \cite[proof of [Cor. 1.3, pp. 104-106]{GW}  to be given by $Be^{p(q-1)t}$ when $p < q$, 
where $B$ is a suitable normalizing constant depending only of the group. Thus the error estimate is $\hbox{vol}(H_t)e^{-\delta t}$. 

When $p=q$, the volume growth of $H_t$ has main term $Bte^{tp(p-1)}$ \cite{GW}, but nevertheless the error estimate in the preceding statement is again $\vol(H_t)e^{-\delta t}=te^{(p(p-1)-\delta)t}$, 
as follows from Theorem \ref{th_dual_pointwise}. 

Note however that we have chosen a norm to define the sets $H_t$, the group $H$ is simple (provided $(p,q)\neq (2,2) $), and $H$ has a spectral gap in $L^2(G/\Gamma)$. Therefore  Theorem \ref{th:main5} applies in the present case. As already noted,  the volume growth of $H_t$ is purely exponential if and only if $p\neq q$, and in the latter  case the normalized sampling operators $\lambda_t$ converge exponentially fast in $L^2$-norm and pointwise almost everywhere to the integral of $\phi$ w.r.t. the limiting density,  for every $\phi\in L^s(\mathcal{Q}_{p,q}(\Delta))$, $s>1$, with compact support and almost every $Q\in \mathcal{Q}_{p,q}(\Delta)$,

\subsection{Affine actions of solvable groups}\label{sec:affine1}
We now turn to discuss the affine action (\ref{eq:affine}) mentioned in the introduction.
It is clear that ergodicity of this action is equivalent to ergodicity of the
action of the matrix $a$ on the torus $\mathbb{R}^d/\Delta$. Hence, this action is ergodic if and only if
the matrix $a$ has no roots of unity as eigenvalues.

To see that Theorem \ref{th:main1} applies to this case, let us first consider the case when
all the eigenvalues of the matrix $a$ are positive. Then $a$ can embedded in a one-parameter
algebraic subgroup $H$ of $\hbox{SL}_d(\mathbb{R})$. We consider the exponential solvable group $G:=H\ltimes \mathbb{R}^d$
which is naturally an algebraic subgroup $\hbox{SL}_{d+1}(\mathbb{R})$ and contains $\Gamma$ as a lattice.
Then $\mathbb{R}^d\simeq H\backslash G$, and the sets $\Gamma_t$ defined in (\ref{eq:ggamma_t})
are given by $\Gamma_t=\{\gamma\in \Gamma:\, \log \|\gamma\|'\le t\}$
with respect to a suitable chosen norm $\|\cdot\|'$ on $\hbox{Mat}_{d+1}(\mathbb{R})$.
Hence, we are in the framework of Theorem \ref{th:main1}, which holds 
for the normalized sampling operators  supported on the sets defined by general norms (see Remark \ref{r:gen_norm}).

We now compute the normalization factor $V(t)$ and the limit measures $\nu_v$, $v\in \mathbb{R}^d$,
following the general formulas from Section \ref{sec:proof}.
Let $\rho$ be the Haar measure on $H$ for which $\rho(H/\left<a\right>)=1$. Then
$$
\rho(H_t)\sim t\quad\hbox{as $t\to\infty$.}
$$
Hence, the correct normalization factor is $V(t)=t$. 
The measure
$$
dm(h,x)=d\rho(h)\frac{dx}{\vol(\mathbb{R}^d/\Delta)},\quad (h,v)\in H\ltimes \mathbb{R}^d,
$$
is the Haar measure on $G$ such that $m(G/\Gamma)=1$. For the section
$\mathsf{s}(x)=(e,x)$ of the factor map $G\to \mathbb{R}^d\simeq H\backslash G$, the
corresponding measure $\xi$, defined by (\ref{eq:measure}), is
$d\xi(x)=\frac{dx}{\vol(\mathbb{R}^d/\Delta)}$,
and according to (\ref{eq:vv_x}) it is equal to the limit measure 
appearing in (\ref{eq:affine1}).

Finally, in the case when the matrix $a$ has negative real eigenvalues 
one can apply the previous argument to the index 2 index subgroup $\left<a^2\right>\ltimes \Delta$
of $\Gamma$ to verify the claim.

We remark that another interesting collection of examples for which Theorem \ref{th:main1} applies arises in the case of dense subgroups of nilpotent groups. For a different approach 
to equidistribution results for dense nilpotent groups we refer to \cite{Br}.

\subsection{Affine actions of lattices}\label{sec:affine2}
Consider the affine action of the group $\Gamma=\hbox{SL}_d(\mathbb{Z})\ltimes \mathbb{Z}^d$ on 
the Euclidean space $\mathbb{R}^d$.  It is natural to consider $\Gamma$ as a
subgroup of $\hbox{SL}_{d+1}(\mathbb{R})$. We fix a Euclidean norm $\hbox{Mat}_{d+1}(\mathbb{R})$
and define the sets $\Gamma_t$ with respect to this norm. As we shall verify, Theorem \ref{th:main5}
applies to the normalized sampling operators $\sum_{\gamma\in \Gamma_t} \phi(x\gamma)$
on $\mathbb{R}^d$.  Therefore, we deduce that for every $\phi\in L^p(\RR^d)$ of compact support, for almost every $\in \RR^d$, and for a fixed $\delta_p > 0$ independent of $\phi$ and $x$, 
\begin{equation}\label{eq:affine_2}
\frac{1}{e^{(d^2-d)t}} \sum_{\gamma\in\Gamma_t}
  \phi(v\gamma)= \frac{c_d}{(1+\|v\|^2)^{d/2}}\int_{\mathbb{R}^d} \phi(x)\, dx+O_{p,\phi,x}\left(e^{-\delta_p t}\right)\,,
\end{equation}
where $c_d=\pi^{d^2/2}\Gamma(d/2)^{-1}\Gamma((d^2-d+2)/2)^{-1}\zeta(2)^{-1}\cdots \zeta(d)^{-1}$.

To verify (\ref{eq:affine_2}), we consider the group $\Gamma$ as a lattice subgroup in the group 
$G:=\hbox{SL}_d(\mathbb{R})\ltimes \mathbb{R}^d$,
which is naturally an algebraic subgroup of $\hbox{SL}_{d+1}(\mathbb{R})$
under the embedding 
$$
(h,x)\mapsto \left(\begin{tabular}{ll} $h$ & 0 \\ $x$ & 1\end{tabular}  \right),\quad (h,x)\in G.
$$
Then $\mathbb{R}^d$ is a homogeneous
space of $G$ with respect to the action by the affine transformations, and
$\mathbb{R}^d\simeq H\backslash G$ where $H=\hbox{SL}(d,\mathbb{R})$.
The action $\hbox{SL}_d(\mathbb{Z})$ on the torus $\mathbb{R}^d/\mathbb{Z}^d$
has spectral gap.
Since the action of $H$ on $G/\Gamma$ is isomorphic to the action induced from this action, it follows that 
it has spectral gap as well.  Thus $H$ is simple and acts with a spectral gap on $L^2(G/\Gamma)$, we have defined the sets $H_t$ using a norm, and we will see below that the rate of growth of $H_t$ is purely exponential. The assumption of  Theorem \ref{th:main5} are therefore satisfied. 

It remains to compute the formulas for the normalization factor $V(t)$ and the limit measure 
$\nu_v$ (following the recipe of Section \ref{sec:proof}).
We fix a Haar measures $\rho$ on $H$ such that $\rho(H/\hbox{SL}_d(\mathbb{Z}))=1$. Then
$$
dm(h,x)=d\rho(h)dx,\quad (h,x)\in G,
$$
is the Haar measure on $G$ such that $m(G/\Gamma)=1$.
Hence, if we take the section $\mathsf{s}:\mathbb{R}^d\to G$ to be $\mathsf{s}(x)=(e,x)$,
the measure $\xi$, defined by (\ref{eq:measure}), is the Lebesgue measure on $\mathbb{R}^d$.
Since
\begin{equation}\label{eq:ass_ll}
\rho(H_t)\sim c_d\, e^{(d^2-d)t} \quad \hbox{as $t\to\infty$}.
\end{equation}
(see \cite[Appendix 1]{drs}), it follows that the normalization factor should be $V(t)=e^{(d^2-d)t}$.
By (\ref{eq:v_x_x}), the limit measure is given by 
$$
\nu_v(x)=\Theta((e,v),(e,x))dx,
$$
where
\begin{align*}
\Theta((e,v),(e,x))&=\lim_{t\to\infty}\frac{\rho(H_t[(e,v),(e,x)])}{e^{(d^2-d)t}}.
\end{align*}
We have
\begin{align*}
\rho(H_t[(e,v),(e,x)])&=\rho(\{h\in H:\, \log \|(e,v)^{-1}\cdot (h,0)\cdot (e,x)\|<t\})\\
&=\rho(\{h\in H:\, \|(h,-vh+x)\|<e^t\})\\
&=\rho(\{h\in H:\, (\|h\|^2+\|vh-x\|^2)^{1/2}<e^{t}\}).
\end{align*}
By the triangle inequality,
$$
\|(h,-vh)\|-\|x\|\le \|(h,-vh+x)\|\le \|(h,-vh)\|+\|x\|.
$$
This implies that the above limit is independent of $x$.
Moreover, since the norm is invariant under $k\in\hbox{SO}_d(\mathbb{R})$, we obtain that
\begin{align*}
\rho(\{h\in H:\, (\|h\|^2+\|vkh\|^2)^{1/2}<e^{t}\})
&=\rho(\{h\in H:\, (\|k^{-1}h\|^2+\|vh\|^2)^{1/2}<e^{t}\})\\
&=\rho(\{h\in H:\, (\|h\|^2+\|vh\|^2)^{1/2}<e^{t}\}).
\end{align*}
Therefore,
\begin{align*}
&\rho(\{h\in H:\, (\|h\|^2+\|vh\|^2)^{1/2}<e^{t}\})\\
=&\rho\left(\left\{h\in H:\, \left(\sum_{i=1}^{d-1} \|{e}_i h\|^2+(1+\|v\|^2) \|e_dh\|^2\right)^{1/2}<e^{t}\right\}\right)
\end{align*}
where $\{e_i\}_{i=1}^d$ is the standard basis of $\mathbb{R}^d$.
Let
$$
h_v=\hbox{diag}\left(1,\ldots,1,(1+\|v\|^2)^{1/2}\right)=(1+\|v\|^2)^{1/(2d)}h_v'\in\hbox{GL}_d(\mathbb{R}).
$$
Then since $h_v'\in H=\hbox{SL}_d(\mathbb{R})$, we get
\begin{align*}
\rho(\{h\in H:\, (\|h\|^2+\|vh\|^2)^{1/2}<e^{t}\}) &=\rho(\{h\in H:\, \|h_vh\|<e^{t}\})\\
&=\rho(\{h\in H:\, \|h\|<(1+\|v\|^2)^{-1/(2d)} e^{t}\})\\
&\sim c_d \,(1+\|v\|^2)^{-(d-1)/(2)} \, e^{(d^2-d)t}
\end{align*}
as $t\to\infty$, by (\ref{eq:ass_ll}). This explains the formula for the limit measure in
(\ref{eq:affine_2}).

We also note that using the method of \cite{GW}, which is based on the Ratner's theory of unipotent flows,
one can prove that 
$$
\lim_{t\to\infty}\frac{1}{e^{(d^2-d)t}} \sum_{\gamma\in\Gamma_t}
  \phi(v\gamma)= \frac{c_d}{(1+\|v\|^2)^{d/2}}\int_{\mathbb{R}^d} \phi(x)\, dx
$$
for every irrational $v\in \mathbb{R}^d$.

\end{document}